\documentclass[11pt]{article}%

\usepackage{fullpage}
\usepackage{amsfonts}
\usepackage{amsmath}
\usepackage{amssymb}
\usepackage{amsbsy}
\usepackage{amsthm}
\usepackage{epsfig}
\usepackage{graphicx}
\usepackage{colordvi}
\usepackage{graphics}
\usepackage{color}
\usepackage{float}
\usepackage{bm}
\usepackage{verbatim} 
\usepackage{epstopdf}

\usepackage{epsfig}
\usepackage{graphicx}
\usepackage{colordvi}
\usepackage{graphics}
\usepackage{color}
\usepackage{float}
\usepackage{bm}
\usepackage{verbatim} 
\usepackage{epstopdf}

\numberwithin{equation}{section}

\newtheorem{theorem}{Theorem}[section]
\newtheorem{thm}{Theorem}[section]

\newtheorem{lemma}{Lemma}[section]

\newtheorem{alg}[thm]{Algorithm}
\newtheorem{remark}[thm]{Remark}

\def \bX{{\mathbf X}}
\def \bV{{\mathbf V}}
\def \bH{{\mathbf H}}

\def \bu{{\mathbf u}}
\def \bv{{\mathbf v}}
\def \bV{{\mathbf V}}
\def \bw{{\mathbf w}}
\def \be{{\mathbf e}}
\def \bE{{\mathbf E}}

\def \bF{{\mathbf F}}
\def \b0{{\mathbf 0}}

\def \bL{{\mathbf L}}

\def \bta{{\mathbf \eta}}

\usepackage{caption}
\usepackage{adjustbox}
\author{
Mine Akbas \thanks{Department of Mathematics, Duzce University, 81620, D{\"{u}}zce (mineakbas@duzce.edu.tr)}
\and
Aytekin \c{C}{\i}b{\i}k \thanks{Department of Mathematics, Gazi University, 06550, Ankara (abayram@gazi.edu.tr)}
}
\title{Continuous Data Assimilation for the Double-Diffusive Natural Convection}
\date{\vspace{-5ex}}

\begin{document}

\maketitle
\begin{abstract}
In this study, we analyzed a continuous data assimilation scheme applied on a double-diffusive natural convection model. The algorithm is introduced with a first order backward Euler time scheme along with a finite element discretization in space. The long time stability and convergence results are presented for different options of nudging parameters. Two elaborative numerical test are given in order to confirm the theory and prove the promise of the algorithm.
\end{abstract}

\section{Introduction}
The need of modelling fluid flows arises in many applications of the science and engineering including
weather forecasting, oceanography, and the design of the aircraft.
The success of these models is highly depend on the precision of the initial data. Unfortunately, such data concerning many real-world observation is not entirely be known, or contain error, which is due to the quality of the instrument and how accurately the position of the measurement is known. Therefore, simulations without precise knowledge of this data can cause instabilities, especially over long time intervals, and leads to results which do not match real-life situations. \\
\indent One sophisticated method to handle with this issue is the data assimilation (DA) which enables us to combine an observable data with a numerical method to make the computed solutions better, and closely resemble the current state of the system.
\\
\indent Since Kalman's seminal paper published in 1960, this powerful method has gained increasing popularity in researchers, and various data assimilation algorithms depending on different goals have arisen \cite{CHJ69, Daley91, LL08, BLSZ13, LSZ15}.
Recently, Azouani, Olson and Titi have proposed a promising DA technique, which is known as continuous data assimilation (CDA) or AOT data assimilation, \cite{AOT14, AT14} (see \cite{CKT01, OT03, HOT11} for early ideas in this direction). This technique adds a nudging term which uses the coarse mesh observables into partial differential equations in order to drive the approximate solutions towards the reference solution corresponding to the observed data. The main advantage of the method is to allow us to use classical interpolant operators, which are linear and satisfy approximation properties, since spatial derivatives are not required for coarse mesh observables.\\
We make a note here that Bl{\"{o}}mker and and co-workers applied a similar approach to stochastic differential equations in \cite{BLSZ13}.\\
\indent The application of this technique can be seen in considerably amount of recent research. For such work, we refer to \cite{JST15, ALT16, BM17, FJT15, FLT16, FLT216, FLT316, FLT17, JMT17}. In particular, the nudging technique was applied to the case of noisy data in \cite{BOT15}, and 
to the case in which measurements are obtained discretely in time and may be contaminated by systematic errors in \cite{FMT16}.
In addition, the nudging technique and its variant was applied for  Navier-Stokes equations in \cite{GOT16}, for the Benard convection equations in \cite{AHKMZT17}, and the Kuramoto-Sivashinsky equations in \cite{LT17, LP}. We note that a continuous-in-time Galerkin approximation of the algorithm was studied in \cite{MT18}. Also, recently the method was studied finite element method in space for the 2D Navier-Stokes equations (NSE) in \cite{leocda, LeoCamilleLariosa, LeoCamille}.
\\
\indent The motivation of this report is to apply this recent promising DA technique on a double-diffusive natural convection. Double-diffusive convection is a mechanism, in which the fluid motion occurs due to buoyancy arising from the combination of temperature and concentration gradients. It is related with an increasing number of fields such as, metallurgy, oceanography, contaminant transport, petroleum drilling etc.,\cite{BS89, XQT98, NT15, SGRE14}. The accurate and efficient numerical solutions of these flows are known to be the core of many applications. 
Under the assumption of Boussinesq approximation, the Darcy-Brinkman equations modelling the double-diffusive natural convection pheonomena are given by 
\begin{eqnarray}\label{bous}
\begin {array}{rcll}
\textbf{u}_t -\nu \Delta \textbf{u}+ (\textbf{u}\cdot\nabla)\textbf{u} +Da^{-1}\textbf{u}+ \nabla p &=& (\beta_T T + \beta_C S)\textbf{g} + \bF&
\mathrm{in }\ \Omega, \\
\nabla \cdot \textbf{u}&=& 0& \mathrm{in }\  \Omega,\\
T_t-\kappa \Delta T+\textbf{u}\cdot \nabla T&=&G& \mathrm{in }\ \Omega,\\
S_t-{D_c}\Delta S +\textbf{u}\cdot \nabla S&=&\Phi& \mathrm{in }\ \Omega,
\end{array}
\end{eqnarray}
with appropriate boundary and initial conditions. Here $\textbf{u}$ denotes the velocity, $p$ is the pressure, $T$ is the temperature, $S$ is the concentration. The kinematic viscosity is shown with $\nu >0$, the Darcy number $Da$ and the thermal diffusivity $\kappa > 0$. We have the mass diffusivity ${D_c} > 0$, the vector in the direction of gravitational acceleration is $\textbf{g}$ and the thermal and solutal expansion coefficients are $\beta_T$, $\beta_C$, respectively. \\
\indent Applying CDA method to \eqref{bous} as in the suggested in \cite{AOT14, AT14, LRZ19}, the model takes the form
\begin{eqnarray}\label{nudge}
\begin {array}{rcll}
\widetilde{\bu}_t -\nu \Delta \widetilde{\bu} + (\widetilde{\bu}\cdot\nabla)\widetilde{\bu} + Da^{-1}\widetilde{\bu} + \nabla p + \mu_1(I_H(\widetilde{\bu} -\bu)) &=& (\beta_T \widetilde{T} + \beta_C \widetilde{S})\textbf{g} + \bF&
\mathrm{in }\ \Omega, \\
\nabla \cdot \widetilde{\bu}&=& 0& \mathrm{in }\  \Omega,\\
\widetilde{T}_t-\kappa \Delta \widetilde{T} + \widetilde{\bu}\cdot \nabla \widetilde{T} + \mu_2(I_H(\widetilde{T} - T))& = &G& \mathrm{in }\ \Omega,\\
\widetilde{S}_t-{D_c}\Delta \widetilde{S} + \widetilde{\bu}\cdot \nabla \widetilde{S} +\mu_3(I_H(\widetilde{S} -S))&=&\Phi& \mathrm{in }\ \Omega.
\end{array}
\end{eqnarray}
Here, the positive scalars $\mu_1, \mu_2, \mu_3$ are knowns as nudging parameters, $I_H$ denotes an appropriate interpolation operator which is linear and satisfies the approximation properties, see \eqref{interp1}-\eqref{interp2} in following section.  $I_H(\widetilde{\bu})$, $I_H(\widetilde{T})$ and $I_H(\widetilde{S})$ are known data obtained from measurements observed at a coarse spatial mesh with mesh size $H$. We assume that these data are continuous in time and error-free. \\
\indent As discussed in \cite{OT08}, numerical instability can be arisen in the nudging terms on the left hand side of \eqref{nudge} for large values of nudging parameters when these terms are treated explicitly. In addition, the use of higher order discretization such as Runge-Kutta-type  methods  or  (fully)  implicit  methods  of  order  greater does not guarantee the same dynamics of the unknowns of \eqref{bous} and \eqref{nudge}, see \cite{OT08}. Because of this reason, the system \eqref{nudge} are discretized by using implicit treatment of these terms, and semi-implicit of the non-linear terms. 
\section{Mathematical Preliminaries} 
In this section, we provide some mathematical preliminaries used  throughout the paper. We study on a domain $\Omega \subset \mathbb{R}^d$, $d=2, 3$ which is a convex polygon or polyhedra. We denote $L^2$ inner product and its induced norm of the scalar valued functions by $(\cdot, \cdot)$ and $\|\cdot\|$, respectively, and $H^{k}(\Omega)$ norm by $\|\cdot\|_{k}$, the $L^{\infty}(\Omega)$ norm by $\|\cdot\|_{\infty}$. We use the same notations for all analogous norms of the vector valued functions. \\

\noindent The natural function spaces for velocity, pressure, temperature and concentration are denoted by
\begin{align*}
\bX & := H_0^1(\Omega)^d= \left \{\bv\in H^1(\Omega)^d: \, \, \bv =0 \hspace{2mm}\text{on} \hspace{2mm} \partial \Omega\right\}, \\
Q&:= L_0^2(\Omega)=\left \{q\in L^2(\Omega)^d: \int\limits_{\Omega}q \mathrm{d}x = 0\right\}, \\
Y & := H_0^1(\Omega)= \left \{\Psi \in H^1(\Omega): \, \,\Psi =0 \hspace{2mm}\text{on} \hspace{2mm} \partial \Omega\right\}, \\
W & := H_0^1(\Omega)= \left \{\chi \in H^1(\Omega): \, \, \chi =0 \hspace{2mm}\text{on} \hspace{2mm} \partial \Omega\right\}. 
\end{align*}
Skew symmetrized trilinear forms for non-linear terms to ensure stability of the numerical method are defined by
\begin{align*}
b_1(\bu, \bv, \bw)& :=\frac{1}{2}\left( (\bu\cdot \nabla\bv, \bw)- (\bu\cdot \nabla\bw, \bv)\right),\hspace{2mm}\forall \, \bu, \bv, \bw \in \bX,\\
b_2(\bu, T, \Psi)& :=\frac{1}{2}\left( (\bu \cdot \nabla T, \Psi)- (\bu\cdot \nabla\Psi, T)\right),\hspace{2mm}\forall \, \bu \in \bX,\,\, T, \Psi \in Y,\\
b_3(\bu, S, \chi)& :=\frac{1}{2}\left( (\bu \cdot \nabla S, \chi)- (\bu\cdot \nabla\chi, S)\right),\hspace{2mm}\forall \, \bu \in \bX,\,\, S, \chi \in W.
\end{align*}
We need some important estimates for $(\bu\cdot \nabla\bv, \bw)1$ that we will employ in subsequent sections, \cite{Lay08}. Analogous estimates also hold for the skew-symmetric operators.
\begin{lemma} \label{trilinearbound}
For $\bu, \bv, \bw \in \bX$, and also $\bv, \nabla\bv\in \bL^{\infty}(\Omega)$ for \eqref{tribound1}, the term $(\bu\cdot\nabla \bv, \bw)$ is bounded by
\begin{eqnarray}
(\bu\cdot \nabla\bv, \bw) & \leq & \|\bu\|\|\nabla \bv\|_{\infty}\|\bw\| , \label{tribound1}\\
(\bu\cdot \nabla\bv, \bw) & \leq & C \|\nabla \bu\|\|\nabla \bv\|\|\nabla \bw\|. \label{tribound2}
\end{eqnarray}
\end{lemma}
\begin{proof}
The first of these bounds can be proved by applying the generalized H\"{o}lder Inequality with $p=2,\, q=\infty, \, r=2$. The second bound follows from the generalized H\"{o}lder Inequality with $p=2,\, q=4, \, r=2$, the Ladyzhenskaya Inequality together with the Poincar{\'{e}}-Friedrichs' Inequality, see \cite{Lay08}. 
\end{proof}
\noindent We frequently call the Poincar{\'{e}}-Friedrich Inequality; there exists a constant $C_{PF}:= C_{PF}(\Omega)$ such that for all $\varphi \in W $, (and $\bm \varphi \in \bX$) 
$$ \|\varphi\|\leq C_{PF}\|\nabla \varphi \|, $$
and Young's Inequality; for any $\varepsilon>0$
\begin{align*}
a\,b\leq \frac{\varepsilon}{p}a^{p} +\frac{\varepsilon^{-q/p}}{q}b^{q},\,\, a,b\geq 0
\end{align*}
where $\frac{1}{p}+ \frac{1}{q}=1$ with $p,q\in [1,\infty).$\\

\noindent We assume a regular, conforming mesh $\tau_h$, with maximum element diameter $h$, and associated velocity-pressure-temperature finite element (FE) spaces $\bX_h\subset \bX$, $Q_h\subset Q$, $Y_h\subset Y$ and $W_h\subset W$ satisfying approximation properties of piecewise polynomials of local degree $k, k-1$, $k$ and $k$ respectively, \cite{GR86}:
\begin{eqnarray*}
    \inf_{\bv_h\in \bX_{h}}\left(\| \bu - \bv_h \| + h \|\nabla( \bu - \bv_h )\| \right) & \le & C h^{k+1} \| \bu \|_{k+1},\;\; \bu \in
    \bH^{k+1}(\Omega),\\
    \inf_{q_h \in Q_{h}} \| p - q_h \| & \le & C h^{k} \| p \|_{k},\;\;  \quad \quad p \in
    H^{k}(\Omega),\\
    \inf_{\omega_h\in Y_{h}}\left(\| T - \Psi_h \| +  h \|\nabla( T - \Psi_h )\| \right)& \le & C h^{k+1} \| T \|_{k+1},\;\; T \in
    H^{k+1}(\Omega),\\
    \inf_{\chi_h\in W_{h}}\left(\| S - \chi_h \| +  h \|\nabla( S - \chi_h )\| \right)& \le & C h^{k+1} \| S \|_{k+1},\;\; S \in
    H^{k+1}(\Omega).
\end{eqnarray*}
The finite element spaces for velocity-pressure are assumed to satisfy the discrete inf-sup condition
for the stability of pressure, i.e., there is a constant $\beta$ independent of the mesh size h such that
$$\inf\limits_{q_h\in Q_h}\sup\limits_{\bv_h\in \bX_h}\frac{(q_h, \nabla\cdot\bv_h)}{\|\nabla\bv_h\|}\geq \beta>0.$$ 
The discretely divergence-free subspace of $\bX_h$ will be denoted by
\begin{eqnarray*}
\bV_h= \{ \bv_h \in \bX_h:\,\, (q_h, \nabla \cdot \bv_h)=0 \ \ \forall q_h \in Q_h \}\, .
\end{eqnarray*}
We also need a regular conforming mesh $\tau_H$, and function spaces denoted by $\bX_H$, $Y_H$ and $W_H$ on this mesh. These spaces are necessary for measurement data interpolation. Denoting the coarse mesh interpolation operator $I_H$, we assume that the following bounds are satisfied:
\begin{align}
\|I_{H}(\varphi)-\varphi\|& \leq C H \|\nabla \varphi\|,\label{interp1}\\
\|I_{H}(\varphi)\|& \leq C \|\varphi\|,\label{interp2}
\end{align}
Requiring these bound allows us mathematical theory for the finite element analysis. We note here that we use the same notation of this interpolation operator for the vector and scaler valued functions.\\

\noindent We also introduce the notation $t^{n+1}:= (n+1)\,\Delta t$, where $\Delta t$ is a chosen time-step, and the following discrete time norms:
\begin{align*}
\||v|\|_{\infty, k}:=\max\limits_{0\leq n \leq N}\|v(t^n,\cdot)\|_{k}, \hspace{3mm} \text{and} \hspace{3mm}\||v|\|_{m, k}:=\bigg(\Delta t\, \sum\limits_{n=0}^{N-1} \|v(t, \cdot)\|_{k}^m  \bigg)^{1/m}.
\end{align*}

\section{Stability and Convergence Analysis}
We devote this section to the stability and convergence analysis of Algorithm~\ref{algbe}. We first show that  solutions of the proposed algorithm are stable at all time levels without time step restriction. Then we prove the convergence of discrete solutions to true solutions of \eqref{bous}.
\begin{alg}\label{algbe}
Let initial conditions $\widetilde{\bu}_h^{0}, \widetilde{T}_h^{0},\widetilde{S}_h^{0}$, and forcing terms be given. Select a time step $\Delta t>0$. For each $n=0,1,2,...$, find $\left(\widetilde{\bu}_h^{n+1}, p_h^{n+1}, \widetilde{T}_h^{n+1}, \widetilde{S}_h^{n+1}\right)\in (\bX_h, Q_h, Y_h, Y_h) \,$ such that it holds: $\forall \, \left(\bv_h, q_h, \Psi_h, \chi_h\right)\in (\bX_h, Q_h, Y_h, Y_h)$
\begin{gather}
\frac{1}{\Delta t}\big(\widetilde{\bu}_h^{n+1}- \widetilde{\bu}_h^{n}, \bv_h\big) + b_1 \,(\widetilde{\bu}_h^{n},\, \widetilde{\bu}_h^{\,n+1}, \,\bv_h)  - (p_h^{n+1}, \nabla\cdot \bv_h) + \nu\big(\nabla\widetilde{\bu}_h^{n+1}, \nabla \bv_h\big)+ D_a^{-1}(\widetilde{\bu}_h^{n+1}, \bv_h) \nonumber\\
\qquad \qquad  \qquad  \quad \quad \quad \quad \mu_1(I_H(\widetilde{\bu}_h^{n+1}- \bu^{n+1}), I_H({\bv_h})) = ((\beta_T \widetilde{T}_h^n + \beta_c \widetilde{S}_h^n)\bm g, \bv_h) + (\bF^{n+1}, \bv_h),\label{disc1}
\\
\qquad \qquad  \qquad \qquad \qquad \qquad \qquad \qquad \qquad\qquad \qquad \qquad \qquad \qquad \qquad \qquad ( \nabla\cdot\widetilde{\bu}_h^{n+1},\, q_h)  = 0,\\
\frac{1}{\Delta t}\big(\widetilde{T}_h^{n+1} - \widetilde{T}_h^{n}, \Psi_h\big) +  b_2\,(\widetilde{\bu}_h^{n+1}, \widetilde{T}_h^{n+1}, \, \Psi_h)+ \kappa\big(\nabla\widetilde{T}_h^{n+1}, \nabla \Psi_h\big) + \mu_2(I_H(\widetilde{T}_h^{n+1}- T^{n+1}), I_H({\Psi_h}))\nonumber\\
\qquad \qquad \qquad \qquad \qquad \qquad \qquad \qquad \qquad \quad \quad \qquad \qquad \qquad \qquad \qquad \qquad \qquad = (G^{n+1}, \Psi_h),\label{disc2}
\\
\frac{1}{\Delta t}\big(\widetilde{S}_h^{n+1} - \widetilde{S}_h^{n}, \chi_h\big) +  b_3\,(\widetilde{\bu}_h^{n+1}, \widetilde{S}_h^{n+1}, \, \chi_h) + D_c\big(\nabla\widetilde{S}_h^{n+1}, \nabla \chi_h\big) +\mu_3(I_H(\widetilde{S}_h^{n+1}- S^{n+1}), I_H({\chi_h}))\nonumber\\
\qquad \qquad \qquad \qquad \qquad \qquad \qquad \qquad \qquad \quad \qquad \qquad \quad \qquad \qquad \quad \qquad \qquad \quad = (\Phi^{n+1}, \chi_h)\label{disc3}.
\end{gather}
\end{alg}
\begin{remark}
We emphasize here that the proposed algorithm herein is consistent with the requirement of \cite{OT08} since the right hand side of the algorithm is evaluated only per time step. This is due to explicit-implicit treatment of the non-linear terms, i.e., no requirement of the multi step methods in the discretization of these terms such as Newton method. In addition we notice that $I_H(\bu^{n+1})$, $I_H(T^{n+1})$ and $I_H(S^{n+1})$ are taken the most recent measurable data, not future data which is unmeasured.
\end{remark}
\subsection{Long Time $L^2$-Stability}
\begin{lemma}
Assume that $\bu\in L^{\infty}(0,\infty; {\bm L^2})$, $T, S\in L^{\infty}(0,\infty; {L^2})$, and $\bF \in L^{\infty}(0,\infty; {\bm H^{-1}})$,  $G, \Phi \in L^{\infty}(0,\infty; {H^{-1}})$. Let initial conditions $\widetilde{\bu}_h^{0}, \widetilde{T}_h^{0},\widetilde{S}_h^{0}$ be given. Then, solutions $\widetilde{\bu}_h^{n+1}, \widetilde{T}_h^{n+1},\widetilde{S}_h^{n+1}$ of Algorithm~\eqref{algbe} are stable at all time levels, and satisfy the bounds: for any $\Delta t>0$,
\begin{align*}
&\|\widetilde{\bu}_h^{n+1}\|^2 + \frac{\nu\Delta t}{4}\|\nabla \widetilde{\bu}_h^{n+1}\|^2  + \frac{D_a^{-1}\Delta t}{4}\|\widetilde{\bu}_h^{n+1}\|^2\nonumber\\
&\leq {(1+\alpha_u)^{-(n+1)}} \left(\|\widetilde{\bu}_h^{0}\|^2 +\frac{\nu\Delta t}{4}\|\nabla \widetilde{\bu}_h^{0}\|^2 +\frac{D_a^{-1}\Delta t}{4}\|\widetilde{\bu}_h^{0}\|^2  \right)\nonumber\\
 &\,\,\,\,\, + \max\left\{\frac{4C_{PF}^2 D_a}{\nu + C_{PF}^2 D_a^{-1}}, \,  D_a \Delta t \right \} \left(\beta_T^2\max\left\{ \|\widetilde{T}_h^{0}\|^2, K_T\right\} + \beta_c^2\max\left\{ \|\widetilde{S}_h^{0}\|^2, K_S\right\}  \right)\|g\|_{L^{\infty}}^2 \nonumber\\
&\,\,\,\,\, + \max\left\{\frac{4C_{PF}^2 \nu^{-1}}{\nu + C_{PF}^2 D_a^{-1}},  \,\nu^{-1}\Delta t \right \} \|\bF\|_{L^{\infty}(0,\infty; {\bm H^{-1}})}^2 + \max\left\{\frac{4C_{PF}^2 \mu_1}{\nu + C_{PF}^2 D_a^{-1}}, \, \mu_1\Delta t \right \} \|\bu\|^2_{L^{\infty}(0,\infty;{\bm L^2})},
\end{align*}
and
\begin{multline*}
\|\widetilde{T}_h^{n+1}\|^2  + \frac{\kappa \Delta t}{4}\|\nabla \widetilde{T}_h^{n+1}\|^2\\
\leq (1+ \lambda_T)^{-(n+1)}\left( \|\widetilde{T}_h^{0}\|^2 +  \kappa \Delta t\|\nabla \widetilde{T}_h^{0}\|^2\right)
+\max \left \{{4 \mu_2 C_{PF}^2 \kappa^{-1}}, \,\mu_2 \Delta t \right \}\|T\|^2_{L^{\infty}(0,\infty; L^2)} \\
+ 
\max \left \{4 C_{PF}^2\kappa^{-2}, \, \kappa^{-1}\Delta t \right \}\|G\|_{L^{\infty}(0,\infty; H^{-1}(\Omega))}^2 =:K_T,
\end{multline*}
and
\begin{multline*}
\|\widetilde{S}_h^{n+1}\|^2  + \frac{D_c\Delta t}{4}\|\nabla \widetilde{S}_h^{n+1}\|^2\\
\leq (1+ \lambda_S)^{-(n+1)}\left( \|\widetilde{S}_h^{0}\|^2 
+ \frac{D_c\Delta t}{4}\|\nabla \widetilde{S}_h^{0}\|^2\right)
+\max \left \{{4 \mu_3 C_{PF}^2 D_c^{-1}}, \,\mu_3 \Delta t \right \}\|S\|^2_{L^{\infty}(0,\infty; L^2)} \\
+ 
\max \left \{4 C_{PF}^2 D_c^{-2}, \, D_c^{-1}\Delta t \right \}\|\Phi\|_{L^{\infty}(0,\infty; H^{-1}(\Omega))}^2=:K_S,
\end{multline*}
where $ \lambda_u:=\min \left \{\frac{\nu \Delta t}{4 C_{PF}^2} + \frac{D_a^{-1} \Delta t}{4}, 1 \right \}$, $ \lambda_T:=\min \left \{\frac{\kappa \Delta t}{4 C_{PF}^2}, 1 \right \}$ and  $ \lambda_S:=\min \left \{\frac{D_c \Delta t}{4 C_{PF}^2}, 1 \right \}$.
\end{lemma}

\begin{proof}
We first obtain bounds on temperature and concentration solution of Algorithm~\ref{algbe}. Setting $\Psi_h= \widetilde{T}_h^{n+1}$ in \eqref{disc2} and $\chi_h= \widetilde{S}_h^{n+1}$ in \eqref{disc3} vanishes non-linear terms. Using the identity $$2(a-b,a)= a^2 -b^2 +(a-b)^2$$ produces
\begin{multline*}
\frac{1}{2 \Delta t}\big(\|\widetilde{T}_h^{n+1}\|^2 - \|\widetilde{T}_h^{n}\|^2 + \|\widetilde{T}_h^{n+1}-\widetilde{T}_h^{n}\|^2\big) +\kappa\|\nabla \widetilde{T}_h^{n+1}\|^2 + \mu_2\|I_H(\widetilde{T}_h^{n+1})\|^2 \\
= \mu_2\left(I_H(T^{n+1}), I_H({\widetilde{T}_h^{n+1}})\right) + \left(G^{n+1}, \widetilde{T}_h^{n+1}\right),
\end{multline*}
and
\begin{multline*}
\frac{1}{2 \Delta t}\big(\|\widetilde{S}_h^{n+1}\|^2 - \|\widetilde{S}_h^{n}\|^2 + \|\widetilde{S}_h^{n+1}-\widetilde{S}_h^{n}\|^2\big) + D_c\|\nabla \widetilde{S}_h^{n+1}\|^2 + \mu_3\|I_H(\widetilde{S}_h^{n+1})\|^2 \\
= \mu_3\left(I_H(S^{n+1}), I_H({\widetilde{S}_h^{n+1}})\right) + \left(\Phi^{n+1}, \widetilde{S}_h^{n+1}\right).
\end{multline*}
Using Cauchy-Schwarz, the interpolation property \eqref{interp2} and Young's inequality on the right hand side yields
\begin{align*}
\frac{1}{2 \Delta t}\big(\|\widetilde{T}_h^{n+1}\|^2 - \|\widetilde{T}_h^{n}\|^2 & + \|\widetilde{T}_h^{n+1}-\widetilde{T}_h^{n}\|^2\big) +\kappa\|\nabla \widetilde{T}_h^{n+1}\|^2 + \mu_2\|I_H(\widetilde{T}_h^{n+1})\|^2\\
&\leq \mu_2\|I_H(T^{n+1})\| \| I_H({\widetilde{T}_h^{n+1}})\| + \|G^{n+1}\|_{-1} \|\nabla\widetilde{T}_h^{n+1}\|\\
&\leq \mu_2\|T^{n+1}\| \| I_H({\widetilde{T}_h^{n+1}})\| + \|G^{n+1}\|_{-1} \|\nabla\widetilde{T}_h^{n+1}\|\\
&\leq \frac{\mu_2}{2}\|T^{n+1}\|^2 + \frac{\mu_2}{2}\| I_H({\widetilde{T}_h^{n+1}})\|^2 + 
\frac{\kappa^{-1}}{2}\|G^{n+1}\|_{-1}^2 + \frac{\kappa}{2}\| \nabla{\widetilde{T}_h^{n+1}}\|^2,
\end{align*}
and
\begin{align*}
\frac{1}{2 \Delta t}\big(\|\widetilde{S}_h^{n+1}\|^2 - \|\widetilde{S}_h^{n}\|^2 & + \|\widetilde{S}_h^{n+1}-\widetilde{S}_h^{n}\|^2\big) + D_c\|\nabla \widetilde{S}_h^{n+1}\|^2 + \mu_3\|I_H(\widetilde{S}_h^{n+1})\|^2\\
&\leq \mu_3\|I_H(S^{n+1})\| \| I_H({\widetilde{S}_h^{n+1}})\| + \|\Phi^{n+1}\|_{-1} \|\nabla\widetilde{S}_h^{n+1}\|\\
&\leq \mu_3\|S^{n+1}\| \| I_H({\widetilde{S}_h^{n+1}})\| + \|\Phi^{n+1}\|_{-1} \|\nabla\widetilde{S}_h^{n+1}\|\\
&\leq \frac{\mu_3}{2}\|S^{n+1}\|^2 + \frac{\mu_3}{2}\| I_H({\widetilde{S}_h^{n+1}})\|^2 + 
\frac{D_c^{-1}}{2}\|\Phi^{n+1}\|_{-1}^2 + \frac{D_c}{2}\| \nabla{\widetilde{S}_h^{n+1}}\|^2.
\end{align*}
Rearranging terms, multiplying by $2\Delta t$ and dropping the non-negative third and fifth terms produces
\begin{align}
\|\widetilde{T}_h^{n+1}\|^2 - \|\widetilde{T}_h^{n}\|^2  + \kappa \Delta t\|\nabla \widetilde{T}_h^{n+1}\|^2 \leq \mu_2 \Delta t\|T^{n+1}\|^2 + 
\kappa^{-1} \Delta t\|G^{n+1}\|_{-1}^2\label{stabtemp1},
\end{align}
and
\begin{align}
\|\widetilde{S}_h^{n+1}\|^2 - \|\widetilde{S}_h^{n}\|^2  + {D_c}\Delta t\|\nabla \widetilde{S}_h^{n+1}\|^2 \leq \mu_3 \Delta t\|S^{n+1}\|^2
+ D_c^{-1}\Delta t\|\Phi^{n+1}\|_{-1}^2\label{stabcons1}.
\end{align}
Now add $\frac{\kappa\Delta t}{4} \|\nabla\widetilde{T}_h^{n}\|^2$ to both side of \eqref{stabtemp1} and $\frac{\kappa\Delta t}{4} \|\nabla\widetilde{S}_h^{n}\|^2$ to \eqref{stabcons1}. Rearranging terms produces:
\begin{align}
\|\widetilde{T}_h^{n+1}\|^2 + \frac{\kappa \Delta t}{4}\|\nabla \widetilde{T}_h^{n+1}\|^2 + \frac{\kappa \Delta t}{4}\left(\|\nabla \widetilde{T}_h^{n+1}\|^2 + \|\nabla \widetilde{T}_h^{n}\|^2\right) + \frac{\kappa \Delta t}{2}\|\nabla \widetilde{T}_h^{n+1}\|^2\nonumber \\
\leq \|\widetilde{T}_h^{n}\|^2 + \frac{\kappa \Delta t}{4}\|\nabla \widetilde{T}_h^{n}\|^2 + \mu_2 \Delta t\|T^{n+1}\|^2 + 
\kappa^{-1} \Delta t\|G^{n+1}\|_{-1}^2\label{stabtemp2},
\end{align}
and
\begin{align}
\|\widetilde{S}_h^{n+1}\|^2 + \frac{D_c \Delta t}{4}\|\nabla \widetilde{S}_h^{n+1}\|^2 + \frac{D_c \Delta t}{4}\left(\|\nabla \widetilde{S}_h^{n+1}\|^2 + \|\nabla \widetilde{S}_h^{n}\|^2\right) + \frac{D_c \Delta t}{2}\|\nabla \widetilde{S}_h^{n+1}\|^2 \nonumber\\
\leq \|\widetilde{S}_h^{n}\|^2 + \frac{D_c \Delta t}{4}\|\nabla \widetilde{S}_h^{n}\|^2 + \mu_3 \Delta t\|S^{n+1}\|^2 + 
D_c^{-1} \Delta t\|\Phi^{n+1}\|_{-1}^2\label{stabcons2}.
\end{align}
The terms on the left hand side of \eqref{stabtemp2} is estimated below by using the Poincar{\'{e}}-Friedrich inequality as follows:
\begin{align}
\frac{\kappa \Delta t}{4}(\|\nabla \widetilde{T}_h^{n+1}\|^2 & + \|\nabla \widetilde{T}_h^{n}\|^2) + \frac{\kappa \Delta t}{2}\|\nabla \widetilde{T}_h^{n+1}\|^2 \nonumber\\
&
\geq \frac{\kappa \Delta t}{4 C_{PF}^2}(\| \widetilde{T}_h^{n+1}\|^2 + \|\widetilde{T}_h^{n}\|^2) + \frac{\kappa \Delta t}{2}\|\nabla \widetilde{T}_h^{n+1}\|^2\nonumber\\
&
\geq \frac{\kappa \Delta t}{4 C_{PF}^2}\| \widetilde{T}_h^{n+1}\|^2 + \frac{\kappa \Delta t}{4}\|\nabla \widetilde{T}_h^{n+1}\|^2\nonumber\\
&
\geq \min \left \{\frac{\kappa \Delta t}{4 C_{PF}^2}, 1 \right \} \left(\| \widetilde{T}_h^{n+1}\|^2 + \frac{\kappa \Delta t}{4}\|\nabla \widetilde{T}_h^{n+1}\|^2\right)\label{stabtemp3}.
\end{align}
Plugging this estimate into \eqref{stabtemp2} yields
\begin{multline}
(1+ \lambda_T)\left(\|\widetilde{T}_h^{n+1}\|^2  + \frac{\kappa \Delta t}{4}\|\nabla \widetilde{T}_h^{n+1}\|^2\right) \\
\leq \|\widetilde{T}_h^{n}\|^2 + \frac{\kappa \Delta t}{4}\|\nabla \widetilde{T}_h^{n}\|^2+ \mu_2 \Delta t\|T\|^2_{L^{\infty}(0,\infty; L^2)} + 
\kappa^{-1} \Delta t\|G\|_{L^{\infty}(0,\infty; H^{-1}(\Omega))}^2 \label{stabtemp4},
\end{multline}
where $ \lambda_T:=\min \left \{\frac{\kappa \Delta t}{4 C_{PF}^2}, 1 \right \}$. Applying similar arguments for the terms in \eqref{stabcons2} leads to
\begin{align}
\frac{D_c \Delta t}{4}(\|\nabla \widetilde{S}_h^{n+1}\|^2 + \|\nabla \widetilde{S}_h^{n}\|^2) & + \frac{D_c \Delta t}{2}\|\nabla \widetilde{S}_h^{n+1}\|^2 \\
&\geq \min \left \{\frac{D_c \Delta t}{4 C_{PF}^2}, 1 \right \} \left(\| \widetilde{S}_h^{n+1}\|^2 + \frac{D_c \Delta t}{4}\|\nabla \widetilde{S}_h^{n+1}\|^2\right)\nonumber,
\end{align}
and inserting this in \eqref{stabcons2} produces
\begin{multline}
(1+ \lambda_S)\left(\|\widetilde{S}_h^{n+1}\|^2  + \frac{D_c \Delta t}{4}\|\nabla \widetilde{S}_h^{n+1}\|^2\right)\\
 \leq \|\widetilde{S}_h^{n}\|^2 + \frac{D_c \Delta t}{4}\|\nabla \widetilde{S}_h^{n}\|^2 + \mu_3 \Delta t\|S\|^2_{L^{\infty}(0, \infty; L^2)} + 
D_c^{-1} \Delta t\|\Phi\|_{L^{\infty}(0,\infty; H^{-1}(\Omega))}^2\label{stabcons3},
\end{multline}
where $\lambda_S:=\min \left \{\frac{D_c \Delta t}{4 C_{PF}^2}, 1 \right \}$. 
Using induction on \eqref{stabtemp4} and \eqref{stabcons3} produces stability estimates on the discrete temperature and concentration solutions of Algorithm~\ref{algbe}. \\

We now obtain the desired stability estimate on discrete velocity solution. Letting $\bv_h=\widetilde{\bu}_h^{n+1}$ in \eqref{disc1}, and using $$2(a-b,a)= a^2 -b^2 +(a-b)^2,$$
and applying Cauchy-Schwarz, \eqref{interp2} and Young's inequality to the right hand side yields
\begin{multline}
\frac{1}{2 \Delta t}\big(\|\widetilde{\bu}_h^{n+1}\|^2 - \|\widetilde{\bu}_h^{n}\|^2  + \|\widetilde{\bu}_h^{n+1}-\widetilde{\bu}_h^{n}\|^2\big) + \nu\|\nabla \widetilde{\bu}_h^{n+1}\|^2 + \mu_1\|I_H(\widetilde{\bu}_h^{n+1})\|^2 + D_a^{-1}\|\widetilde{\bu}_h^{n+1}\|^2 \\
\leq \frac{D_a}{2}\left(\beta_T^2 \|\widetilde{T}_h^{n}\|^2 + \beta_c^2\|\widetilde{S}_h^{n}\|^2  \right)\|g\|_{L^{\infty}}^2 + \frac{D_a^{-1}}{2}\|\widetilde{\bu}_h^{n+1}\|^2
+ 
\frac{\nu^{-1}}{2}\|\bF^{n+1}\|_{-1}^2 + \frac{\nu}{2}\| \nabla{\widetilde{\bu}_h^{n+1}}\|^2\\
 + \frac{\mu_1}{2}\|\bu^{n+1}\|^2 + \frac{\mu_1}{2}\| I_H({\widetilde{\bu}_h^{n+1}})\|^2
\label{stabvel1}.
\end{multline}
Rearranging terms, multiplying by $2\Delta t$ and dropping the non-negative third and fifth terms on the left hand side yields
\begin{align*}
\|\widetilde{\bu}_h^{n+1}\|^2 & + \nu\Delta t\|\nabla \widetilde{\bu}_h^{n+1}\|^2  + D_a^{-1}\Delta t\|\widetilde{\bu}_h^{n+1}\|^2 \\
&\leq \|\widetilde{\bu}_h^{n}\|^2 + {D_a}\Delta t\left(\beta_T^2 \|\widetilde{T}_h^{n}\|^2 + \beta_c^2\|\widetilde{S}_h^{n}\|^2  \right)\|g\|_{L^{\infty}}^2 + 
{\nu^{-1}}\Delta t\|\bF^{n+1}\|_{-1}^2
+ {\mu_1}\Delta t\|\bu^{n+1}\|^2 .
\end{align*}
Now add $\frac{\nu\Delta t}{4}\|\nabla \widetilde{\bu}_h^{n}\|^2 + \frac{D_a^{-1}\Delta t}{4}\|\widetilde{\bu}_h^{n}\|^2 $ to both sides and rewriting produces
\begin{align*}
&\|\widetilde{\bu}_h^{n+1}\|^2 + \frac{\nu\Delta t}{4}\|\nabla \widetilde{\bu}_h^{n+1}\|^2  + \frac{D_a^{-1}\Delta t}{4}\|\widetilde{\bu}_h^{n+1}\|^2 \\
& + \frac{\nu\Delta t}{4}\left( \|\nabla \widetilde{\bu}_h^{n+1}\|^2 + \|\nabla \widetilde{\bu}_h^{n}\|^2\right) 
+ \frac{\nu\Delta t}{2}\|\nabla \widetilde{\bu}_h^{n+1}\|^2
+ \frac{D_a^{-1}\Delta t}{4}\left(\|\widetilde{\bu}_h^{n+1}\|^2 + \|\widetilde{\bu}_h^{n}\|^2\right) 
+ \frac{D_a^{-1}\Delta t}{2}\|\widetilde{\bu}_h^{n+1}\|^2  \\
&\leq \|\widetilde{\bu}_h^{n}\|^2 +\frac{\nu\Delta t}{4}\|\nabla \widetilde{\bu}_h^{n}\|^2 + \frac{D_a^{-1}\Delta t}{4}\|\widetilde{\bu}_h^{n}\|^2  + {D_a}\Delta t\left(\beta_T^2 \|\widetilde{T}_h^{n}\|^2 + \beta_c^2\|\widetilde{S}_h^{n}\|^2  \right)\|g\|_{L^{\infty}}^2\\
&\,\,\,\,\, + 
{\nu^{-1}}\Delta t\|\bF^{n+1}\|_{-1}^2
+ {\mu_1}\Delta t\|\bu^{n+1}\|^2 .
\end{align*}
Repeating the arguments used in obtaining temperature (and concentration) stability estimate, we get 
\begin{align*}
&\frac{\nu\Delta t}{4}\left( \|\nabla \widetilde{\bu}_h^{n+1}\|^2 + \|\nabla \widetilde{\bu}_h^{n}\|^2\right) 
 + \frac{\nu\Delta t}{2}\|\nabla \widetilde{\bu}_h^{n+1}\|^2
+ \frac{D_a^{-1}\Delta t}{4}\left(\|\widetilde{\bu}_h^{n+1}\|^2 + \|\widetilde{\bu}_h^{n}\|^2\right) 
+ \frac{D_a^{-1}\Delta t}{2}\|\widetilde{\bu}_h^{n+1}\|^2  \\
&
\geq \frac{\nu\Delta t}{4 C_{PF}^2}(\| \widetilde{\bu}_h^{n+1}\|^2 + \|\widetilde{\bu}_h^{n}\|^2) + \frac{\nu \Delta t}{2}\|\nabla \widetilde{\bu}_h^{n+1}\|^2
+\frac{D_a^{-1}\Delta t}{4}\left(\|\widetilde{\bu}_h^{n+1}\|^2 + \|\widetilde{\bu}_h^{n}\|^2\right) 
+ \frac{D_a^{-1}\Delta t}{2}\|\widetilde{\bu}_h^{n+1}\|^2  
\\
&
\geq \left(\frac{\nu \Delta t}{4 C_{PF}^2} + \frac{D_a^{-1}\Delta t}{4}\right)\left(\|\widetilde{\bu}_h^{n+1}\|^2 + \|\widetilde{\bu}_h^{n}\|^2\right) 
 + \frac{\nu\Delta t}{4}\|\nabla \widetilde{\bu}_h^{n+1}\|^2 
 + \frac{D_a^{-1}\Delta t}{4}\|\widetilde{\bu}_h^{n+1}\|^2 
\\
&
\geq \min \left \{\frac{\nu \Delta t}{4 C_{PF}^2} + \frac{D_a^{-1}\Delta t}{4}, 1 \right \} \left(\|\widetilde{\bu}_h^{n+1}\|^2 + \frac{\nu\Delta t}{4}\|\nabla \widetilde{\bu}_h^{n+1}\|^2 
 + \frac{D_a^{-1}\Delta t}{4}\|\widetilde{\bu}_h^{n+1}\|^2 \right),
\end{align*}
and inserting this into \eqref{stabvel1}
\begin{align*}
&\|\widetilde{\bu}_h^{n+1}\|^2 + \frac{\nu\Delta t}{4}\|\nabla \widetilde{\bu}_h^{n+1}\|^2  + \frac{D_a^{-1}\Delta t}{4}\|\widetilde{\bu}_h^{n+1}\|^2\\
&\leq \frac{1}{(1+\alpha_u)} \left(\|\widetilde{\bu}_h^{n}\|^2 +\frac{\nu\Delta t}{4}\|\nabla \widetilde{\bu}_h^{n}\|^2 +\frac{D_a^{-1}\Delta t}{4}\|\widetilde{\bu}_h^{n}\|^2  \right)\\
& \,\,\,\,\, + \frac{1}{(1+\alpha_u)} \left({D_a}\Delta t\left(\beta_T^2 \|\widetilde{T}_h^{n}\|^2 + \beta_c^2\|\widetilde{S}_h^{n}\|^2  \right)\|g\|_{L^{\infty}}^2 + 
{\nu^{-1}}\Delta t\|\bF^{n+1}\|_{-1}^2
+ {\mu_1}\Delta t\|\bu^{n+1}\|^2 
\right).
\end{align*}
Multiplying by $2\Delta t$, using induction together with the stability estimates on the velocity and temperature gives the desired bound on the velocity.
\end{proof}
\subsection{Long Time $L^2 $Accuracy}
In this section, we provide two long time accuracy results of Algorithm~\ref{algbe}: the first for the nudging parameters $\mu_1, \mu_2, \mu_3>0$, and the second for $\mu_1>0,\,\, \mu_2, \mu_3=0.$ In the convergence theory given below, we use the Scott-Vogelius finite elements. This is due to the fact that it simplifies our analysis. In the use of the non-divergence-free element, the optimal convergence results can be still obtained. However, the analysis requires the estimation of some additional terms, and becomes more technical. \\
We note here that our convergence analysis needs a restriction on nudging parameters, and course mesh size.  
\begin{theorem}\label{teo1}[Long time $L^2$-accuracy with $\mu_1>0, \mu_2>0$ and $\mu_3>0$]
Let $\bu\in L^{\infty}(0, \infty; H^{k+1}(\Omega))$, $p\in L^{\infty}(0, \infty; H^{k}(\Omega))$ and $T, S \in L^{\infty}(0, \infty; H^{k+1}(\Omega))$ be true solutions of (\ref{bous}). Assume that $\Delta t$ is sufficiently small such that it holds
\begin{align*}
 \min\left\{{C\nu^{-1}\left( \|\bta_{\bu}^{n+1}\|^2_{L^{\infty}} + \|\bu^{n+1}\|^2_{L^{\infty}}\right)},\,\, {D_a\beta_T^2 \| g\|_{L^{\infty}}^2}, \,\, {D_a \beta_c^2 \| g\|_{L^{\infty}}^2} \right \}< \Delta t.
\end{align*}
In addition, we assume that nudging parameters satisfy the following
\begin{align*}
\max\{1, \, \,\nu^{-1}\big( \|\bta_{\bu}^{n+1}\|^2_{L^{\infty}} & + \|\bu^{n+1}\|^2_{L^{\infty}}\big), \, \,\kappa^{-1}\left( \|\bta_{T}^{n+1}\|^2_{L^{\infty}} + \|T^{n+1}\|^2_{L^{\infty}}\right),\, \, D_S^{-1}\left( \|\bta_{S}^{n+1}\|^2_{L^{\infty}} + \|S^{n+1}\|^2_{L^{\infty}}\right)   \}\\
&\quad\quad\quad \quad\quad\quad\quad\quad\quad\quad\quad\quad\quad\quad\quad\quad\quad\quad\quad\quad\quad\quad\quad\quad\quad\quad\quad\leq \mu_1\leq \frac{\nu}{CH^2},\\
\max \{1,\,\, 4D_a\beta_T^2 \|g\|_{L^{\infty}}^2 \}&\leq \mu_2\leq \frac{\kappa}{CH^2}, \hspace{4mm}
\max \{1,\,\, 4 D_c \beta_c^2 \|g\|_{L^{\infty}}^2 \}\leq \mu_3\leq \frac{\kappa}{CH^2}.
\end{align*}
Then, for any time level $t^{n+1}, \,\, n=0,1,...$, the errors between the true solutions and solutions of Algorithm~\ref{algbe} satisfies the bound
\begin{align*}
\|E_{\bu}^{n+1}\|^2 + \|E_{T}^{n+1}\|^2 +\|E_{h,S}^{n+1}\|^2
\leq C \lambda^{-1} K + \left(1+\lambda\Delta t\right)^{-(n+1)}\left(\|\bu^{0}\|^2 + \|T^{0}\|^2 + \|S^{0}\|^2 \right).
\end{align*}
where
\begin{align*}
K:= \left( D_a+ 
\mu_1^{-1}(\beta_T^2 + \beta_c^2) + C_{PF}^2\kappa^{-1} + C_{PF}^2 D_c^{-1}\right) h^{2k+2}  + (\nu^{-1} + \kappa^{-1}+ D_c^{-1})h^{2k+2} +\nu^{-1}h^{4k+4}\\
+(D_a^{-1} + \mu_1 + \mu_2\ +\mu_3 )h^{2k+2}
+ \Delta t^2 (D_a + C_{PF}^2\kappa^{-1} + C_{PF}^2 D_c^{-1} + D_a(\beta_T^2 +\beta_c^2)\| g\|_{L^{\infty}}^2 ),
\end{align*}
$ \lambda:=\min\{ \lambda_1, \lambda_2, \lambda_3\}$, and 
\begin{align*}
\lambda_1:={D_a^{-1}} +\frac{\mu_1}{2} +\frac{\nu C_{PF}^{-2}}{2},\, \,  \lambda_2:=\frac{\mu_2}{2} +\frac{\kappa \,  C_{PF}^{-2}}{2},\,\,
\lambda_3:=\frac{\mu_3}{2} +\frac{D_c\, C_{PF}^{-2}}{2}.
\end{align*}
\end{theorem}
\begin{proof}
True solutions of (\ref{bous}) at time level $n+1$ satisfies the equations:
\begin{gather}
\frac{1}{\Delta t}\big(\bu^{n+1}- \bu^{n}, \bv_h\big) + b_1\,(\bu^{n}, \bu^{n+1}, \,\bv_h) +  b_1\,(\bu^{n+1} -\bu^{n}, \bu^{n+1}, \,\bv_h) - (p^{n+1}, \nabla\cdot \bv_h)
+ \nu\big(\nabla\bu^{n+1}, \nabla \bv_h\big)\nonumber\\
\qquad \qquad \,  + \, D_a^{-1}(\bu^{n+1}, \bv_h)
 = \left(\frac{\bu^{n+1}- \bu^{n}}{\Delta t} -\bu_t^{n+1}, \bv_h\right) 
+ \, ((\beta_T (T^{n+1} - T^n) \, 
+ \, \beta_c (S^{n+1} - S^n))\bm g, \bv_h) \nonumber\\
\qquad \qquad  \qquad \qquad \qquad \qquad \qquad  \qquad \qquad \qquad\qquad \qquad + ((\beta_T T^n + \beta_c S^n)\bm g, \bv_h) 
\, 
+ \, (\bF^{n+1}, \bv_h), \label{true1}\\
\frac{1}{\Delta t}\big(T^{n+1}- T^{n}, \Psi_h\big) + b_2\,(\bu^{n+1}, T^{n+1}, \, \Psi_h) \, 
+ \, \kappa\big(\nabla T^{n+1}, \nabla \Psi_h\big) = \left(\frac{T^{n+1}-T^{n}}{\Delta t} -T_t^{n+1}, \Psi_h\right) \nonumber \\
\qquad \qquad  \qquad \qquad \qquad \qquad \qquad \qquad \qquad\qquad \qquad \qquad \qquad  \qquad \qquad \qquad \qquad \,  + \, (G^{n+1}, \Psi_h), \label{true2}\\
\frac{1}{\Delta t}\big(S^{n+1}- S^{n}, \chi_h\big) + b_3\,(\bu^{n+1}, S^{n+1}, \, \chi_h) + D_c\big(\nabla S^{n+1}, \nabla \chi_h\big) =
\big(\frac{S^{n+1}-S^{n}}{\Delta t} - S_t^{n+1}, \chi_h\big)
\nonumber \\
\qquad \qquad  \qquad \qquad \qquad \qquad \qquad \qquad \qquad\qquad \qquad \qquad \qquad  \qquad \qquad \qquad \qquad \,  + \, (\Phi^{n+1}, \chi_h)
\label{true3}. 
\end{gather}
Subtracting \eqref{disc1}-\eqref{disc3} from \eqref{true1}-\eqref{true3} yields the following error equations:
\begin{gather*}
\frac{1}{\Delta t}\big(\bE_{\bu}^{n+1}- \bE_{\bu}^{n}, \bv_h\big) + b_1\,(\bu^{n}, \bE_{\bu}^{n+1}, \,\bv_h)
 - b_1\,(\bE_{\bu}^{n}- \bE_{\bu}^{n+1}, \bE_{\bu}^{n+1}, \,\bv_h)
 + b_1\,(\bE_{\bu}^{n}- \bE_{\bu}^{n+1}, {\bu}^{n+1}, \,\bv_h) \nonumber\\
- b_1\,(\bE_{\bu}^{n+1}, \bE_{\bu}^{n+1}, \,\bv_h)
+ b_1\,(\bE_{\bu}^{n+1}, \bu^{n+1}, \,\bv_h)
+ b_1\,({\bu}^{n+1}- {\bu}^{n}, \bu^{n+1}, \,\bv_h)
+ \nu\big(\nabla \bE_{\bu}^{n+1}, \nabla \bv_h\big)\nonumber\\
+ D_a^{-1}(\bE_{\bu}^{n+1}, \bv_h)
+\mu_1(I_H(\bE_{\bu}^{n+1}), I_H({\bv_h}))\nonumber\\
= \left(\frac{\bu^{n+1}- \bu^{n}}{\Delta t} -\bu_t^{n+1}, \bv_h\right)
+ \left((\beta_T (T^{n+1} - T^n) + \beta_c(S^{n+1} - S^n)\right)\bm g, \bv_h)  
+ \left((\beta_T E_T^n + \beta_c E_S^n \right)\bm g, \bv_h),
\end{gather*}
and
\begin{gather*}
\frac{1}{\Delta t}\big(E_{T}^{n+1}- E_{T}^{n}, \Psi_h\big)
+ b_2\,(\bu^{n+1}, E_{T}^{n+1}, \, \Psi_h)
- b_2\,(\bE_{\bu}^{n+1}, E_{T}^{n+1}, \, \Psi_h)
+ b_2\,(\bE_{\bu}^{n+1}, {T}^{n+1}, \, \Psi_h)\\
+\kappa(\nabla E_T^{n+1}, \nabla \Psi_h)
+\mu_2(I_H(E_{T}^{n+1}), I_H({\Psi_h}))
= \left(\frac{T^{n+1}- T^{n}}{\Delta t} -T_t^{n+1}, \Psi_h\right),
\end{gather*}
and
\begin{gather*}
\frac{1}{\Delta t}\big(E_{S}^{n+1}- E_{S}^{n}, \chi_h\big)
+ b_3\,(\bu^{n+1},E_{S}^{n+1}, \, \chi_h)
- b_3\,(\bE_{\bu}^{n+1}, E_{S}^{n+1}, \, \chi_h)
+ b_3\,(\bE_{\bu}^{n+1}, {S}^{n+1}, \, \chi_h)\\
+ D_c(\nabla E_S^{n+1}, \nabla \chi_h)
+ \mu_3(I_H(E_{S}^{n+1}), I_H({\chi_h}))
= \left(\frac{S^{n+1}- S^{n}}{\Delta t} - S_t^{n+1}, \chi_h\right).
\end{gather*}
Decompose the errors into a term that lies in the discrete space $\bV_h$ and one outside 
\begin{align*}
\bE_{\bu}^{n+1}:=\phi_{h,\bu}^{n+1} - \bta_{\bu}^{n+1}, \,\, E_{T}^{n+1}:=\phi_{h,T}^{n+1} - \bta_{T}^{n+1},\,\, E_{S}^{n+1}:=\phi_{h,S}^{n+1} - \bta_{S}^{n+1}.
\end{align*}
Then, letting $\bv_h=\phi_{h,\bu}^{n+1}$, $\Psi_h=\phi_{h,T}^{n+1}$, $\chi_h=\phi_{h,S}^{n+1}$ in error equations results in
\begin{align}
\frac{1}{2\Delta t}&\big( \|\phi_{h,\bu}^{n+1}\|^2 - \|\phi_{h,\bu}^{n}\|^2  + \|\phi_{h,\bu}^{n+1}- \phi_{h,\bu}^{n}\|^2 \big) + \nu  \|\nabla\phi_{h,\bu}^{n+1}\|^2 + Da^{-1} \|\phi_{h,\bu}^{n+1}\|^2 + \mu_1 \|\phi_{h,\bu}^{n+1}\|^2 \nonumber\\
&=\frac{1}{\Delta t}\left(\bta_{\bu}^{n+1} - \bta_{\bu}^{n}, \phi_{h,\bu}^{n+1}\right) 
+ b_1\,(\bu^{n}, \bta_{\bu}^{n+1}, \,\phi_{h,\bu}^{n+1})
- b_1\,(\bE_{\bu}^{n}- \bE_{\bu}^{n+1}, \bta_{\bu}^{n+1}, \,\phi_{h,\bu}^{n+1})\nonumber\\
&
\, \, \, \, \, \,+ b_1\,(\bE_{\bu}^{n}- \bE_{\bu}^{n+1}, {\bu}^{n+1}, \,\phi_{h,\bu}^{n+1}) 
- b_1\,(\bE_{\bu}^{n+1}, \bta_{\bu}^{n+1}, \,\phi_{h,\bu}^{n+1})
+ b_1\,(\bE_{\bu}^{n+1}, \bu^{n+1}, \,\phi_{h,\bu}^{n+1})\nonumber\\
&
\, \, \, \, \, \, + b_1\,({\bu}^{n+1}- {\bu}^{n}, \bu^{n+1}, \,\phi_{h,\bu}^{n+1})
+ 2 \mu_1(I_H(\phi_{h,\bu}^{n+1})-\phi_{h,\bu}^{n+1}, \phi_{h,\bu}^{n+1}))+ \mu_1\|I_H(\phi_{h,\bu}^{n+1})-\phi_{h,\bu}^{n+1}\|^2 \nonumber\\
&
\, \, \,\, \, \,
+ \mu_1( I_H(\bta_{\bu}^{n+1}), I_H(\phi_{h,\bu}^{n+1}) )
+ D_a^{-1}(\bta_{\bu}^{n+1}, \phi_{h,\bu}^{n+1})
+ \left(\bu_t^{n+1} - \frac{\bu^{n+1}- \bu^{n}}{\Delta t} , \bv_h\right)\nonumber\\
&
\, \, \, \, \, \,+ \left((\beta_T (T^n - T^{n+1}) + \beta_c (S^n- S^{n+1})\right)\bm g, \bv_h)  
 - \left((\beta_T E_T^n + \beta_c E_S^n \right)\bm g, \bv_h)\label{vel1},
\end{align}
and 
\begin{align}
\frac{1}{2\Delta t}&\big( \|\phi_{h,T}^{n+1}\|^2 - \|\phi_{h,T}^{n}\|^2  + \|\phi_{h,T}^{n+1}- \phi_{h,T}^{n}\|^2 \big) + \kappa \|\nabla\phi_{h,T}^{n+1}\|^2 + \mu_2 \|\phi_{h,T}^{n+1}\|^2 \nonumber\\
&=\frac{1}{\Delta t}\left(\bta_{T}^{n+1} - \bta_{T}^{n}, \phi_{h,T}^{n+1}\right)
+ b_2\,(\bu^{n+1}, \eta_{T}^{n+1}, \, \phi_{h,T}^{n+1})
+ b_2\,(\bE_{\bu}^{n+1}, \eta_{T}^{n+1}, \, \phi_{h,T}^{n+1})\nonumber\\
& 
\, \, \, \, \, \,+ b_2\,(\bE_{\bu}^{n+1}, {T}^{n+1}, \, \phi_{h,T}^{n+1})
+ 2\mu_2(I_H(\phi_{h,T}^{n+1})-\phi_{h,T}^{n+1}, I_H(\phi_{h,T}^{n+1}))
+\mu_2 \|I_H(\phi_{h,T}^{n+1}) - \phi_{h,T}^{n+1}\|^2 \nonumber\\
& 
\, \, \, \, \, \,+ \mu_2(I_H(\bta_{T}^{n+1}), I_H(\phi_{h,T}^{n+1}))
+ \frac{\Delta t}{2}(T_{tt}(s^*), \phi_{h,T}^{n+1})\label{temp1},
\end{align}
and
\begin{align}
\frac{1}{2\Delta t}&\big( \|\phi_{h,S}^{n+1}\|^2 - \|\phi_{h, S}^{n}\|^2  + \|\phi_{h,S}^{n+1}- \phi_{h,S}^{n}\|^2 \big) + D_c  \|\nabla\phi_{h, S}^{n+1}\|^2 + \mu_3 \|\phi_{h, S}^{n+1}\|^2 \nonumber\\
&=\frac{1}{\Delta t}\left(\bta_{S}^{n+1} - \bta_{S}^{n}, \phi_{h,S}^{n+1}\right)
+ b_3\,(\bu^{n+1}, \eta_{S}^{n+1}, \, \phi_{h,S}^{n+1})
+ b_3\,(\be_{\bu}^{n+1}, \eta_{S}^{n+1}, \, \phi_{h,S}^{n+1})\nonumber\\
& 
\, \, \, \, \, \,+ b_3\,(\bE_{\bu}^{n+1}, {S}^{n+1}, \, \phi_{h,S}^{n+1})
+ 2\mu_3(I_H(\phi_{h,S}^{n+1})-\phi_{h,S}^{n+1}, I_H(\phi_{h,S}^{n+1}))
+\mu_3 \|I_H(\phi_{h,S}^{n+1}) - \phi_{h,S}^{n+1}\|^2 \nonumber\\
& 
\, \, \, \, \, \,+ \mu_3(I_H(\bta_{S}^{n+1}), I_H(\phi_{h,S}^{n+1}))
+ \frac{\Delta t}{2}(S_{tt}(s^*), \phi_{h,S}^{n+1})\label{con1}.
\end{align}
We first bound the right hand side terms of \eqref{vel1}. The first term is bounded using Taylor expansion followed by Cauchy-Schwarz and Young's Inequalities
\begin{align*}
\frac{1}{\Delta t}\left(\bta_{\bu}^{n+1} - \bta_{\bu}^{n}, \phi_{h, \bu}^{n+1}\right)& \, \leq  \,\|\bta_{t,\bu}(s^{**})\| \|\phi_{h, \bu}^{n+1}\|\\
\, &\leq \,  C D_a\|\bta_{t,\bu}(s^{**})\|^2 + \frac{D_a^{-1}}{12}\|\phi_{h,\bu}^{n+1}\|^2.
\end{align*}
For nudging terms, we will use Cauchy-Schwarz, Estimate \eqref{interp1} and Young's Inequality to get
\begin{align*}
	2\mu_1(I_H(\phi_{h,\bu}^{n+1})-\phi_{h,\bu}^{n+1}, I_H(\phi_{h,\bu}^{n+1}))& \leq 2 \mu_1\|I_H(\phi_{h, \bu}^{n+1})-\phi_{h, \bu}^{n+1}\|\|\phi_{h,\bu}^{n+1}\|\\
	&\leq C\mu_1 H^2 \|\nabla \phi_{h,\bu}^{n+1}\|^2 + \frac{\mu_1}{8}\|\phi_{h,\bu}^{n+1}\|^2,\\
	\mu_1 \|I_H(\phi_{h,\bu}^{n+1}) - \phi_{h,\bu}^{n+1}\|^2& \leq C\mu_1 H^2 \|\nabla \phi_{h, \bu}^{n+1}\|^2,\\
\mu_1(I_H(\bta_{\bu}^{n+1}), I_H(\phi_{h,\bu}^{n+1}))& \leq C\mu_1\|\bta_{\bu}^{n+1} \| \|\phi_{h,\bu}^{n+1}\|\\
&\leq C\mu_1\|\bta_{\bu}^{n+1} \| + \frac{\mu_1}{8}\|\phi_{h,\bu}^{n+1}\|^2.
\end{align*}
The last four terms are bounded as follows:
\begin{align*}
 D_a^{-1}(\bta_{\bu}^{n+1}, \phi_{h,\bu}^{n+1}) & \leq D_a^{-1}\|\bta_{\bu}^{n+1}\|\|\phi_{h,\bu}^{n+1}\|\\
 & \leq C  D_a^{-1}\|\bta_{\bu}^{n+1}\|^2 + \frac{D_a^{-1}}{12}\|\phi_{h,\bu}^{n+1}\|^2\\
\frac{\Delta t}{2}(\bu_{tt}(s^*), \phi_{h,\bu}^{n+1})
&\leq C\Delta t\|\bu_{tt}(s^*)\| \|\phi_{h,\bu}^{n+1}\|\\
& \leq C  D_a \Delta t^2\|\bu_{tt}(s^{*})\|^2 + \frac{D_a^{-1}}{12}\|\phi_{h,\bu}^{n+1}\|^2,
\end{align*}
and
\begin{align*}
((\beta_T (T^n - T^{n+1}) & + \beta_c (S^n- S^{n+1}))\bm g, \phi_{h,\bu}^{n+1}) \\
& \leq\Delta t ( \beta_T T_t(s^*) + \beta_c S_t(s^*))\bm g, \, \phi_{h,\bu}^{n+1}) \\
& \leq \Delta t \left(\beta_T\| T_t(s^*)\| + \beta_c\| S_t(s^*)\| )\|\bm g\|_{L^{\infty}}\right)\|\phi_{h,\bu}^{n+1}\|\\
& \leq D_a \Delta t^2\left(\beta_T^2\| T_t\|^2_{L^{\infty}(0,\infty; L^2)} + \beta_c^2\| S_t\|^2_{L^{\infty}(0,\infty; L^2)} )\|\bm g\|_{L^{\infty}}^2\right) +\frac{D_a^{-1}}{12}\|\phi_{h,\bu}^{n+1}\|^2,
\end{align*}
The last term is first rewritten as follows
\begin{gather*}
-\left((\beta_T E_T^n + \beta_c E_S^n \right)\bm g, \phi_{h,\bu}^{n+1})\\
= \left((\beta_T( E_T^{n+1} - E_T^{n}) + \beta_c (E_S^{n+1}-E_S^n)\right)\bm g, \phi_{h,\bu}^{n+1}) + \left((\beta_T E_T^{n+1} + \beta_c E_S^{n+1} \right)\bm g, \phi_{h,\bu}^{n+1}).
\end{gather*}
After expanding error terms, we apply Cauchy-Schwarz and Young's Inequalities to the first term of the right hand side to get  
\begin{gather*}
\left((\beta_T( E_T^{n+1}-E_T^{n}) + \beta_c (E_S^{n+1}-E_S^n)\right)\bm g, \phi_{h,\bu}^{n+1})\\
 \leq \left((\beta_T (\eta_T^{n+1}- \eta_T^n) + \beta_c (\eta_S^{n+1}-\eta_S^n )\right)\bm g, \phi_{h,\bu}^{n+1})
-\left((\beta_T (\phi_{h,T}^{n+1} - \phi_{h,T}^n) + \beta_c(\phi_{h, S}^{n+1} -\phi_{h, S}^{n} )\right)\bm g, \phi_{h,\bu}^{n+1})\\
\leq \left( \beta_T \|\eta_T^{n+1}-\eta_T^n \| + \beta_c \|\eta_S^{n+1}-\eta_S^n \|\right)\|\bm g\|_{L^{\infty}}\|\phi_{h,\bu}^{n+1}\|
+
\left( \beta_T \|\phi_{h,T}^{n+1} - \phi_{h,T}^n\| + \beta_c \|\phi_{h,S}^{n+1} - \phi_{h,S}^n \|\right)\|\bm g\|_{L^{\infty}}\|\phi_{h,\bu}^{n+1}\|\\
\leq
C \mu_1^{-1}\left( \beta_T^2 \|\eta_T^{n+1}-\eta_T^n \|^2 + \beta_c^2 \|\eta_S^{n+1}-\eta_S^n \|^2\right)\|\bm g\|_{L^{\infty}}^2 + \frac{\mu_1}{8}\|\phi_{h,\bu}^{n+1}\|^2 \\
\, \, \, \, \, \, \, \, \, \, \, \,\, \, \, \, \, \,+
D_a\left( \beta_T^2 \|\phi_{h,T}^{n+1} - \phi_{h,T}^n\|^2 + \beta_c^2\|\phi_{h,S}^{n+1} - \phi_{h,S}^n \|^2\right)\|\bm g\|_{L^{\infty}}^2 + \frac{D_a^{-1}}{12}\|\phi_{h,\bu}^{n+1}\|.
\end{gather*}
In a similar manner, we have the bound for the second term
\begin{align}
\big((\beta_T E_T^{n+1} & + \beta_c E_S^{n+1} \big)\bm g, \phi_{h,\bu}^{n+1})
 = \left(\left(\beta_T \eta_T^{n+1} + \beta_c \eta_S^{n+1} \right)\bm g, \phi_{h,\bu}^{n+1}\right) - \left((\beta_T \phi_{h,T}^{n+1} + \beta_c \phi_{h,S}^{n+1} \right)\bm g, \phi_{h,\bu}^{n+1})\nonumber\\
&\leq \left(\beta_T \|\eta_T^{n+1}\| + \beta_c \|\eta_S^{n+1}\| \right)\|\bm g\|_{L^{\infty}}\| \phi_{h,\bu}^{n+1}\| + \left(\beta_T \|\phi_{h,T}^{n+1}\| + \beta_c \|\phi_{h,S}^{n+1}\| \right)\|\bm g\|_{L^{\infty}}\| \phi_{h,\bu}^{n+1}\|\nonumber\\
& 
\leq C \mu_1^{-1}\left(\beta_T^2 \|\eta_T^{n+1}\|^2 + \beta_c^2 \|\eta_S^{n+1}\|^2 \right)\|\bm g\|^2_{L^{\infty}} +\frac{\mu_1}{8}\| \phi_{h,\bu}^{n+1}\|^2\nonumber\\
&\, \, \, \, \, \, +\left(\beta_T^2 \|\phi_{h,T}^{n+1}\|^2 + \beta_c^2 \|\phi_{h,S}^{n+1}\|^2 \right)\|\bm g\|_{L^{\infty}}^2 + \frac{D_a^{-1}}{12}\| \phi_{h,\bu}^{n+1}\|^2.\label{farkterim}
\end{align}
To estimate the first non-linear term, we use H{\"{o}}lder and Young's Inequalities 
\begin{align*}
	b_1\,(\bu^{n}, \eta_{\bu}^{n+1}, \, \phi_{h,\bu}^{n+1})
& =
(\bu^{n}\cdot\nabla\eta_{\bu}^{n+1}, \, \phi_{h,\bu}^{n+1} ) = - (\bu^{n}\cdot\nabla\phi_{h,\bu}^{n+1}, \eta_{\bu}^{n+1})\\
&
\leq
\|\bu^{n}\|_{L^{\infty}}\|\nabla\phi_{h,\bu}^{n+1}\| \|\eta_{\bu}^{n+1}\|\\
&
\leq C\nu^{-1}\|\bu^{n}\|_{L^{\infty}}^2 \|\eta_{\bu}^{n+1}\|^2 + \frac{\nu}{20} \|\nabla\phi_{h,\bu}^{n+1}\|^2.
\end{align*}
To bound the remaining four terms, we first expand error terms, then apply H{\"{o}}lder and Young's Inequalities to get
\begin{align*}
 b_1\,(\bE_{\bu}^{n+1} - \bE_{\bu}^{n}, \bta_{\bu}^{n+1}, \, \phi_{h,\bu}^{n+1})&=-((\bta_{\bu}^{n} - \bta_{\bu}^{n+1})\cdot\nabla \phi_{h,\bu}^{n+1}, \, \eta_{\bu}^{n+1}) + ((\phi_{h,\bu}^{n}-\phi_{h,\bu}^{n+1})\cdot\nabla \phi_{h, \bu}^{n+1},\, \eta_{\bu}^{n+1})\\
&
\leq \| \bta_{\bu}^{n} - \bta_{\bu}^{n+1}\| \|\nabla \phi_{h,T}^{n+1}\| \|\eta_{\bu}^{n+1}\|_{L^{\infty}} + \|\phi_{h,\bu}^{n}-\phi_{h,\bu}^{n+1}\| \|\nabla \phi_{h, \bu}^{n+1}\| \|\eta_{\bu}^{n+1}\|_{L^{\infty}}\\
&
\leq C\nu^{-1}\|\bta_{\bu}^{n} - \bta_{\bu}^{n+1}\| ^2 \|\eta_{\bu}^{n+1}\|^2_{L^{\infty}} + \frac{\nu}{20} \|\nabla \phi_{h,\bu}^{n+1}\|^2 \\
& \, \, \, \, \,+  C\nu^{-1}\|\phi_{h,\bu}^{n}-\phi_{h,\bu}^{n+1}\|^2 \|\eta_{\bu}^{n+1}\|^2_{L^{\infty}} + \frac{\nu}{20}\|\nabla \phi_{h,\bu}^{n+1}\|^2, 
 \end{align*}
 and 
\begin{align*}
 b_1\,(\bE_{\bu}^{n+1} -  \bE_{\bu}^{n}, {\bu}^{n+1}, \, \phi_{h,\bu}^{n+1}) &=-((\bta_{\bu}^{n} - \bta_{\bu}^{n+1})\cdot\nabla \phi_{h,\bu}^{n+1}, \, {\bu}^{n+1}) + ((\phi_{h,\bu}^{n}-\phi_{h,\bu}^{n+1})\cdot\nabla \phi_{h, \bu}^{n+1},\, {\bu}^{n+1})\\
& 
 \leq
 \|\bta_{\bu}^{n} - \bta_{\bu}^{n+1}\| \|\nabla \phi_{h,\bu}^{n+1}\| \| {\bu}^{n+1}\|_{L^{\infty}} 
 +
 \|\phi_{h,\bu}^{n}-\phi_{h,\bu}^{n+1}\| \|\nabla \phi_{h, \bu}^{n+1}\| \| {\bu}^{n+1}\|_{L^{\infty}}\\
 &
 \leq 
 C\nu^{-1}\|\bta_{\bu}^{n} - \bta_{\bu}^{n+1}\|^2 \| {\bu}^{n+1}\|_{L^{\infty}}^2 +\frac{\nu}{20}\|\nabla \phi_{h,\bu}^{n+1}\|^2 \\
&
 \, \, \, \, \,+ C\nu^{-1} \|\phi_{h,\bu}^{n}-\phi_{h,\bu}^{n+1}\|^2\| {\bu}^{n+1}\|_{L^{\infty}}^2 + \frac{\nu}{20} \|\nabla \phi_{h, \bu}^{n+1}\|^2,
 \end{align*} 
 and
\begin{align*}
 b_1\,(\bE_{\bu}^{n+1}, \bta_{\bu}^{n+1}, \, \phi_{h,\bu}^{n+1})&=-(\bta_{\bu}^{n+1}\cdot\nabla \phi_{h,\bu}^{n+1}, \, \eta_{\bu}^{n+1}) + (\phi_{h,\bu}^{n+1}\cdot\nabla \phi_{h, \bu}^{n+1},\, \eta_{\bu}^{n+1})\\
&
\leq \|\bta_{\bu}^{n+1}\| \|\nabla \phi_{h, \bu}^{n+1}\| \|\eta_{\bu}^{n+1}\|_{L^{\infty}} + \|\phi_{h,\bu}^{n+1}\| \|\nabla \phi_{h, \bu}^{n+1}\| \|\eta_{\bu}^{n+1}\|_{L^{\infty}}\\
&
\leq C\nu^{-1}\|\bta_{\bu}^{n+1}\| ^2 \|\eta_{\bu}^{n+1}\|^2_{L^{\infty}} + \frac{\nu}{20} \|\nabla \phi_{h,\bu}^{n+1}\|^2 \\
& \, \, \, \, \,+  C\nu^{-1}\|\phi_{h,\bu}^{n+1}\|^2 \|\eta_{\bu}^{n+1}\|^2_{L^{\infty}} + \frac{\nu}{20}\|\nabla \phi_{h,\bu}^{n+1}\|^2,
 \end{align*}
and 
\begin{align*}
 b_1\,(\bE_{\bu}^{n+1}, {\bu}^{n+1}, \, \phi_{h,\bu}^{n+1})&=-(\bta_{\bu}^{n+1}\cdot\nabla \phi_{h,\bu}^{n+1}, \, {\bu}^{n+1}) + (\phi_{h,\bu}^{n+1}\cdot\nabla \phi_{h, \bu}^{n+1},\, {\bu}^{n+1})\\
&
\leq \|\bta_{\bu}^{n+1}\| \|\nabla \phi_{h, \bu}^{n+1}\| \|{\bu}^{n+1}\|_{L^{\infty}} + \|\phi_{h,\bu}^{n+1}\| \|\nabla \phi_{h, \bu}^{n+1}\| \|{\bu}^{n+1}\|_{L^{\infty}}\\
&
\leq C\nu^{-1}\|\bta_{\bu}^{n+1}\| ^2 \|{\bu}^{n+1}\|^2_{L^{\infty}} + \frac{\nu}{20} \|\nabla \phi_{h,\bu}^{n+1}\|^2 \\
& \, \, \, \, \,+  C\nu^{-1}\|\phi_{h,\bu}^{n+1}\|^2 \|{\bu}^{n+1}\|^2_{L^{\infty}} + \frac{\nu}{20}\|\nabla \phi_{h,\bu}^{n+1}\|^2. 
\end{align*}
We estimate the last term by using Taylor expansion in the first term, and then apply standard inequalities
\begin{align*}
 b_1\,({\bu}^{n+1} -{\bu}^{n}, {\bu}^{n+1}, \, \phi_{h,\bu}^{n+1})&=-(({\bu}^{n+1} -{\bu}^{n})\cdot\nabla \phi_{h,\bu}^{n+1}, \, {\bu}^{n+1})
 = -\Delta t \, ({\bu}_{t}(t^*)\cdot\nabla \phi_{h,\bu}^{n+1}, \, {\bu}^{n+1}) \\
&
\leq \Delta t\|{\bu}_{t}(t^*)\| \|\nabla \phi_{h, \bu}^{n+1}\| \|{\bu}^{n+1}\|_{L^{\infty}}\\
&
\leq C\nu^{-1}\Delta t^2\|{\bu}_{t}(t^*)\|^2 \|{\bu}^{n+1}\|^2_{L^{\infty}} + \frac{\nu}{20} \|\nabla \phi_{h,\bu}^{n+1}\|^2 .
 \end{align*} 
Plugging all these estimates to the right hand side of \eqref{vel1}, and rearranging terms yields
\begin{align}
&\frac{1}{2\Delta t}\big( \|\phi_{h,\bu}^{n+1}\|^2 - \|\phi_{h,\bu}^{n}\|^2  \big)
+ \left( \frac{1}{2\Delta t} -C\nu^{-1}(\|\eta_{\bu}^{n+1}\|^2_{L^{\infty}}+\|{\bu}^{n+1}\|^2_{L^{\infty}})\right)\|\phi_{h,\bu}^{n+1}- \phi_{h,\bu}^{n}\|^2 \nonumber\\
& + \frac{\nu}{4} \|\nabla\phi_{h,\bu}^{n+1}\|^2  + \left( \frac{\nu}{4} -C\mu_1 H^2\right) \|\nabla\phi_{h,\bu}^{n+1}\|^2
 +\frac{Da^{-1} }{2}\|\phi_{h,\bu}^{n+1}\|^2 +\frac{\mu_1 }{2} \|\phi_{h,\bu}^{n+1}\|^2 \nonumber\\
 &
\leq C D_a\|\bta_{t,\bu}\|^2_{L^{\infty}(0,\infty; L^2)}\nonumber\\
&
+
C\nu^{-1}\left( \|\bu^{n}\|_{L^{\infty}}^2 \|\bta_{\bu}^{n+1}\|^2 +  \|\bta_{\bu}^{n+1}\|^2_{L^{\infty}} \|\bta_{\bu}^{n}-\bta_{\bu}^{n+1}\|^2 + \|{\bu}^{n+1}\|^2_{L^{\infty}} \|\bta_{\bu}^{n}-\bta_{\bu}^{n+1}\|^2
+\|\bta_{\bu}^{n+1}\|^2_{L^{\infty}} \|\bta_{\bu}^{n+1}\|^2\right)\nonumber\\
&
+ C\nu^{-1}\left( \|\bta_{\bu}^{n+1}\|_{L^{\infty}}^2  +  \|{\bu}^{n+1}\|_{L^{\infty}}^2  \right)\|\phi_{h,\bu}^{n+1}\|^2
+ C D_a^{-1}\|\bta_{\bu}^{n+1}\|^2 + C\mu_1\|\bta_{\bu}^{n+1}\|^2 + C D_a \Delta t^2\|\bu_{tt}(t^*)\|^2 \nonumber\\
& 
+C D_a \Delta t^2 \left(\beta_T^2 \|T_t\|^2_{L^{\infty}(0,\infty;L^2)} + \beta_c^2\|S_t\|^2_{L^{\infty}(0,\infty;L^2)} \right)\|\bm g\|^2_{L^{\infty}} \nonumber\\
&
+C \mu_1^{-1}\left( \beta_T^2 \|\bta_{T}^n - \bta_{T}^{n+1}\|^2 + \beta_c^2  \|\bta_{S}^n - \bta_{S}^{n+1}\|^2\right)\|\bm g\|^2_{L^{\infty}}\nonumber\\
&
+C D_a\left( \beta_T^2 \|\phi_{h,T}^n - \phi_{h,T}^{n+1}\|^2 + \beta_c^2  \|\phi_{h,T}^n - \phi_{h,T}^{n+1}\|^2\right)\|\bm g\|^2_{L^{\infty}}\nonumber\\
&
+C \mu_1^{-1}\left( \beta_T^2 \|\bta_{T}^{n+1}\|^2 + \beta_c^2  \|\bta_{S}^{n+1}\|^2\right)\|\bm g\|^2_{L^{\infty}} +C D_a\left( \beta_T^2 \|\phi_{h,T}^{n+1}\|^2 + \beta_c^2  \|\phi_{h,S}^{n+1}\|^2\right)\|\bm g\|^2_{L^{\infty}}\label{vel2}.
\end{align}
We bound the right hand side terms of \eqref{temp1} in a similar manner. Using Taylor expansion, Cauchy-Schwarz and Young's inequalities along with the interpolation property \ref{interp1} and \ref{interp2} yields
\begin{align*}
\frac{1}{\Delta t}\left(\bta_{T}^{n+1} - \bta_{T}^{n}, \phi_{T}^{n+1}\right)& \, \leq  \,\|\bta_{t,T}(s^{**})\| C_{PF}\|\nabla\phi_{T}^{n+1}\|\\
&
\, \leq \,  C C_{PF}^2 \kappa^{-1}\|\bta_{t,T}(s^{**})\|^2 + \frac{\kappa}{14}\|\nabla\phi_{h, T}^{n+1}\|^2,\\
	2\mu_2(I_H(\phi_{h,T}^{n+1})-\phi_{h,T}^{n+1}, I_H(\phi_{h,T}^{n+1}))
	& \leq 2 \mu_2\|I_H(\phi_{h,T}^{n+1})-\phi_{h,T}^{n+1}\|\|\phi_{h,T}^{n+1}\|\\
	&\leq C\mu_2 H^2 \|\nabla \phi_{h,T}^{n+1}\|^2 + \frac{\mu_2}{4}\|\phi_{h,T}^{n+1}\|^2,\\
	\mu_2 \|I_H(\phi_{h,T}^{n+1}) - \phi_{h,T}^{n+1}\|^2& \leq C\mu_2 H^2 \|\nabla \phi_{h,T}^{n+1}\|^2,\\
\mu_2(I_H(\bta_{T}^{n+1}), I_H(\phi_{h,T}^{n+1}))
	& \leq C\mu_2\|\bta_{T}^{n+1} \| \|\phi_{h,T}^{n+1}\|\nonumber\\
	& \leq C\mu_2\|\bta_{T}^{n+1} \|^2 + \frac{\mu_2}{4}\|\phi_{h,T}^{n+1}\|^2,\\	
	\frac{\Delta t}{2}(T_{tt}(s^*), \phi_{h,T}^{n+1})
&\leq C\Delta t\|T_{tt}(s^*)\| C_{PF}\|\nabla\phi_{h,T}^{n+1}\|\nonumber\\
&\leq C C_{PF}^2 \kappa^{-1} \Delta t^2\|T_{tt}(s^{*})\|^2 + \frac{\kappa}{14}\|\nabla\phi_{h, T}^{n+1}\|^2.
\end{align*}
The non-linear terms are estimated by applying H{\"{o}}lder and Young's inequalities
\begin{align*}
	b_2\,(\bu^{n+1}, \eta_{T}^{n+1}, \, \phi_{h,T}^{n+1})
&=
(\bu^{n+1}\cdot\nabla\eta_{T}^{n+1}, \, \phi_{h,T}^{n+1} ) = - (\bu^{n+1}\cdot\nabla\phi_{h,T}^{n+1}, \eta_{T}^{n+1})\\
&
\leq C \|\bu^{n+1}\|_{L^{\infty}}\|\nabla\phi_{h,T}^{n+1}\|\|\eta_{T}^{n+1}\|\\
&
\leq C \kappa ^{-1}\|\bu^{n+1}\|^2_{L^{\infty}}\|\eta_{T}^{n+1}\|^2 + \frac{\kappa}{14}\|\nabla\phi_{h,T}^{n+1}\|^2,\\
	 b_2\,(\bE_{\bu}^{n+1}, \eta_{T}^{n+1}, \, \phi_{h,T}^{n+1})&=-(\bta_{\bu}^{n+1}\cdot\nabla \phi_{h,T}^{n+1}, \eta_{T}^{n+1}) + (\phi_{h,\bu}^{n+1}\cdot\nabla \phi_{h,T}^{n+1}, \eta_{T}^{n+1})\\
& \leq C\|\bta_{\bu}^{n+1}\|\|\nabla \phi_{h,T}^{n+1}\| \|\eta_{T}^{n+1}\|_{L^{\infty}} 
+  C\|\phi_{h,\bu}^{n+1}\|\|\nabla \phi_{h,T}^{n+1}\| \|\eta_{T}^{n+1}\|_{L^{\infty}}\\
& \leq C\kappa^{-1}\|\bta_{\bu}^{n+1}\|^2 \|\eta_{T}^{n+1}\|_{L^{\infty}}^2 +\frac{\kappa}{14}\|\nabla \phi_{h,T}^{n+1}\|^2\\
& \, \, \, \, \, \,+ C \kappa ^{-1}\|\phi_{h,\bu}^{n+1}\|^2\|\eta_{T}^{n+1}\|_{L^{\infty}}^2+\frac{\kappa}{14}\|\nabla \phi_{h,T}^{n+1}\|^2,\\
 	b_2\,(\bE_{\bu}^{n+1}, {T}^{n+1}, \, \phi_{h,T}^{n+1}) & = (\bE_{\bu}^{n+1}\cdot\nabla {T}^{n+1}, \, \phi_{h,T}^{n+1})\\
 & =-(\bta_{\bu}^{n+1}\cdot\nabla \phi_{h,T}^{n+1}, {T}^{n+1}) + (\phi_{h, \bu}^{n+1}\cdot\nabla\phi_{h,T}^{n+1}, {T}^{n+1})\\
 & \leq C\kappa^{-1}\|\bta_{\bu}^{n+1}\|\|{T}^{n+1}\|_{L^{\infty}}\|\nabla \phi_{h,T}^{n+1}\| + C\kappa^{-1}\|\phi_{h, \bu}^{n+1}\|\|\nabla\phi_{h,T}^{n+1}\|\|{T}^{n+1}\|_{L^{\infty}}\\
 &\leq C\kappa^{-1}\|\bta_{\bu}^{n+1}\|^2_{L^{\infty}} \|T^{n+1}\|^ 2 +\frac{\kappa}{14}\|\nabla \phi_{h,T}^{n+1}\|^2\\
& \, \, \, \, \, \,+ C\kappa^{-1}
\|\phi_{h,\bu}^{n+1}\|^2\|{T}^{n+1}\|_{L^{\infty}}^2 +\frac{\kappa}{14}\|\nabla \phi_{h,T}^{n+1}\|^2.
\end{align*}
Plugging all these estimates into right hand side of \eqref{temp1} and reducing leads to
\begin{align}
&\frac{1}{2\Delta t}\big( \|\phi_{h,T}^{n+1}\|^2 - \|\phi_{h,T}^{n}\|^2  + \|\phi_{h,T}^{n+1}- \phi_{h,T}^{n}\|^2 \big) 
+ \frac{\kappa}{4}\|\nabla\phi_{h,T}^{n+1}\|^2
+ \left(\frac{\kappa}{4} -C\mu_2H^2\right)\|\nabla\phi_{h,T}^{n+1}\|^2 
+ \frac{\mu_2}{2} \|\phi_{h,T}^{n+1}\|^2 \nonumber\\
&\leq C C_{PF}^2\kappa^{-1} \|\eta_{t,T}\|^2_{L^{\infty}\left(0,\infty,; L^2\right)} 
+  \kappa^{-1}\big(\|\bu^{n+1}\|_{L^{\infty}}^2\|\eta_T^{n+1}\|^2 + \|\eta_{\bu}^{n+1}\|^2\|\eta_T^{n+1}\|^2_{L^{\infty}}  +\|T^{n+1}\|^2_{L^{\infty}} \|\eta_{\bu}^{n+1}\|^2 \big)\nonumber\\
&\, \, \, \, \, \, + C \kappa^{-1}\left(\|T^{n+1}\|^2_{L^{\infty}} +   \|\eta_T^{n+1}\|^2_{L^{\infty}}\right) \|\phi_{h,\bu}^{n+1}\|^2 + C\mu_2 \|\eta_T^{n+1}\|^2 + C C_{PF}^2 \kappa^{-1}\Delta t^2\|T_{tt}\|_{L^{\infty}\left(0,\infty,; L^2\right)}^2\label{temp2}.
\end{align}
Similarly bounding the right hand side terms of \eqref{con1}, we obtain
\begin{align}
&\frac{1}{2\Delta t}\big( \|\phi_{h,S}^{n+1}\|^2 - \|\phi_{h,S}^{n}\|^2  + \|\phi_{h, S}^{n+1}- \phi_{h,S}^{n}\|^2 \big) 
+ \frac{D_c}{4}\|\nabla\phi_{h, S}^{n+1}\|^2
+ \left(\frac{D_c}{4} -C\mu_3 H^2\right)\|\nabla\phi_{h,S}^{n+1}\|^2 
+ \frac{\mu_3}{2} \|\phi_{h,S}^{n+1}\|^2 \nonumber\\
&\leq C C_{PF}^2D_c^{-1} \|\eta_{t,S}\|^2_{L^{\infty}\left(0,\infty,; L^2\right)} 
+ C D_c^{-1}\big(\|\bu^{n+1}\|_{L^{\infty}}^2\|\eta_S^{n+1}\|^2 + \|\eta_{\bu}^{n+1}\|^2\|\eta_S^{n+1}\|^2_{L^{\infty}}  + \|S^{n+1}\|^2_{L^{\infty}} \|\eta_{\bu}^{n+1}\|^2 \big)\nonumber\\
&\, \, \, \, \, \, + CD_c^{-1}\left(\|S^{n+1}\|^2_{L^{\infty}} +   \|\eta_S^{n+1}\|^2_{L^{\infty}}\right) \|\phi_{h,\bu}^{n+1}\|^2 + C\mu_3 \|\eta_S^{n+1}\|^2 + C C_{PF}^2 D_c^{-1}\Delta t^2\|S_{tt}\|_{L^{\infty}\left(0,\infty,; L^2\right)}^2\label{con2}.
\end{align}
Summing \eqref{vel2}, \eqref{temp2} and \eqref{con2} produces:
\begin{gather*}
\frac{1}{2\Delta t}\left(\|\phi_{h,\bu}^{n+1}\|^2 + \|\phi_{h,T}^{n+1}\|^2 + \|\phi_{h,S}^{n+1}\|^2\right) - \frac{1}{2\Delta t}\left(\|\phi_{h,\bu}^{n}\|^2 + \|\phi_{h,T}^{n}\|^2 + \|\phi_{h,S}^{n}\|^2  \right) \\
+ \left( \frac{1}{2\Delta t} - C\nu^{-1}(\|\eta_{\bu}^{n+1}\|^2_{L^{\infty}}+\|{\bu}^{n+1}\|^2_{L^{\infty}})\right)\|\phi_{h,\bu}^{n+1}- \phi_{h,\bu}^{n}\|^2 
+\left( \frac{1}{2\Delta t} - D_a\beta_T^2\|\bm g\|_{L^{\infty}}^2\right)\|\phi_{h,T}^{n+1}- \phi_{h,T}^{n}\|^2\\
+\left( \frac{1}{2\Delta t} - D_a \beta_c^2\|\bm g\|_{L^{\infty}}^2\right)\|\phi_{h,S}^{n+1}- \phi_{h,S}^{n}\|^2 
+
\frac{\nu}{4}\|\nabla\phi_{h,\bu}^{n+1}\|^2 
+ 
\left(\frac{\nu}{4} - C\mu_1 H^2 \right)\|\nabla\phi_{h,\bu}^{n+1}\|^2\\
+\frac{\kappa}{4}\|\nabla\phi_{h,T}^{n+1}\|^2 + \left(\frac{\kappa}{4} -C\mu_2H^2 \right)\|\nabla\phi_{h,T}^{n+1}\|^2
+\frac{D_c}{4}\|\nabla\phi_{h,S}^{n+1}\|^2 + \left(\frac{D_c}{4} - C\mu_3 H^2 \right)\|\nabla\phi_{h, S}^{n+1}\|^2\\
+ \frac{D_a^{-1}}{2}\|\phi_{h,\bu}^{n+1}\|^2 +\frac{\mu_1}{4}\|\phi_{h,\bu}^{n+1}\|^2\\
+
\bigg( \frac{\mu_1}{4} - C \nu^{-1} \left(\|\bta_{\bu}^{n+1}\|_{L^{\infty}}^2 + \|{\bu}^{n+1}\|_{L^{\infty}}^2\right) - C\kappa^{-1} \left(\|\bta_{T}^{n+1}\|_{L^{\infty}}^2 + \|{T}^{n+1}\|_{L^{\infty}}^2 \right) \\
- C D_c^{-1} \left(\|\bta_{S}^{n+1}\|_{L^{\infty}}^2 + \|S^{n+1}\|_{L^{\infty}}^2 \right) \bigg) \|\phi_{h,\bu}^{n+1}\|^2\\
+
\frac{\mu_2}{4}\|\phi_{h,T}^{n+1}\|^2 
+ 
\left( \frac{\mu_2}{4} - D_a \beta_T^2\|\bm g\|_{L^{\infty}}^2\right)\|\phi_{h,T}^{n+1}\|^2 
+
\frac{\mu_3}{4}\|\phi_{h,S}^{n+1}\|^2 
+ 
\left( \frac{\mu_3}{4} - D_a \beta_c^2 \|\bm g\|_{L^{\infty}}^2\right)\|\phi_{h, S}^{n+1}\|^2 \\
\leq
C\bigg(
D_a\|\bta_{t,\bu}\|^2_{L^{\infty}(0,\infty; L^2)} + C_{PF}^2 \kappa^{-1} \|\eta_{t,T}\|^2_{L^{\infty}(0,\infty; L^2)} +C_{PF}^2 D_c^{-1} \|\bta_{t,S}\|^2_{L^{\infty}(0,\infty; L^2)}
\bigg)\\
+ C\nu^{-1}\bigg( ( \|\bu^{n}\|_{L^{\infty}}^2 + \|\bu^{n+1}\|_{L^{\infty}}^2 ) \|\bta_{\bu}^{n+1}\|^2 +  \|\bta_{\bu}^{n+1}\|^2_{L^{\infty}} \|\bta_{\bu}^{n}-\bta_{\bu}^{n+1}\|^2 + \|{\bu}^{n+1}\|^2_{L^{\infty}} \|\bta_{\bu}^{n}-\bta_{\bu}^{n+1}\|^2\\
+\|\bta_{\bu}^{n+1}\|^2_{L^{\infty}} \|\bta_{\bu}^{n+1}\|^2\bigg)\nonumber\\ 
+ C\kappa^{-1}\bigg(\|\bu^{n+1}\|_{L^{\infty}}^2\|\eta_T^{n+1}\|^2 + \|\eta_{\bu}^{n+1}\|^2\|\eta_T^{n+1}\|^2_{L^{\infty}}  +\|T^{n+1}\|^2_{L^{\infty}} \|\eta_{\bu}^{n+1}\|^2 \bigg)\nonumber\\
+ C D_c^{-1}\bigg(\|\bu^{n+1}\|_{L^{\infty}}^2\|\eta_S^{n+1}\|^2 + \|\eta_{\bu}^{n+1}\|^2\|\eta_S^{n+1}\|^2_{L^{\infty}}  + \|S^{n+1}\|^2_{L^{\infty}} \|\eta_{\bu}^{n+1}\|^2 \bigg)\nonumber\\
+ C \left( D_a^{-1}\|\bta_{\bu}^{n+1}\|^2 + \mu_1\|\bta_{\bu}^{n+1}\|^2 + \mu_2\|\eta_{T}^{n+1}\|^2 + \mu_3\|\eta_{S}^{n+1}\|^2  \right) \nonumber\\
+ C\Delta t^2 \bigg(
D_a \|\bu_{tt}\|_{L^{\infty}(0,\infty;L^2)}^2 + C_{PF}^2\kappa^{-1}\|T_{tt}\|_{L^{\infty}(0,\infty;L^2)}^2 	 + C_{PF}^2 D_c^{-1}\|S_{tt}\|_{L^{\infty}(0,\infty;L^2)^2} \bigg) \\	
+ D_a \Delta t^2 \left( \beta_T^2 \|T_t\|_{L^{\infty}(0,\infty;L^2)}^2 +  \beta_c^2 \|S_t\|_{L^{\infty}(0,\infty;L^2)}^2\right) \\
+\mu_1^{-1}\bigg(\beta_T^2\left( \|\eta_T^n -\eta_T^{n+1}\|^2 + \|\eta_T^{n+1}\|^2 \right)
+ \beta_c^2\left( \|\eta_S^n -\eta_S^{n+1}\|^2 + \|\eta_S^{n+1}\|^2 \right)   
\bigg)\|\bm g\|^2_{L^{\infty}}
\end{gather*}
Using assumptions on the true solutions yields

\begin{gather*}
\frac{1}{2\Delta t}\left(\|\phi_{h,\bu}^{n+1}\|^2 + \|\phi_{h,T}^{n+1}\|^2 + \|\phi_{h,S}^{n+1}\|^2 \right)
+ \frac{D_a^{-1}}{2}\|\phi_{h,\bu}^{n+1}\|^2 +\frac{\mu_1}{4}\|\phi_{h,\bu}^{n+1}\|^2
+
\frac{\mu_2}{4}\|\phi_{h,T}^{n+1}\|^2 
+ 
\frac{\mu_3}{4}\|\phi_{h,S}^{n+1}\|^2  \\
+
\frac{\nu}{4}\|\nabla\phi_{h,\bu}^{n+1}\|^2 
+\frac{\kappa}{4}\|\nabla\phi_{h,T}^{n+1}\|^2 
+\frac{D_c}{4}\|\nabla\phi_{h,S}^{n+1}\|^2
\\
\leq C\bigg[ D_a+ 
\mu_1^{-1}(\beta_T^2 + \beta_c^2) + C_{PF}^2(\kappa^{-1} +  D_c^{-1})
 \bigg] h^{2k+2}
+ C\bigg[
(\nu^{-1} + \kappa^{-1} + D_c^{-1} ) h^{2k+2}
+ \nu^{-1} h^{4k+4}
 \bigg]\\
+ C  \left[D_a^{-1} +\mu_1 +\mu_2 +\mu_3 \right]h^{2k+2}
+ C\Delta t^2\left[D_a + C_{PF}^2(\kappa^{-1} + D_c^{-1}) + D_a(\beta_T^2 +\beta_c^2) \right]\\
+ \frac{1}{2\Delta t}\left[\|\phi_{h,\bu}^{n}\|^2 + \|\phi_{h,T}^{n}\|^2 + \|\phi_{h,S}^{n}\|^2  \right].
\end{gather*}
First apply the Poincar{\'{e}}-Friedrich Inequality on the left hand side, and denote
\begin{align*}
\lambda_1:={D_a^{-1}} +\frac{\mu_1}{2} +\frac{\nu C_{PF}^{-2}}{2},\, \,  \lambda_2:=\frac{\mu_2}{2} +\frac{\kappa \,  C_{PF}^{-2}}{2},\,\,
\lambda_3:=\frac{\mu_3}{2} +\frac{D_c\, C_{PF}^{-2}}{2}.
\end{align*}
Then multiplying by $2\Delta t$ 
yields
\begin{gather*}
\left(1+\lambda_1\Delta t\right)\|\phi_{h,\bu}^{n+1}\|^2 
+
\left(1+\lambda_2\Delta t\right)\|\phi_{h,T}^{n+1}\|^2
+
\left(1+\lambda_3\Delta t\right)\|\phi_{h,S}^{n+1}\|^2
\\
\leq 
 C\Delta t\bigg[\left( D_a+ 
\mu_1^{-1}(\beta_T^2 + \beta_c^2) + C_{PF}^2(\kappa^{-1} + D_c^{-1})\right) h^{2k+2} + (\nu^{-1} + \kappa^{-1}+ D_c^{-1})h^{2k+2} +\nu^{-1}h^{4k+4}\\
+(D_a^{-1} + \mu_1 + \mu_2\ +\mu_3 )h^{2k+2}
+ \Delta t^2 \bigg[D_a + C_{PF}^2(\kappa^{-1} + D_c^{-1}) + D_a(\beta_T^2 +\beta_c^2)\bigg]\\
+ \|\phi_{h,\bu}^{n}\|^2 + \|\phi_{h,T}^{n}\|^2 + \|\phi_{h,S}^{n}\|^2 .
\end{gather*}
Denoting 
\begin{align*}
K:= \left( D_a+ 
\mu_1^{-1}(\beta_T^2 + \beta_c^2) + C_{PF}^2(\kappa^{-1} + D_c^{-1})\right) h^{2k+2}  + (\nu^{-1} + \kappa^{-1}+ D_c^{-1})h^{2k+2} +\nu^{-1}h^{4k+4}\\
+(D_a^{-1} + \mu_1 + \mu_2\ +\mu_3 )h^{2k+2}
+ \Delta t^2 (D_a + C_{PF}^2(\kappa^{-1} + D_c^{-1}) + D_a(\beta_T^2 +\beta_c^2)\| g\|_{L^{\infty}}^2),
\end{align*}
and taking $ \lambda:=\min\{ \lambda_1, \lambda_2, \lambda_3\}$ yields
\begin{align*}
\left(1+\lambda\Delta t\right)\bigg[\|\phi_{h,\bu}^{n+1}\|^2 + \|\phi_{h,T}^{n+1}\|^2 +\|\phi_{h,S}^{n+1}\|^2
\bigg] \leq C\Delta t K + \bigg[\|\phi_{h,\bu}^{n}\|^2 + \|\phi_{h,T}^{n}\|^2 + \|\phi_{h,S}^{n}\|^2 \bigg].
\end{align*}
Using induction together with the triangle inequality finishes the proof.
\end{proof}
\begin{theorem}[Long time $L^2$-accuracy with $\mu_1>$ and $ \mu_2 = \mu_3= 0$]
Let $\bu\in L^{\infty}(0, \infty; H^{k+1}(\Omega))$, $p\in L^{\infty}(0, \infty; H^{k}(\Omega))$ and $T, S \in L^{\infty}(0, \infty; H^{k+1}(\Omega))$ be true solutions of (\ref{bous}). Assume that $\Delta t$ is sufficiently small such that it holds
\begin{align*}
 \min\left\{{C\nu^{-1}\left( \|\bta_{\bu}^{n+1}\|^2_{L^{\infty}} + \|\bu^{n+1}\|^2_{L^{\infty}}\right)},\,\, {C D_a\beta_T^2 \| g\|_{L^{\infty}}^2}, \,\, {C D_a \beta_c^2 \| g\|_{L^{\infty}}^2} \right \}< \Delta t.
\end{align*}
In addition, we assume that nudging parameters satisfy the following
\begin{align*}
\max\{1, \, \,\nu^{-1}\left( \|\bta_{\bu}^{n+1}\|^2_{L^{\infty}} + \|\bu^{n+1}\|^2_{L^{\infty}}\right), \, \,\kappa^{-1}\left( \|\bta_{T}^{n+1}\|^2_{L^{\infty}} + \|T^{n+1}\|^2_{L^{\infty}}\right),\, \, D_c^{-1}\left( \|\bta_{S}^{n+1}\|^2_{L^{\infty}} + \|S^{n+1}\|^2_{L^{\infty}}\right)   \}\\
\leq \mu_1\leq \frac{\nu}{CH^2}.
\end{align*}
Denote 
\begin{align*}
K:=  ( D_a +\mu_1^{-1}(\beta_T^2 + \beta_c^2)+ C_{PF}^2(\kappa^{-1} + D_c^{-1}) ) h^{2k+2} + (\nu^{-1} + \kappa^{-1}+ D_c^{-1} + D_a^{-1} +\mu_1)\,h^{2k+2} +\nu^{-1}h^{4k+4}\\
+ \Delta t^2 (D_a (1 + \beta_T^2 +\beta_c^2) + C_{PF}^2(\kappa^{-1} + D_c^{-1})),
\end{align*}
$ \lambda:=\min\{ \lambda_1, \lambda_2, \lambda_3\}$ with
\begin{align*}
\lambda_1:={D_a^{-1}} +\frac{\mu_1}{2} +\frac{\nu C_{PF}^{-2}}{2},\, \,  \lambda_2:=\frac{\kappa \,  C_{PF}^{-2}}{2},\,\,
\lambda_3:=\frac{D_c\, C_{PF}^{-2}}{2}.
\end{align*}
Then, for any time level $t^{n+1}, \,\, n=0,1,...$, the errors between the true solutions and solutions of Algorithm~\ref{algbe} satisfies the bound
\begin{align*}
\|E_{\bu}^{n+1}\|^2 + \|E_{T}^{n+1}\|^2 +\|E_{h,S}^{n+1}\|^2
\leq C \lambda^{-1} K + \left(1+\lambda\Delta t\right)^{-(n+1)}\left(\|\phi_{h,\bu}^{0}\|^2 + \|\phi_{h,T}^{0}\|^2 + \|\phi_{h,S}^{0}\|^2 \right).
\end{align*}
\end{theorem}
\begin{proof}
Take $\mu_2 = \mu_3=0$ in Algorithm~\ref{algbe}, and  proceed as in the proof of Theorem 3.1. The error equations are the same with \eqref{vel1}, \eqref{temp1} and \eqref{con1}. Moreover,  all right hand side terms are bounded identically, except the term which is estimated below as follows
\begin{align*}
&\big((\beta_T E_T^{n+1}  + \beta_c E_S^{n+1} \big)\bm g, \phi_{h,\bu}^{n+1})
 = \left(\left(\beta_T \eta_T^{n+1} + \beta_c \eta_S^{n+1} \right)\bm g, \phi_{h,\bu}^{n+1}\right) - \left((\beta_T \phi_{h,T}^{n+1} + \beta_c \phi_{h,S}^{n+1} \right)\bm g, \phi_{h,\bu}^{n+1})\\
&\leq \left(\beta_T \|\eta_T^{n+1}\| + \beta_c \|\eta_S^{n+1}\| \right)\|\bm g\|_{L^{\infty}}\|  \phi_{h,\bu}^{n+1}\| + \left(\beta_T C_{PF} \|\nabla\phi_{h,T}^{n+1}\| + \beta_c C_{PF}\|\nabla\phi_{h,S}^{n+1}\| \right)\|\bm g\|_{L^{\infty}}\| \phi_{h,\bu}^{n+1}\|\\
& 
\leq C \mu_1^{-1}\left(\beta_T^2 \|\eta_T^{n+1}\|^2 + \beta_c^2 \|\eta_S^{n+1}\|^2 \right)\|\bm g\|^2_{L^{\infty}} +\frac{\mu_1}{8}\| \phi_{h,\bu}^{n+1}\|^2\\
&\, \, \, \, \, \, +C \, C_{PF}^2\left(\beta_T^2 \|\nabla\phi_{h,T}^{n+1}\|^2 + \beta_c^2 \|\nabla\phi_{h,S}^{n+1}\|^2 \right)\|\bm g\|_{L^{\infty}}^2 + \frac{D_a^{-1}}{12}\| \phi_{h,\bu}^{n+1}\|^2.
\end{align*} 
Replace this estimate with the right hand side of \eqref{farkterim}. Then following exactly same lines of the proof of Theorem 3.1 by denoting
$ \lambda:=\min\{ \lambda_1, \lambda_2, \lambda_3\},$ 
where
\begin{align*}
\lambda_1:={D_a^{-1}} +\frac{\mu_1}{2} +\frac{\nu C_{PF}^{-2}}{2},\, \,  \lambda_2:=\frac{\kappa \,  C_{PF}^{-2}}{2},\,\,
\lambda_3:=\frac{D_c\, C_{PF}^{-2}}{2},
\end{align*}
finishes the proof.

\end{proof}

\section{Numerical Experiments}
This section aims to present some numerical tests in order to see the effectiveness of the algorithm given with (\ref{disc1})-(\ref{disc3}). In the first test, a known analytical solution of the system is used to measure the error. Secondly, a more practical test, namely double-diffusive flow in a cavity is considered. We choose the Scott-Vogelious finite element pair along with a barycenter refined mesh for velocity-pressure couple. We also pick second order piecewise polynomials for temperature and concentration spaces. Furthermore, we select piecewise constants for coarse mesh spaces for all variables. The interpolation operators mentioned in the algorithm are considered to be $L^2$ projection operators onto coarse meshes which satisfy (\ref{interp1})-(\ref{interp2}). All simulations are carried out with public license finite element software FreeFem++ \cite{hec}. We would like to point out here that due to the virtue of the method proposed here, all the initial conditions for all variables are taken to be zero
\subsection{Rates of Convergence}
To verify the error rates predicted by the theory, we consider the particular analytical solution given by:
\begin{align}\label{truesol}
\textbf{u}=\left(%
\begin{array}{c}
\cos(y) \\
\sin(x)%
\end{array}%
\right)e^t,\quad
p=(x-y)(1+t),\quad
T=\sin(x+y)e^{1-t}, \quad
S=\cos(x+y)e^{1-t}.
\end{align}
with the parameters $\nu=D_c= \kappa=\beta_T=\beta_S=1$, $Da=\infty$ and the right hand side functions $F,G$ and $\Phi$ are chosen such that (\ref{truesol}) satisfies (\ref{bous}).
In general, we did the calculations for two different cases. We first run the code for all nudging parameters to be greater than zero. Then we have a rerun for the case $\mu_1>0, \mu_2=\mu_3=0$ in order to see whether there is any difference.
\subsubsection{CASE I: $\mu_1>0, \mu_2>0, \mu_3>0$}
In this case, we take $\mu_1=100, \mu_2=100, \mu_3=100$ so the nudging applies on all equations of the system. We first fix the time step to $\Delta t =10^{-3}$ for the time interval $[0,1]$. We measure the spatial error with $L^2$ norm. The results are given in Table \ref{table:tab1}.
\begin{table}[h!]
	\begin{center}
		
		\begin{tabular}{|c|c|c|c|c|c|c|}
			\hline
			$h$ & $\|\textbf{u}-\textbf{u}_h\|$ & Rate &$\|T-T_h\|$ & Rate& $\|S-S_h\| $ & Rate\\
			\hline
			1&0.0149 &--  &0.0023  &-- &0.0024&--  \\
			\hline
			1/2&0.0018 &3.05  &0.0012 &0.93&0.0019&0.33  \\
			\hline
			1/4&0.00021 &3.16  &0.00016  &2.90 &0.0002&3.20  \\
			\hline
			1/8&2.7e-5 &2.9  &2.15e-5  &2.92 &2.9e-5& 2.85 \\
			\hline
			1/16&3.3e-6 &3.07  &2.7e-6  &3.00&3.6e-6&3.00   \\
			\hline
			1/32&4.14e-7 &3.04 & 3.40e-7 & 2.99&4.62e-7&2.96   \\
			\hline
		\end{tabular}
		\caption{Spatial errors and rates of convergence for $\mu_1= \mu_2 =\mu_3 =100$.}
		\label{table:tab1}
	\end{center}
\end{table}
As could be seen from the Table \ref{table:tab1}, the rates of convergence are coherent with Theorem \ref{teo1}, since making use of the SV elements we expect an error of order $3$ for all variables. Thus we can conclude that the addition of the nudging term does not deteriorate the error rates and taking random initial values still results with correct numerical results thanks to the method we propose.

Next we fix the mesh width to $h=1/32$ and run the code for different nudging parameters $\mu_1= \mu_2 =\mu_3 =1,10,100,1000$. The results are given in Figure \ref{fig:expcon1}. One could easily discover the exponential convergence in time  from these figures for all variables. Also we can deduce from these figures that as the nudging parameters increase, we observe a faster convergence as expected.
\begin{figure}[H]
	\centerline{\hbox{
			\epsfig{figure=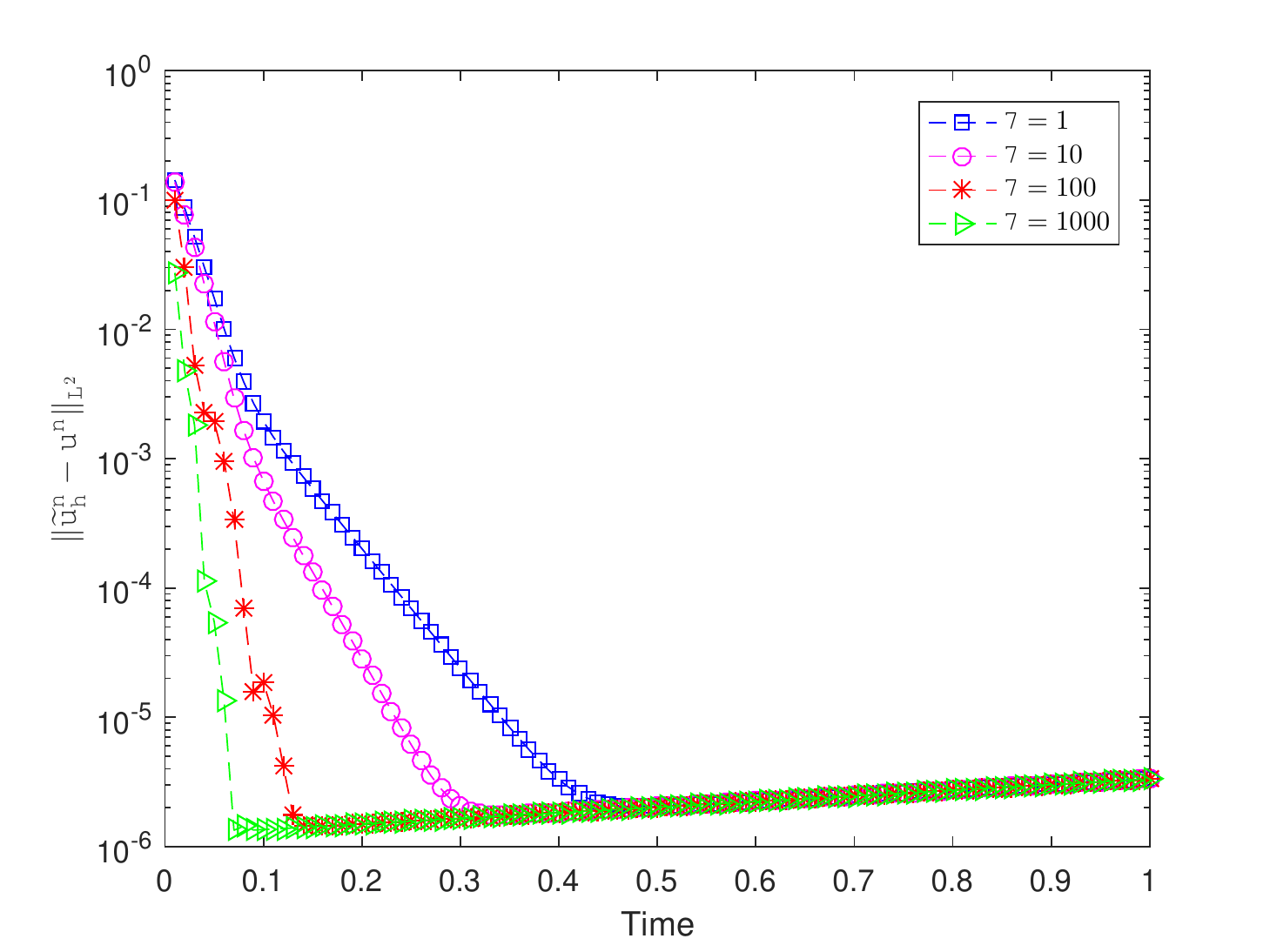,width=0.35\textwidth}
			\epsfig{figure=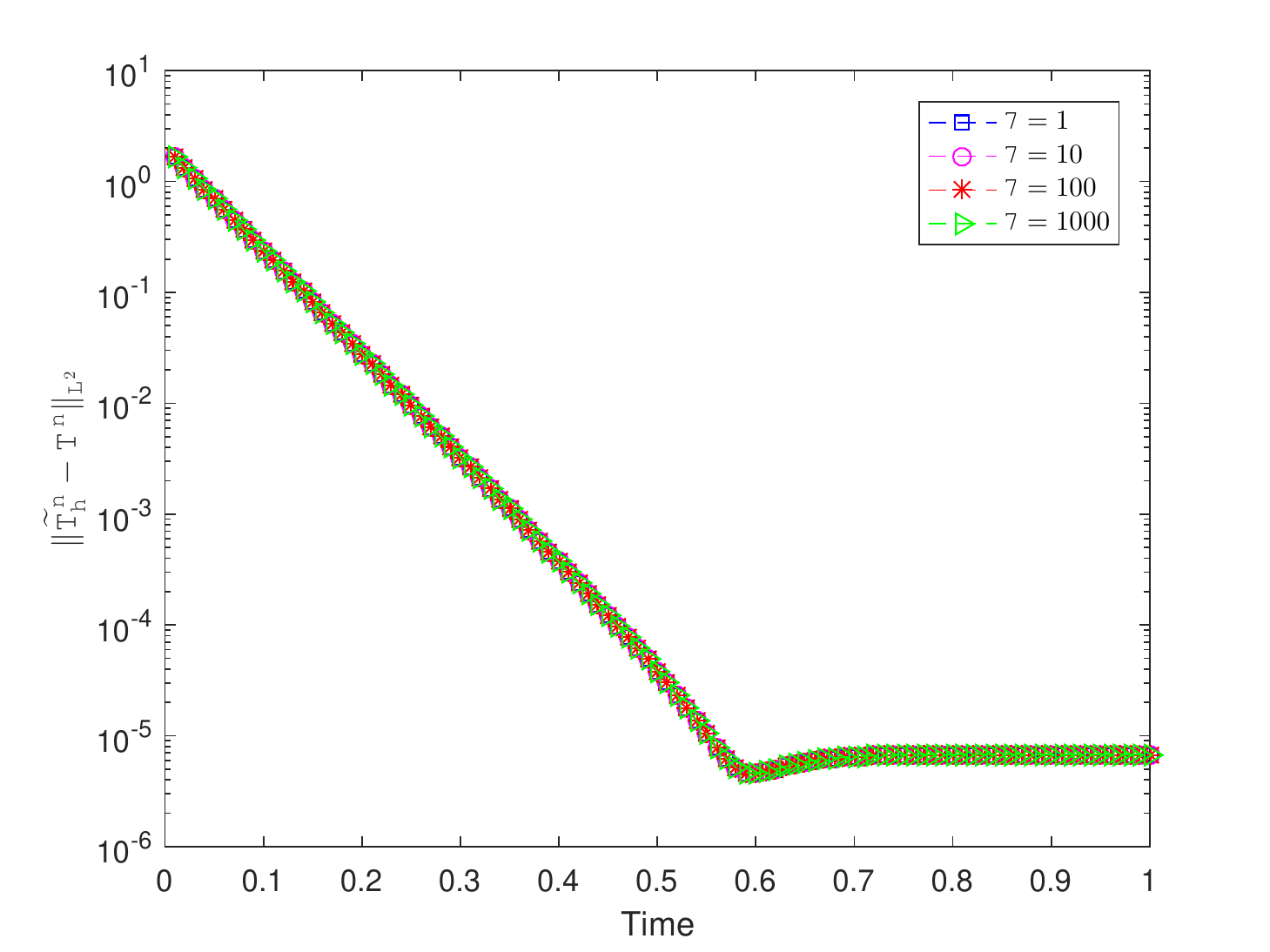,width=0.35\textwidth}		            \epsfig{figure=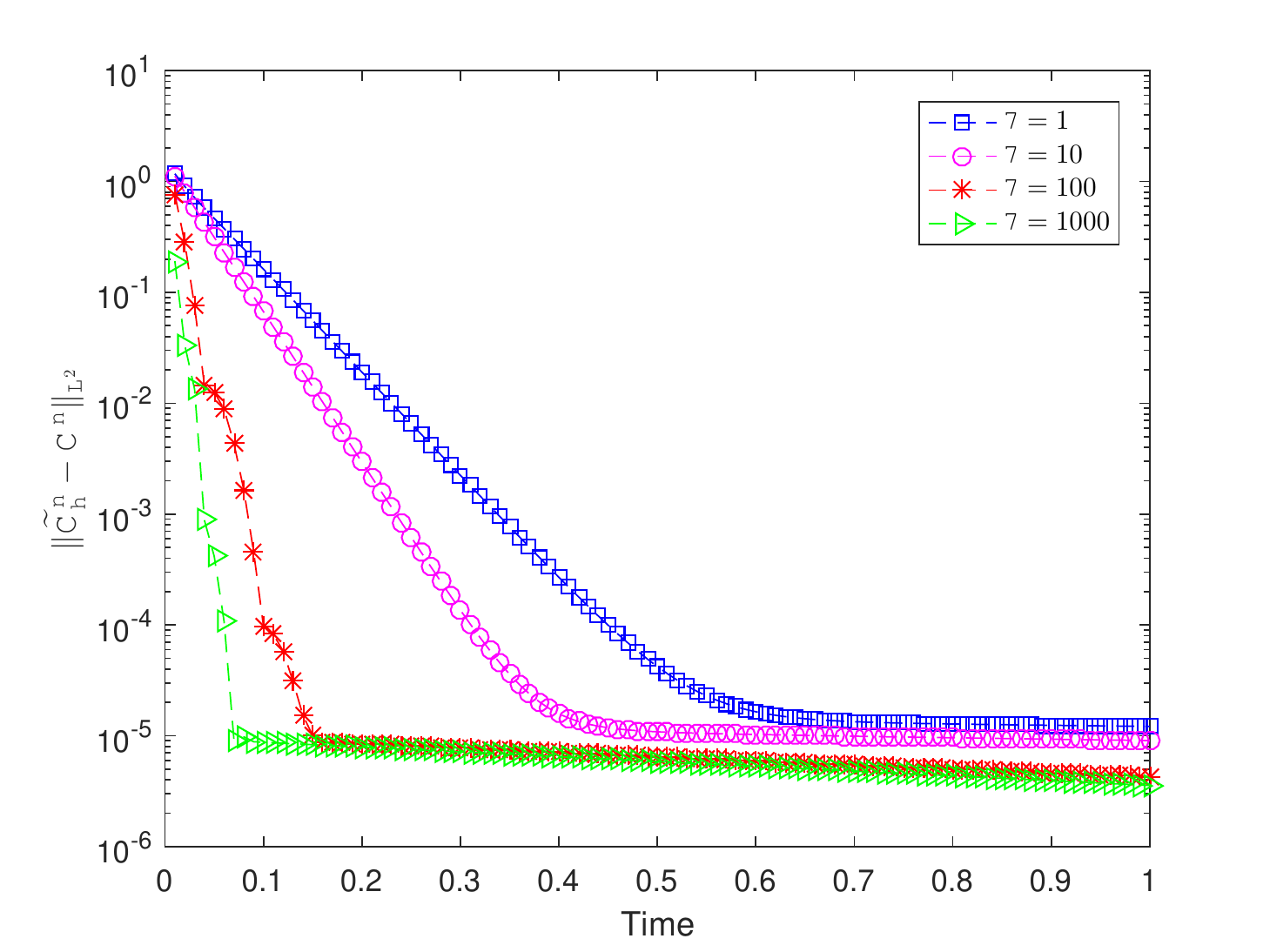,width=0.35\textwidth}
	}}
	\caption{Errors for different  $\mu_1, \mu_2, \mu_3$ for velocity, temperature and concentration respectively. }\label{fig:expcon1}
\end{figure}
\subsubsection{CASE II: $\mu_1>0, \mu_2=0, \mu_3=0$}
This time, we only nudge the velocity equation. This could be considered as, data might be collected for all variables in areal life system. We take $\mu_1=100, \mu_2=0, \mu_3=0$ so the nudging applies only on velocity equation . As in Case I,  the time step is fixed to $\Delta t =10^{-3}$ for the time interval $[0,1]$.. The results are given in Table~\ref{table:tab2}.
\begin{table}[h!]
	\begin{center}	
		\begin{tabular}{|c|c|c|c|c|c|c|}
			\hline
			$h$ & $\|\textbf{u}-\textbf{u}_h\|$ & Rate &$\|T-T_h\|$ & Rate& $\|S-S_h\| $ & Rate\\
			\hline
			1&0.0149 &--  &0.0023  &-- &0.0024&--  \\
			\hline
			1/2&0.0017 &3.13  &0.0012 &0.93&0.0019&0.33  \\
			\hline
			1/4&0.00021 &3.15  &0.00016  &2.90 &0.0002&3.20  \\
			\hline
			1/8&2.6e-5 &3.01  &2.15e-5  &2.92 &2.9e-5& 2.85 \\
			\hline
			1/16&3.3e-6 &2.97  &2.7e-6  &3.00&3.6e-6&3.00   \\
			\hline
			1/32&4.14e-7 &2.99 & 3.46e-7 & 2.99&4.79e-7&2.90   \\
			\hline
		\end{tabular}
		\caption{Spatial errors and rates of convergence for $\mu_1=100, \mu_2 =\mu_3 =0$.}
		\label{table:tab2}
	\end{center}
\end{table}
Again, we obtain the optimal rate of convergence in this case too. The results and errors are almost same with the previous case. Only there is a very slight difference in error rates as $h$ decreases. We should point out here that, taking $\mu_1=100, \mu_2 =\mu_3 =0$ results with a little longer CPU time to obtain these results. 

We now test the errors for fixed spatial step size, $h=1/32$ in this case too. The results are given in Figure~ \ref{fig:expcon2}. We compute the errors for $\mu_1=1,10,100,1000$ and the results still clearly indicate the exponential convergence. The results obtained in Figure \ref{fig:expcon2} are obtained for a little longer CPU time when compared with Case I. However, convergence to true solution is obtained faster for smaller $\mu_1$ values this time. Specifically, altering $\mu_1$ almost does not affect the convergence of temperature and concentration as illustrated in Figure \ref{fig:expcon2}.
\begin{figure}[H]
	\centerline{\hbox{
			\epsfig{figure=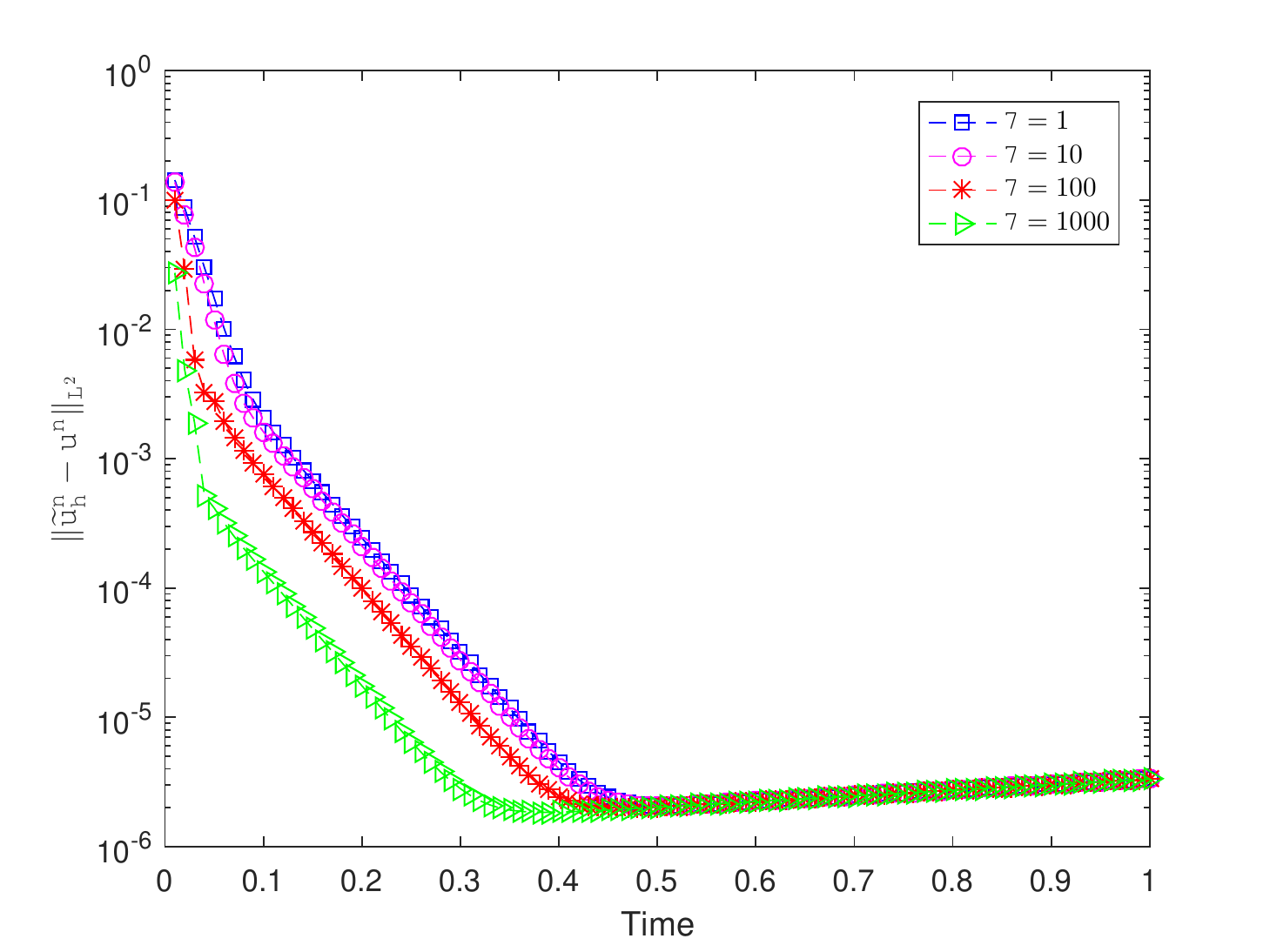,width=0.35\textwidth}
			\epsfig{figure=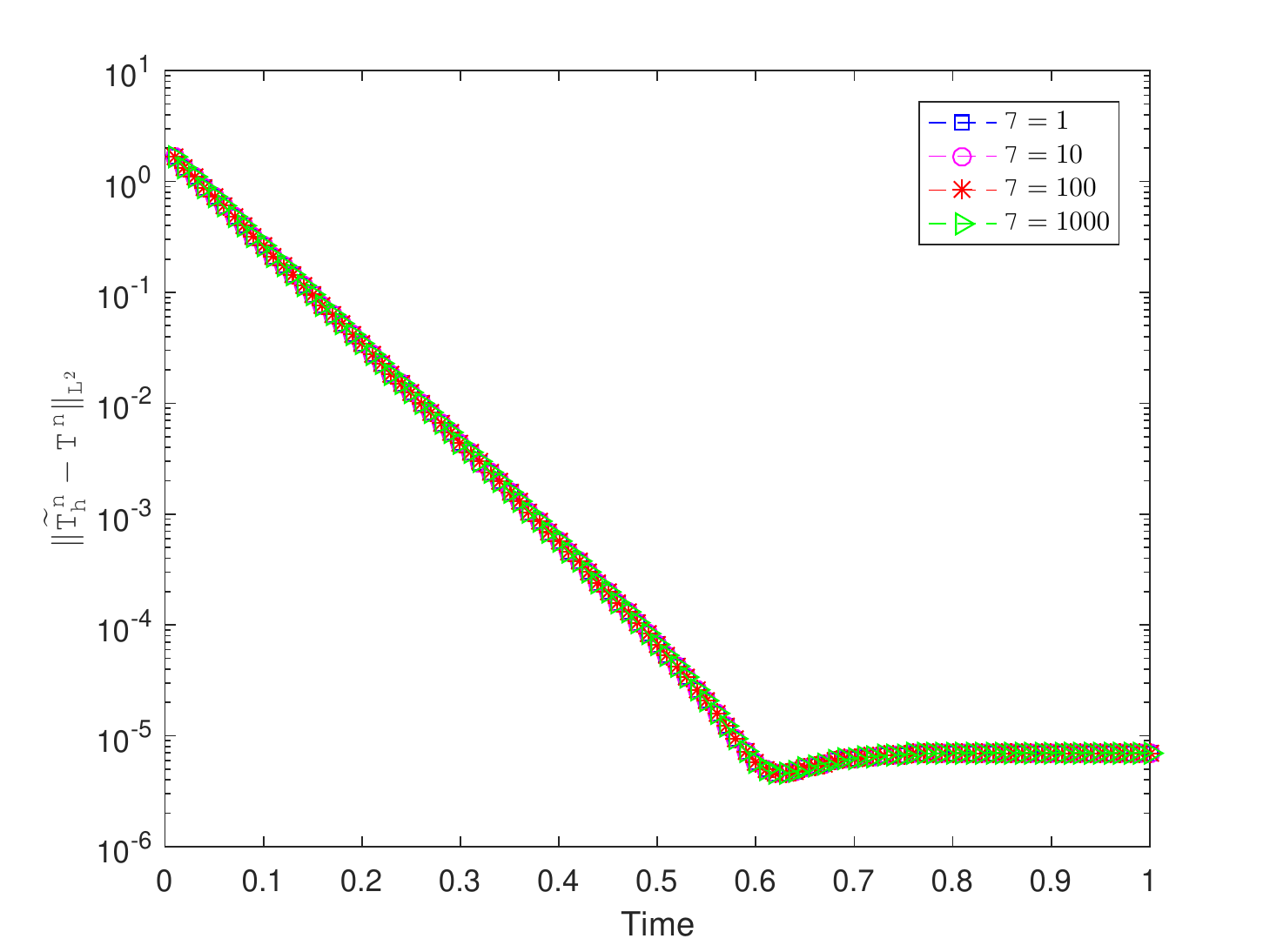,width=0.35\textwidth}
			\epsfig{figure=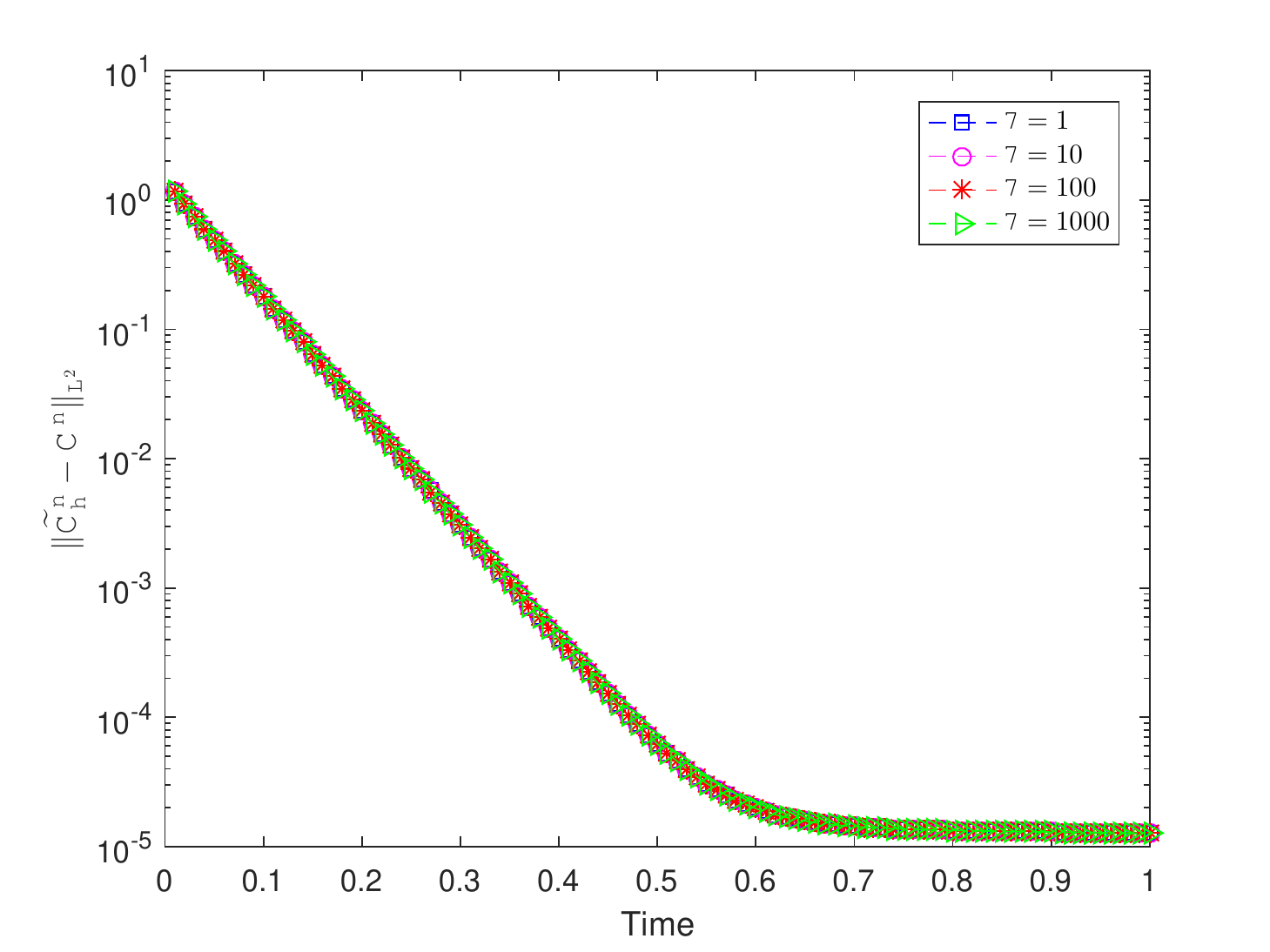,width=0.35\textwidth}
	}}
	\caption{Errors for different  $\mu_1, \mu_2, \mu_3$ for velocity, temperature and concentration respectively. }\label{fig:expcon2}
\end{figure}
\subsection{Flow in a rectangular cavity}
In this second test, we apply the scheme to a more practical test which is adapted from \cite{bizimdouble}. In this test, flow inside a porous cavity initiates by the effect of temperature and concentration differences along with the gravity. The domain is a rectangular cavity of height $2$ and width $1$. No-slip velocity boundary conditions are assumed at entire domain and vertical walls are kept at different temperature and concentration values. Horizontal walls are assumed to be adiabatic and impermeable. The computational domain is illustrated in Figure \ref{fig:newdomain}.

\begin{figure}[H]
	\centerline{\hbox{
			\epsfig{figure=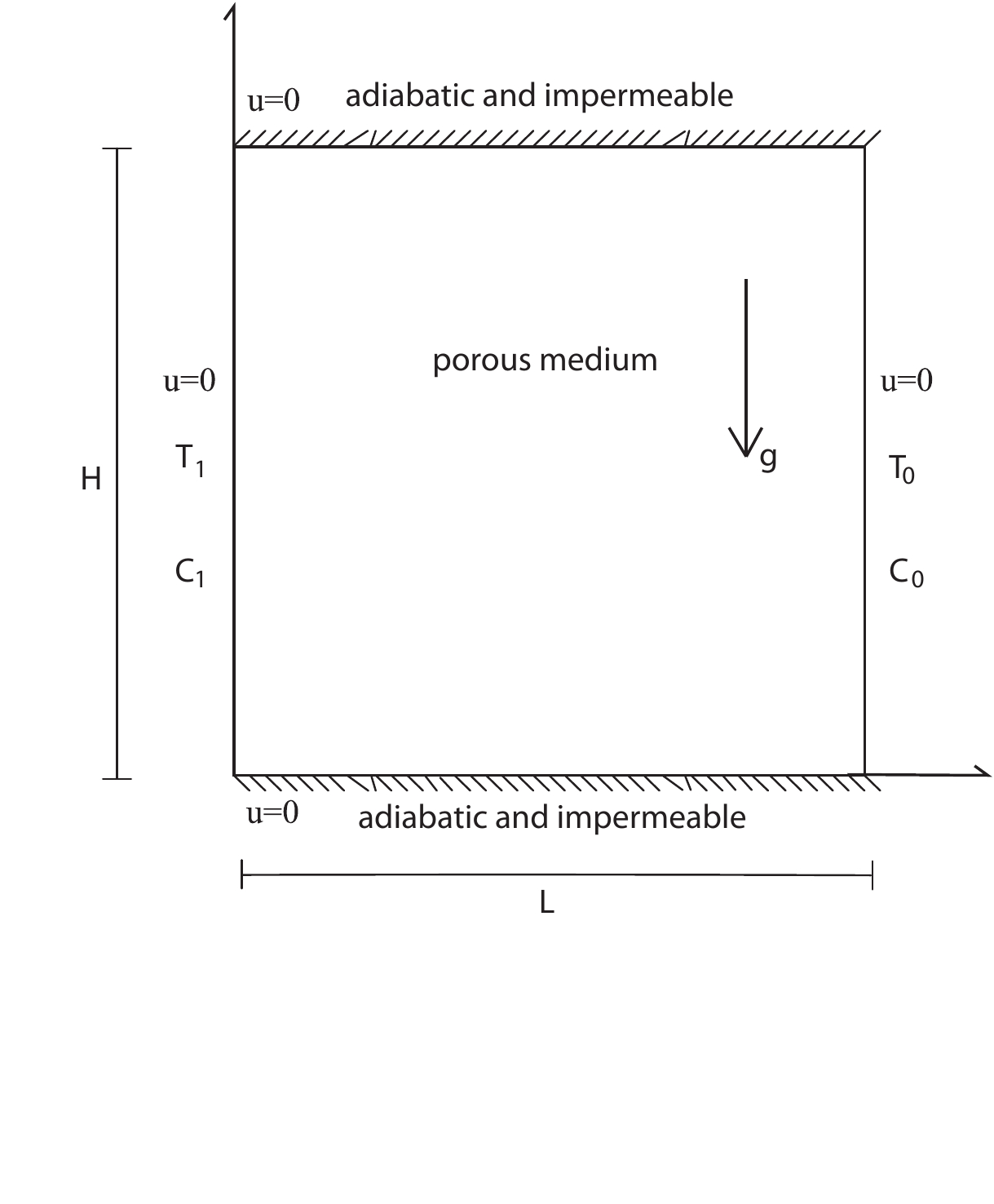,width=0.6\textwidth}
	}}
	\caption{\label{fig:newdomain} Computational domain with its boundary conditions}
\end{figure}
Here we pick $Pr=1, Ra=10^3, Le=2, N=0.8$. Time step size is taken as $\Delta t =0.02$. Inspired from \cite{leocda}, a DNS(Direct Numerical Simulation) first ran for $80$ time units (i.e. $4000$ time iterations) then the scheme takes part for another $80$ time units along with DNS simulation again. Our algorithm is sampled from DNS solutions at the beginning. We compute the convergence of our solution to DNS in time first. We again carried out the computations for  cases $\mu_1>0, \mu_2>0, \mu_3>0$ and $\mu_1>0, \mu_2=0, \mu_3=0$ distinctly.

We make our first run with $\mu_1=1\,\, \mbox{and}\,\, 10, \mu_2=1\,\,\mbox{and}\,\, 10,\,\, \mu_3=1\,\, \mbox{and}\,\, 10$. Errors measured in $L^2$ norm which consist of the difference between our numerical solutions and corresponding DNS solutions are depicted in Figure \ref{fig:dns1}.
\begin{figure}[H]
	\centerline{\hbox{
			\epsfig{figure=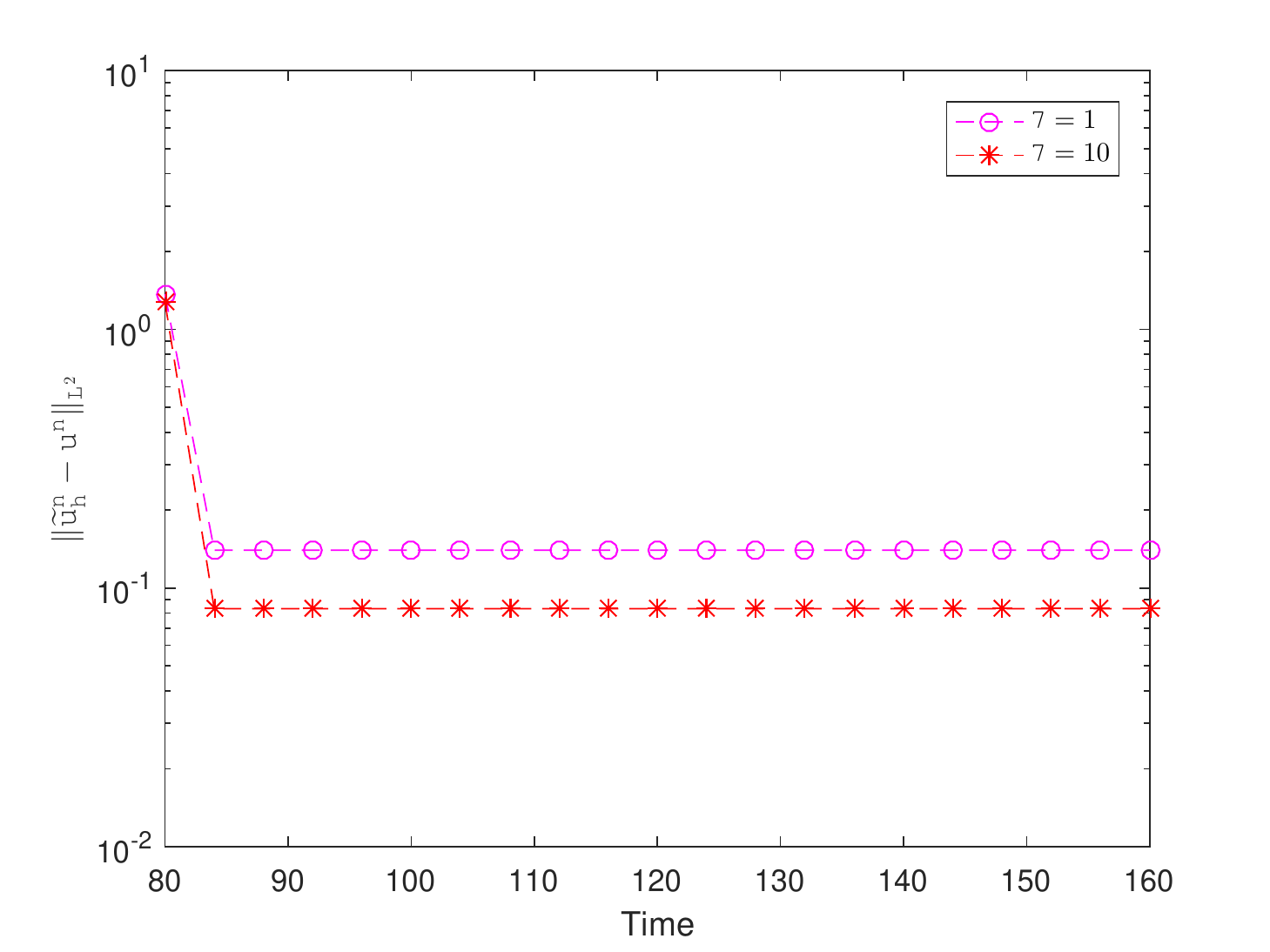,width=0.35\textwidth}
			\epsfig{figure=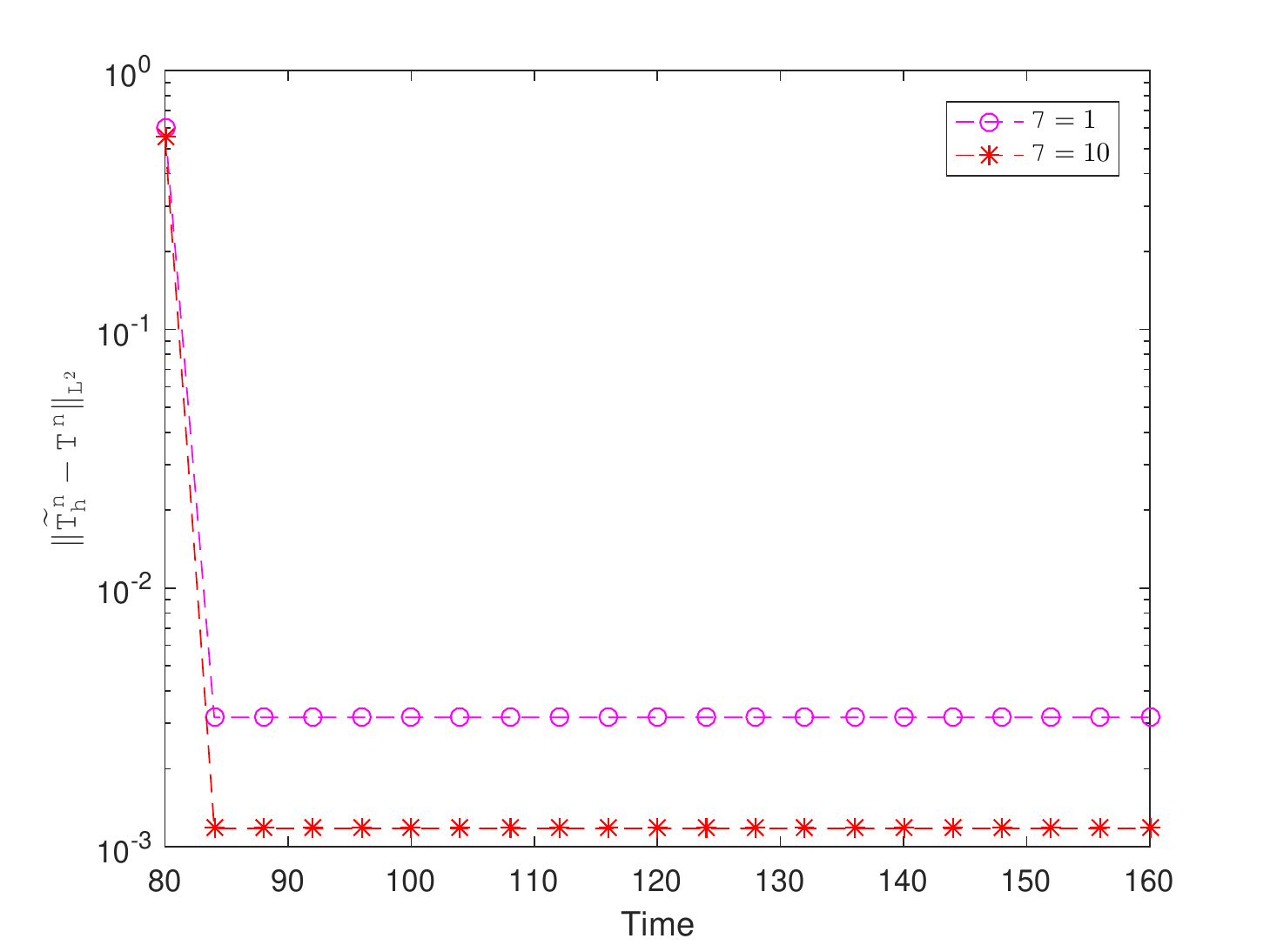,width=0.35\textwidth}
			\epsfig{figure=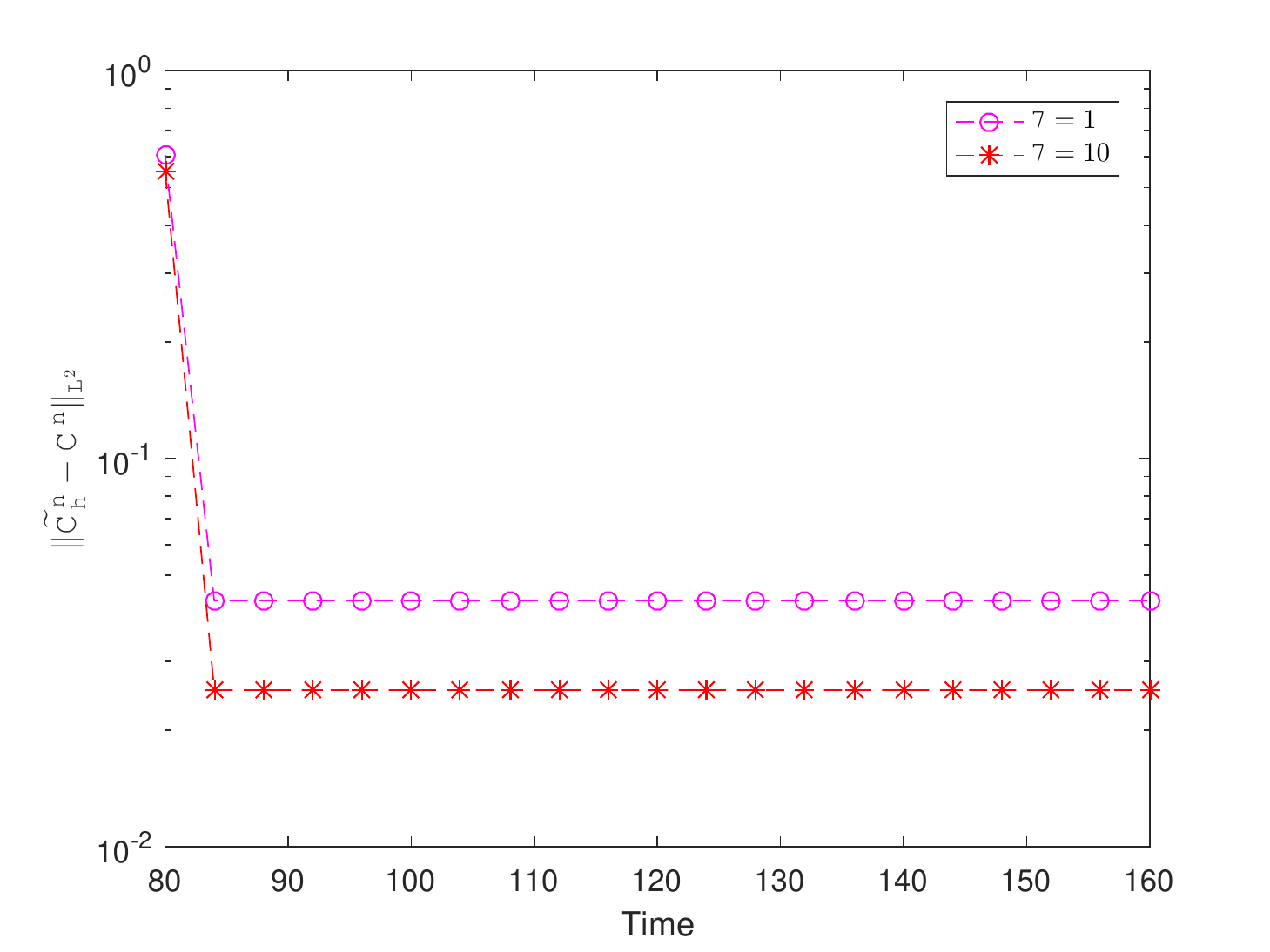,width=0.35\textwidth}
	}}
	\caption{\label{fig:dns1}Convergence of numerical solution to DNS for $\mu_1=1\,\, \mbox{and}\,\, 10, \mu_2=1\,\,\mbox{and}\,\, 10,\,\, \mu_3=1\,\, \mbox{and}\,\, 10$  }
\end{figure}

Then we make another run for the case $\mu_1=1\,\, \mbox{and}\,\, 10, \mu_2=0, \mu_3=0$. The results obtained from second case is given in Figure \ref{fig:dns2}.

\begin{figure}[H]
	\centerline{\hbox{
			\epsfig{figure=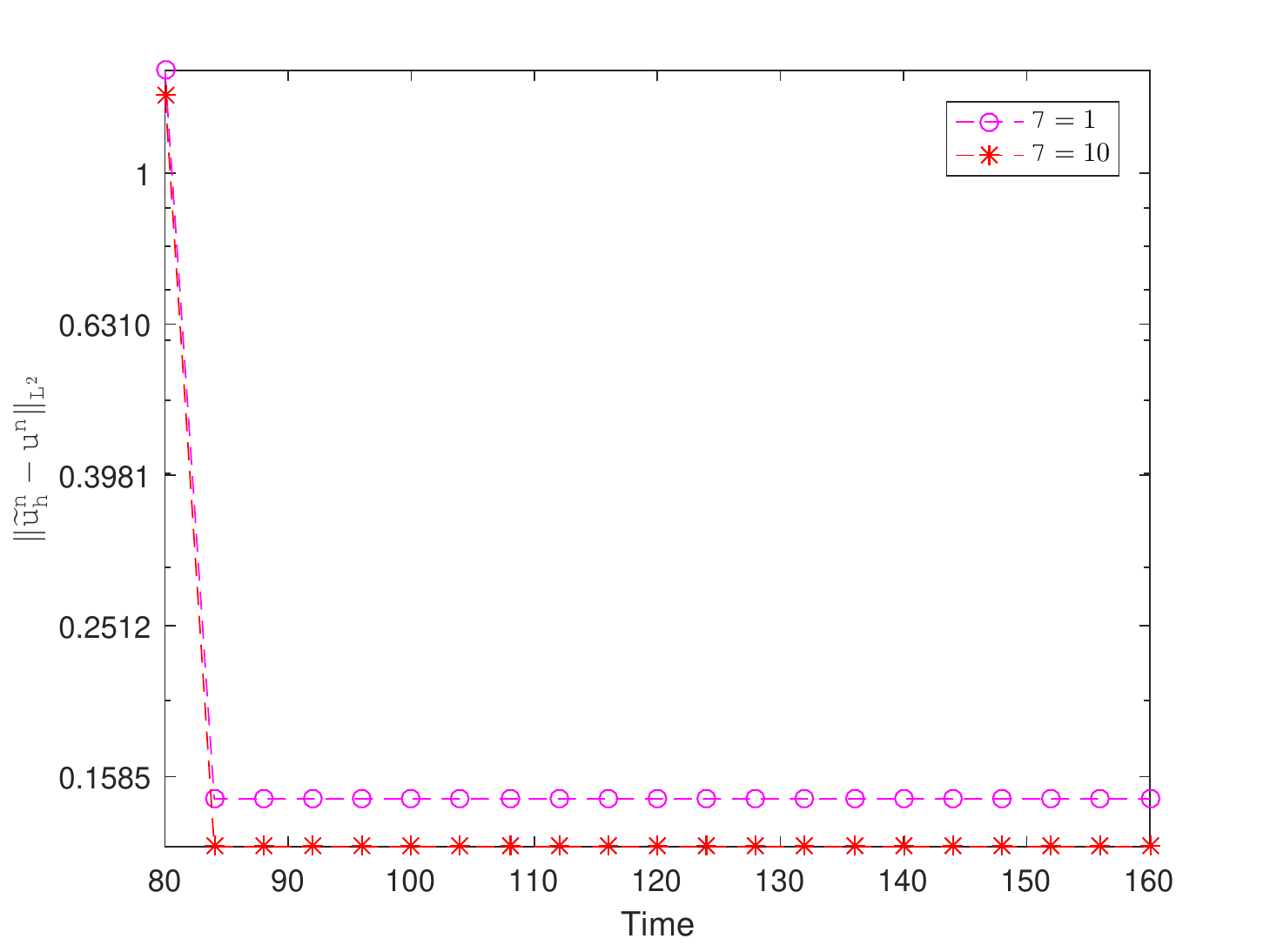,width=0.35\textwidth}
			\epsfig{figure=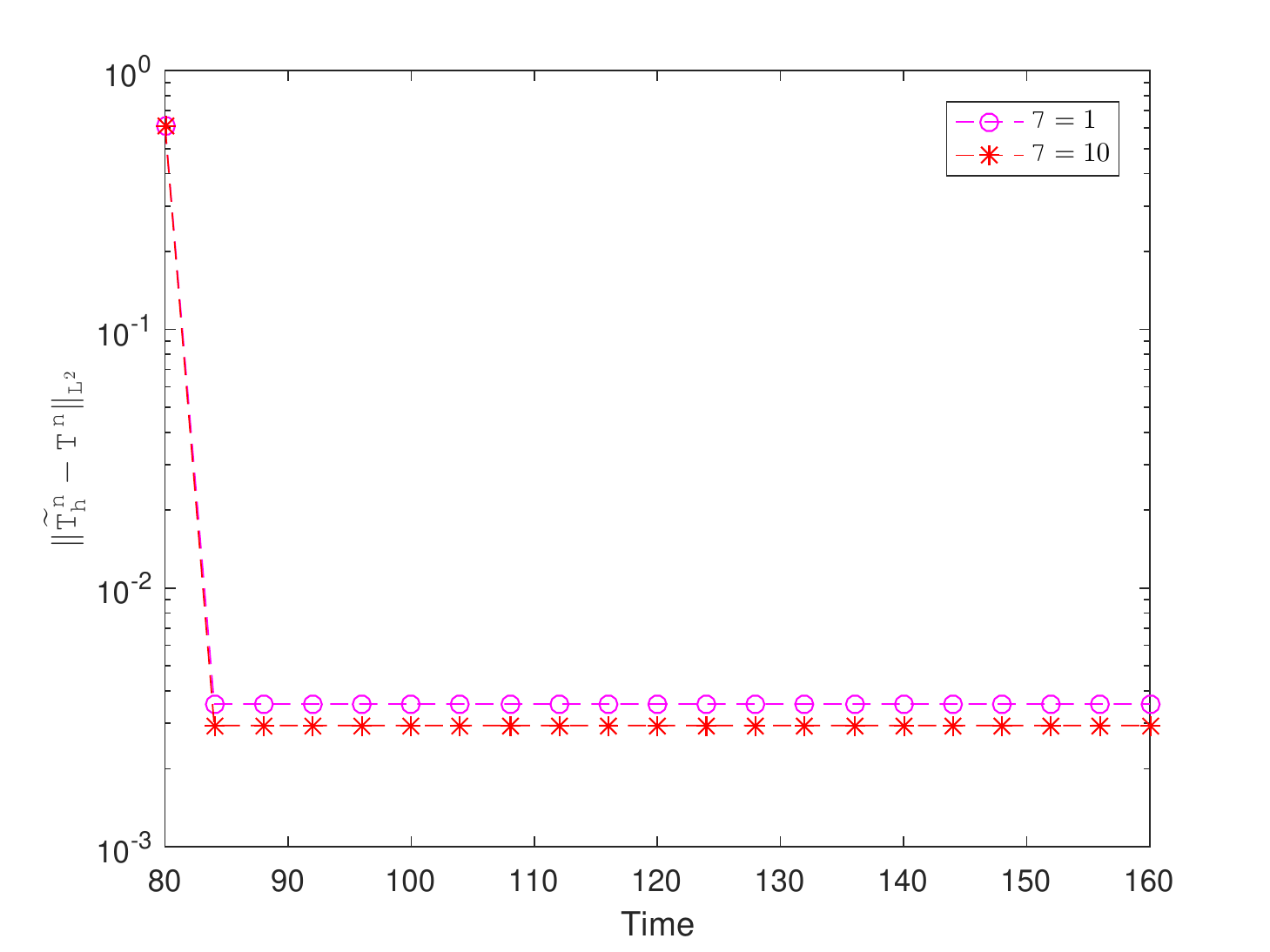,width=0.35\textwidth}
			\epsfig{figure=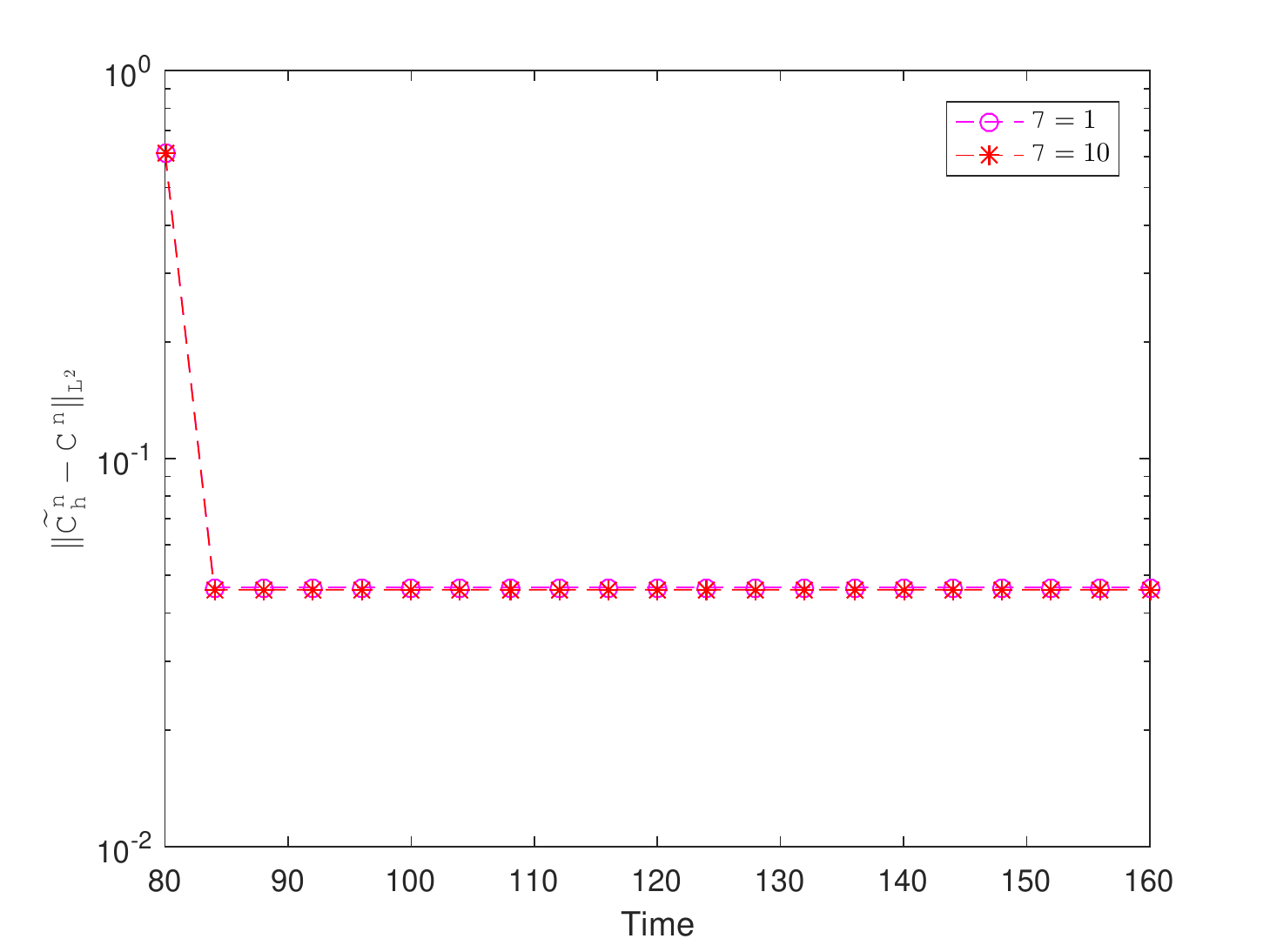,width=0.35\textwidth}
	}}
	\caption{\label{fig:dns2}Convergence of numerical solution to DNS for $\mu_1=1\,\, \mbox{and}\,\, 10, \mu_2=0, \mu_3=0$  }
\end{figure}
As one examines Figures \ref{fig:dns1} and \ref{fig:dns2}, it is easy to deduce that these results are almost identical. Convergence time and error magnitudes are almost same. The only difference to note might be, when $\mu_2=0, \mu_3=0$ the effect of $\mu_1$ is strictly decreased when compared with the case $\mu_2>0, \mu_3>0$ in which the magnitude of $\mu_1, \mu_2, \mu_3$ seems to effect convergence time and error magnitude.

Lastly, we draw some important flow patterns related to double-diffusive convection flow. We give these figures for both DNS and the scheme in order to see whether there is any difference. As in previous case we consider the cases $\mu_1>0, \mu_2>0, \mu_3>0$ and $\mu_1>0, \mu_2=0, \mu_3=0$. Figures \ref{fig:cav1} and \ref{fig:cav2} shows the results of the simulation for $\mu_1=\mu_2=\mu_3=1$ and $\mu_1=\mu_2=\mu_3=10$ respectively.

\begin{figure}[H]
	\centerline{\hbox{
			\epsfig{figure=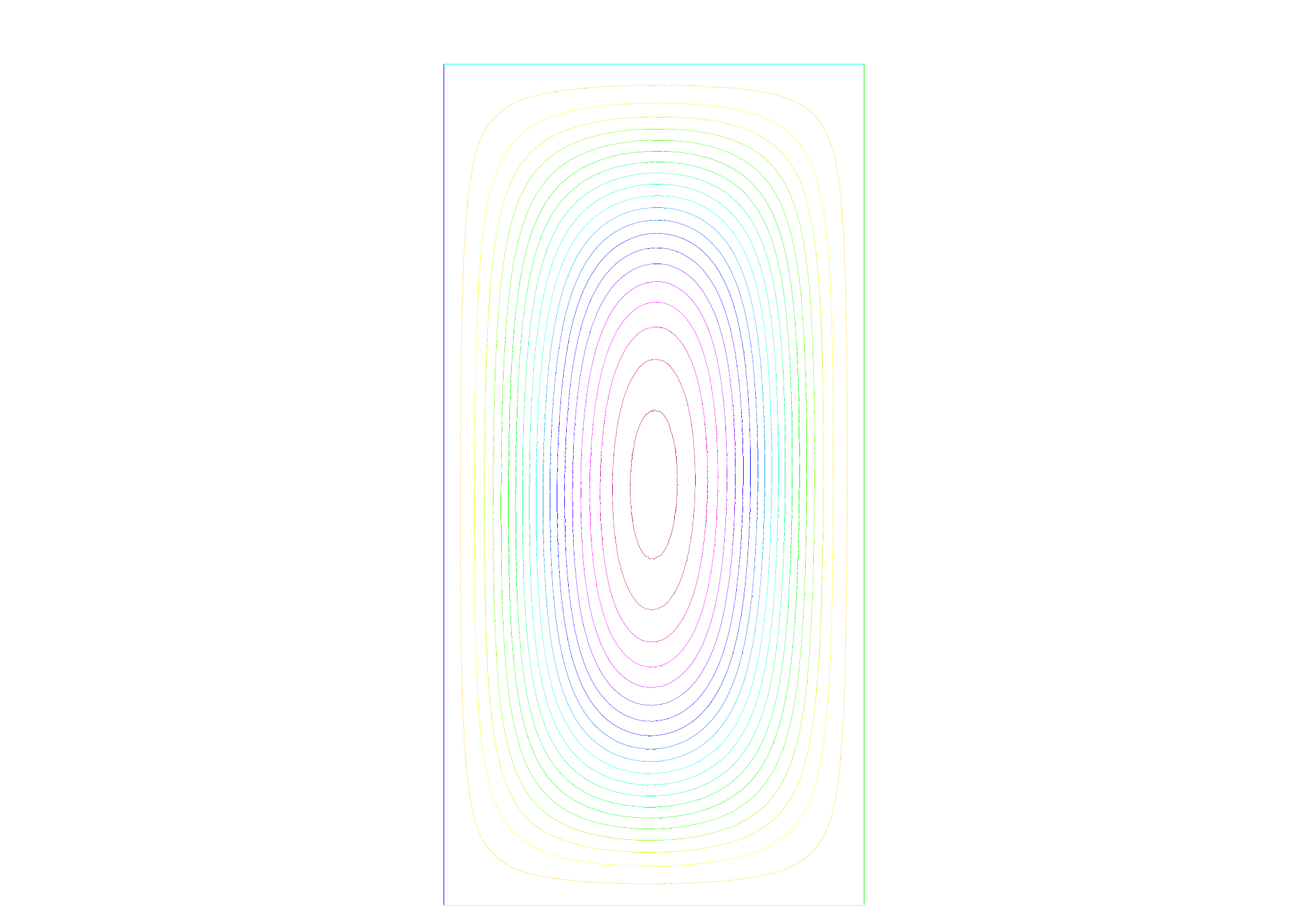,width=0.35\textwidth}
			\epsfig{figure=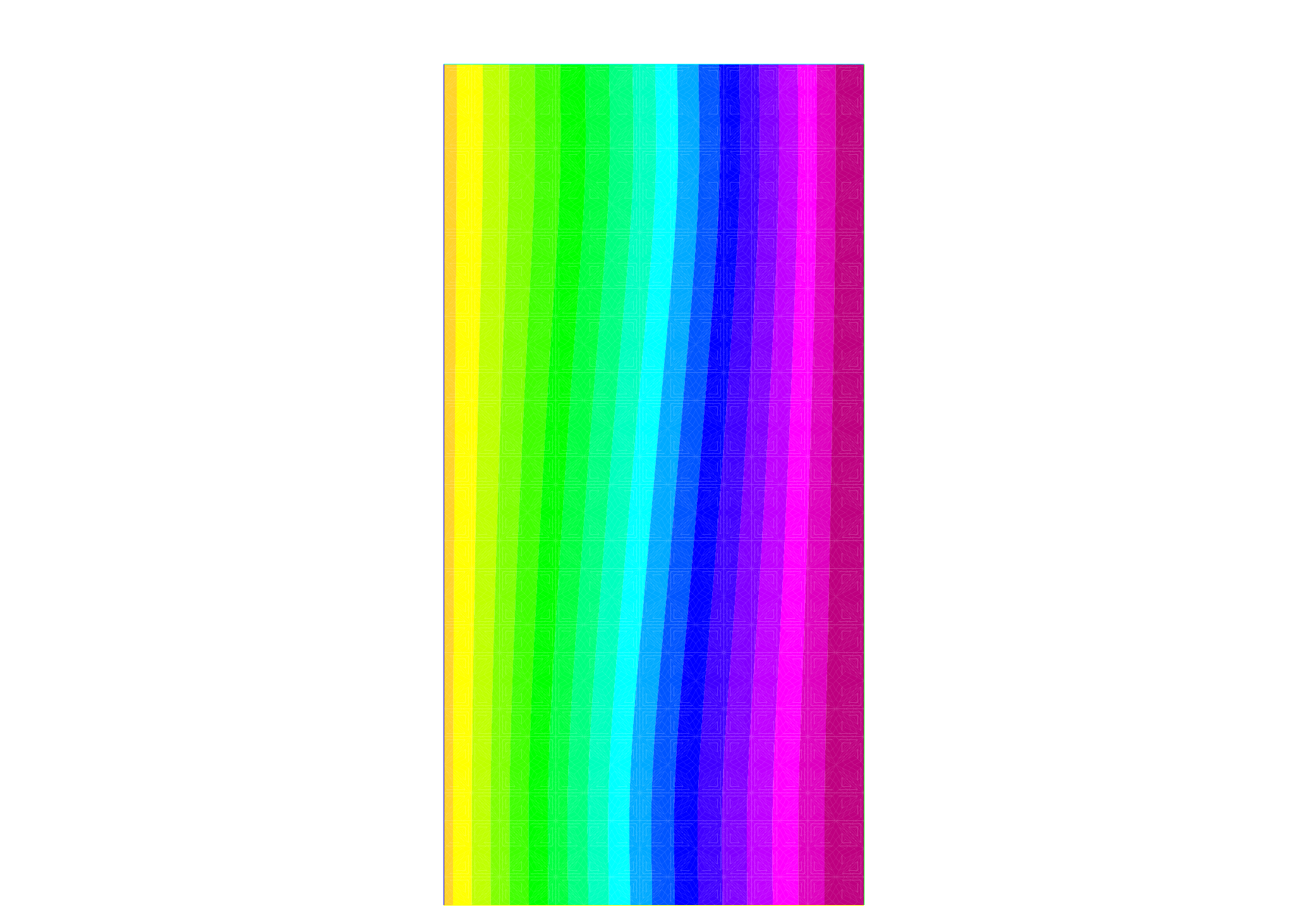,width=0.35\textwidth}
			\epsfig{figure=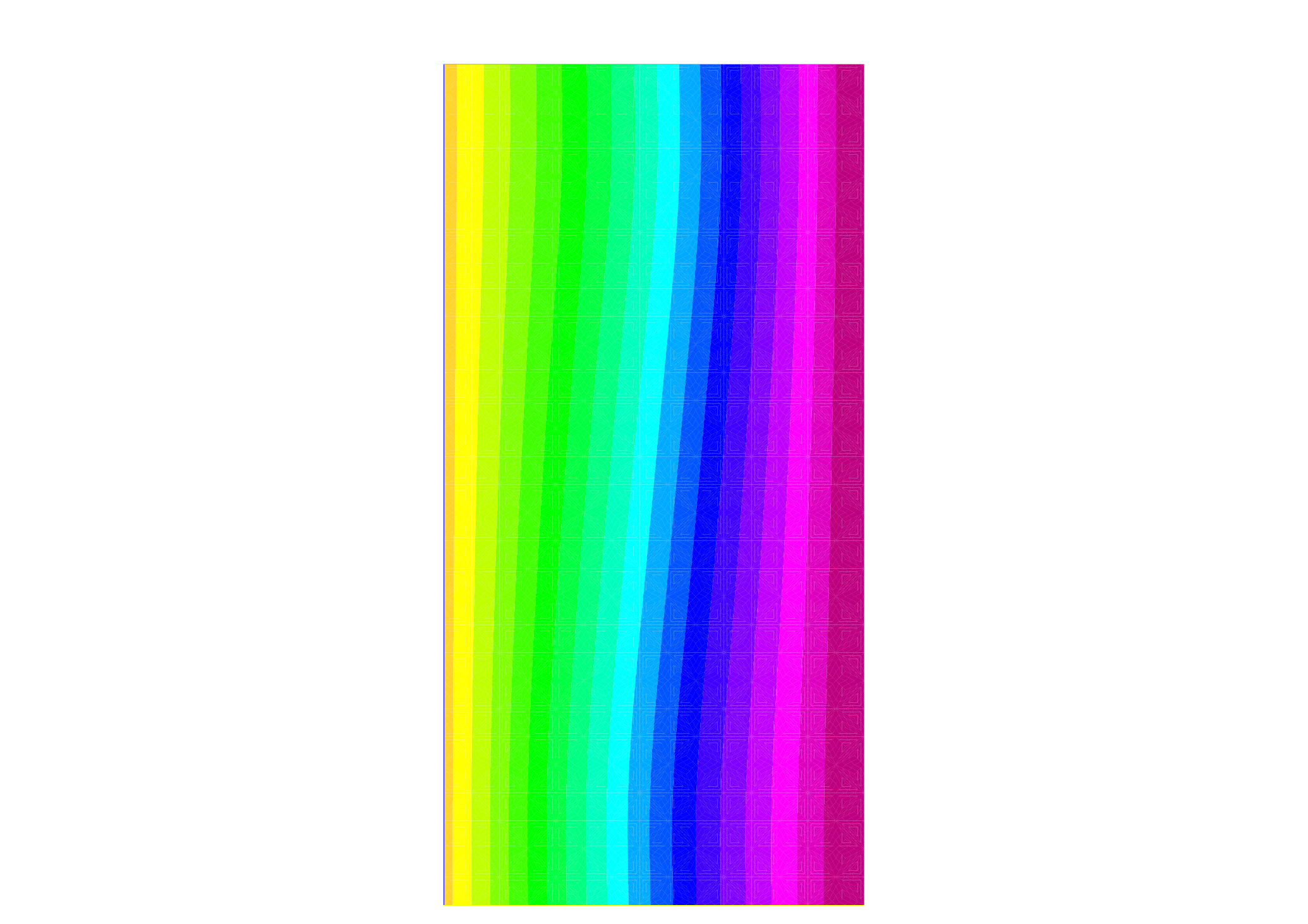,width=0.35\textwidth}
	}}
	\centerline{\hbox{
			\epsfig{figure=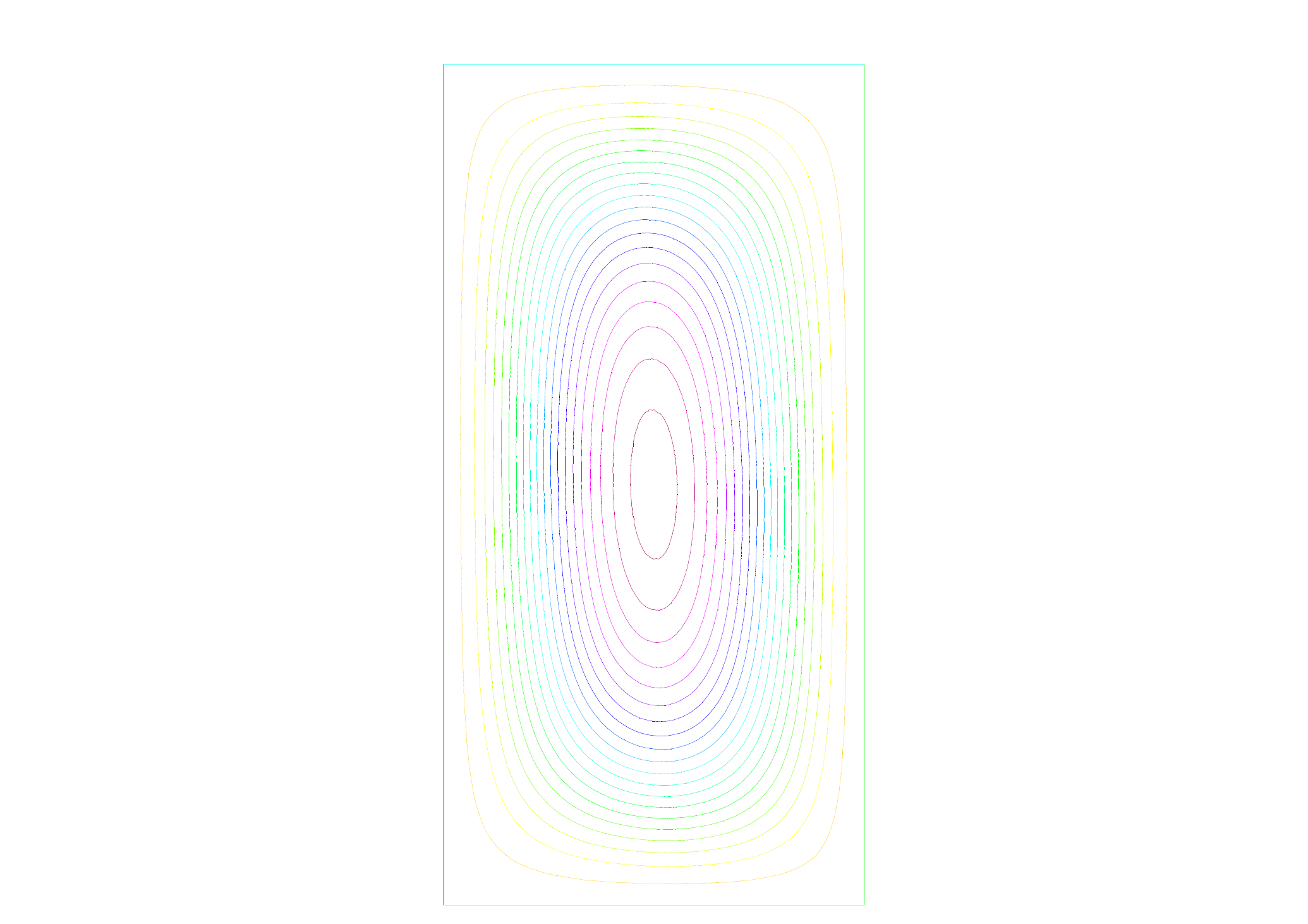,width=0.35\textwidth}
		\epsfig{figure=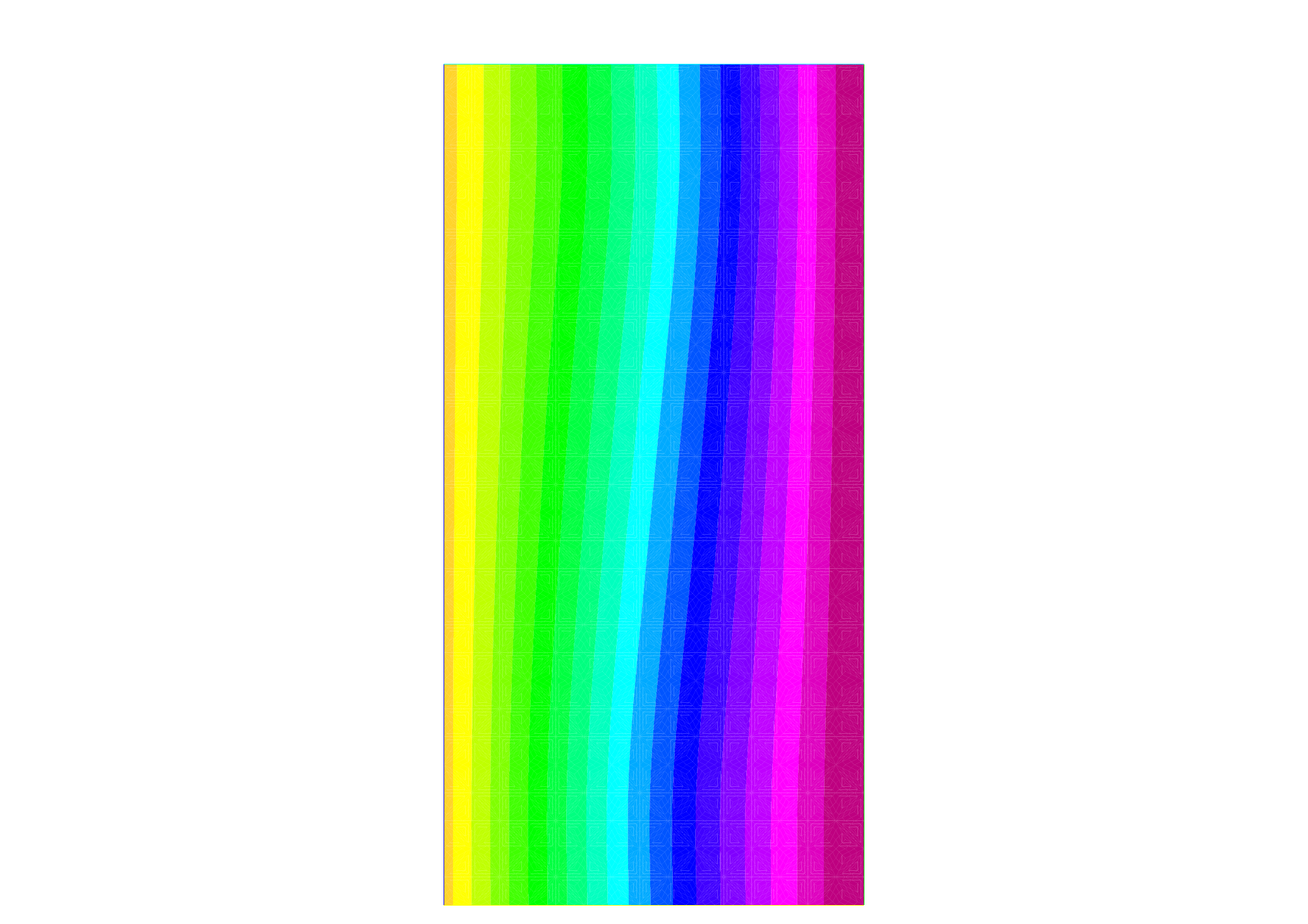,width=0.35\textwidth}
		\epsfig{figure=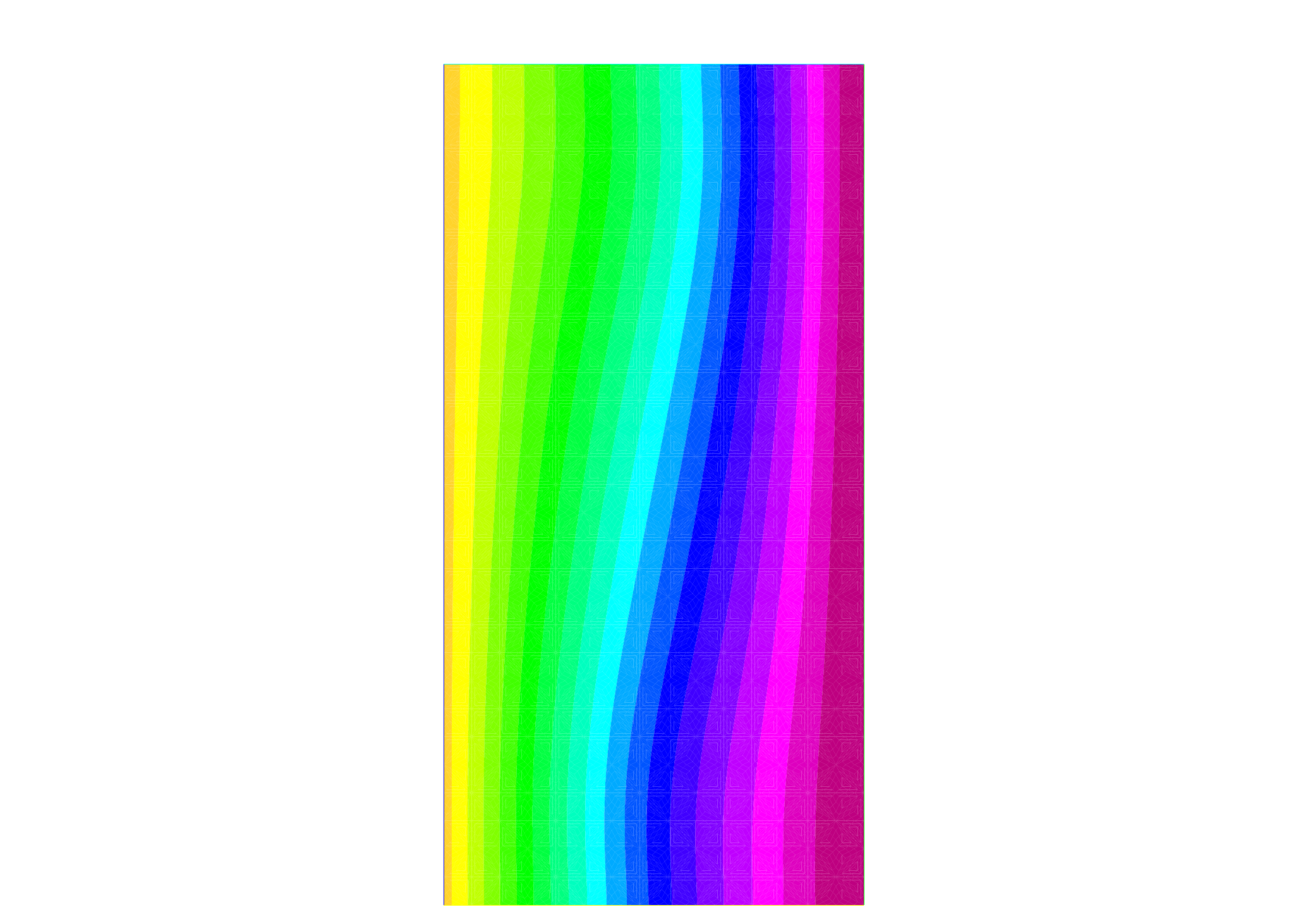,width=0.35\textwidth}
}}
	\caption{\label{fig:cav1}Streamlines, temperature contours and concentration contours for the scheme (up) and for DNS (down) for $\mu_1=\mu_2=\mu_3=1$  }
\end{figure}

\begin{figure}[H]
	\centerline{\hbox{
			\epsfig{figure=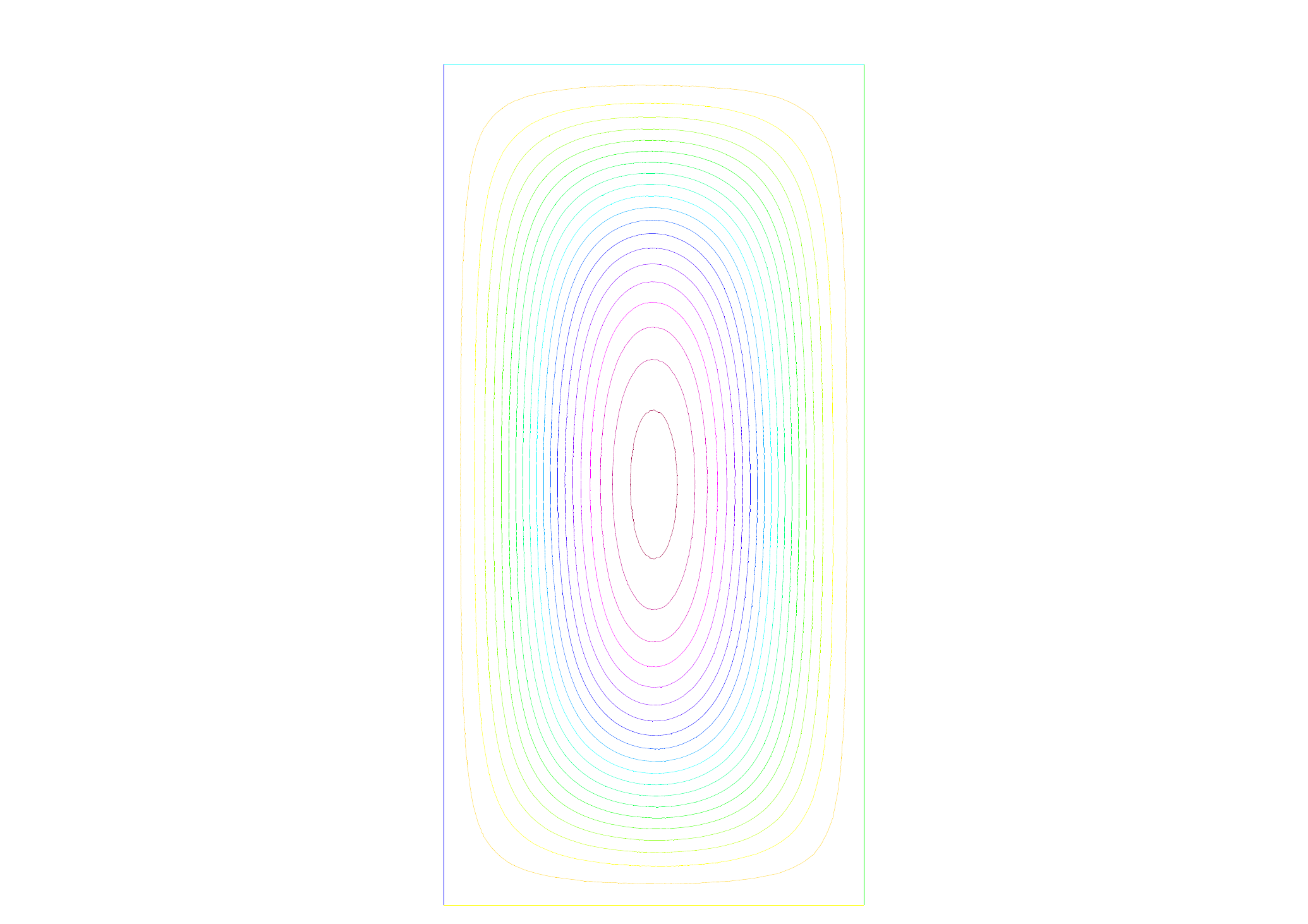,width=0.35\textwidth}
			\epsfig{figure=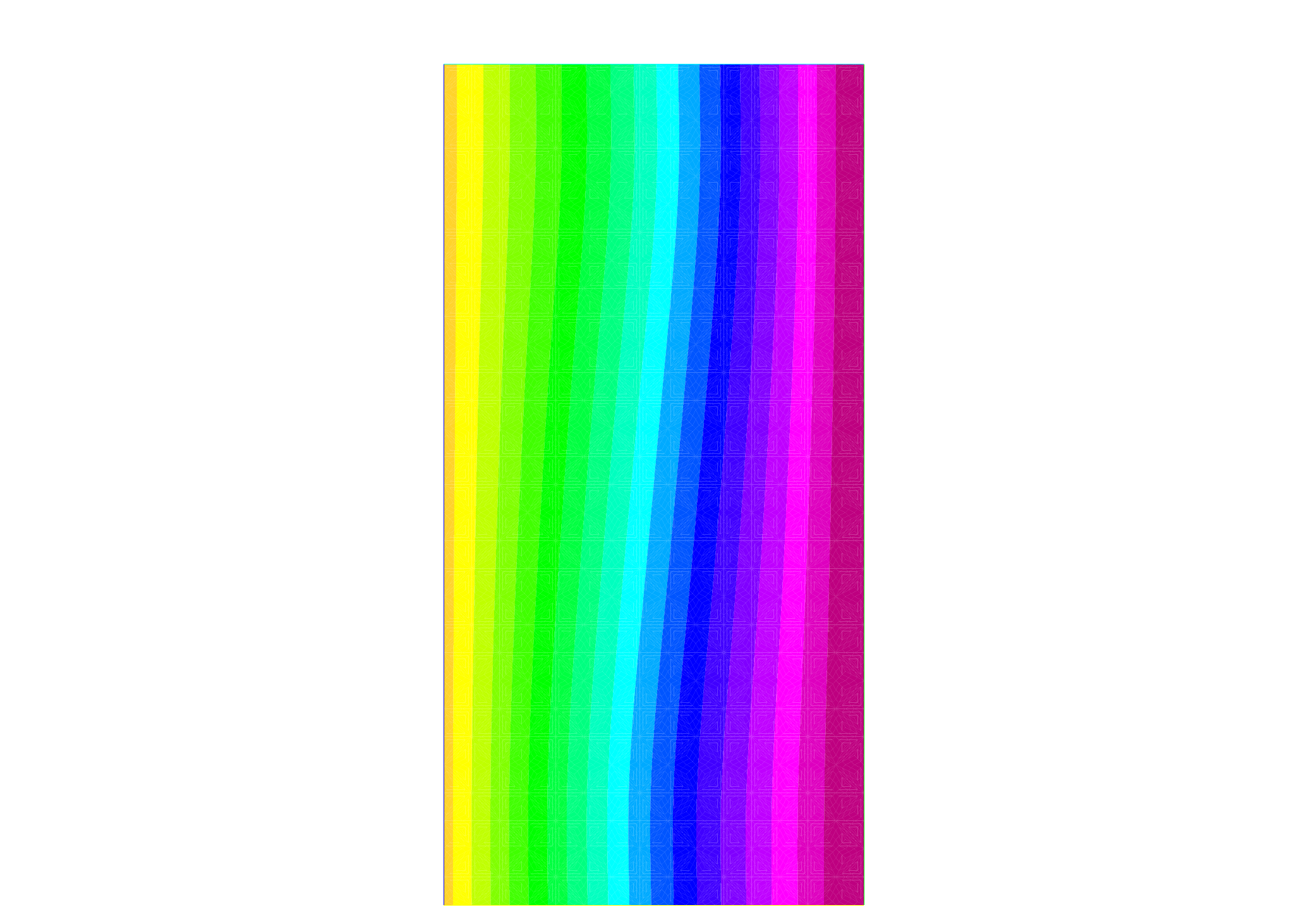,width=0.35\textwidth}
			\epsfig{figure=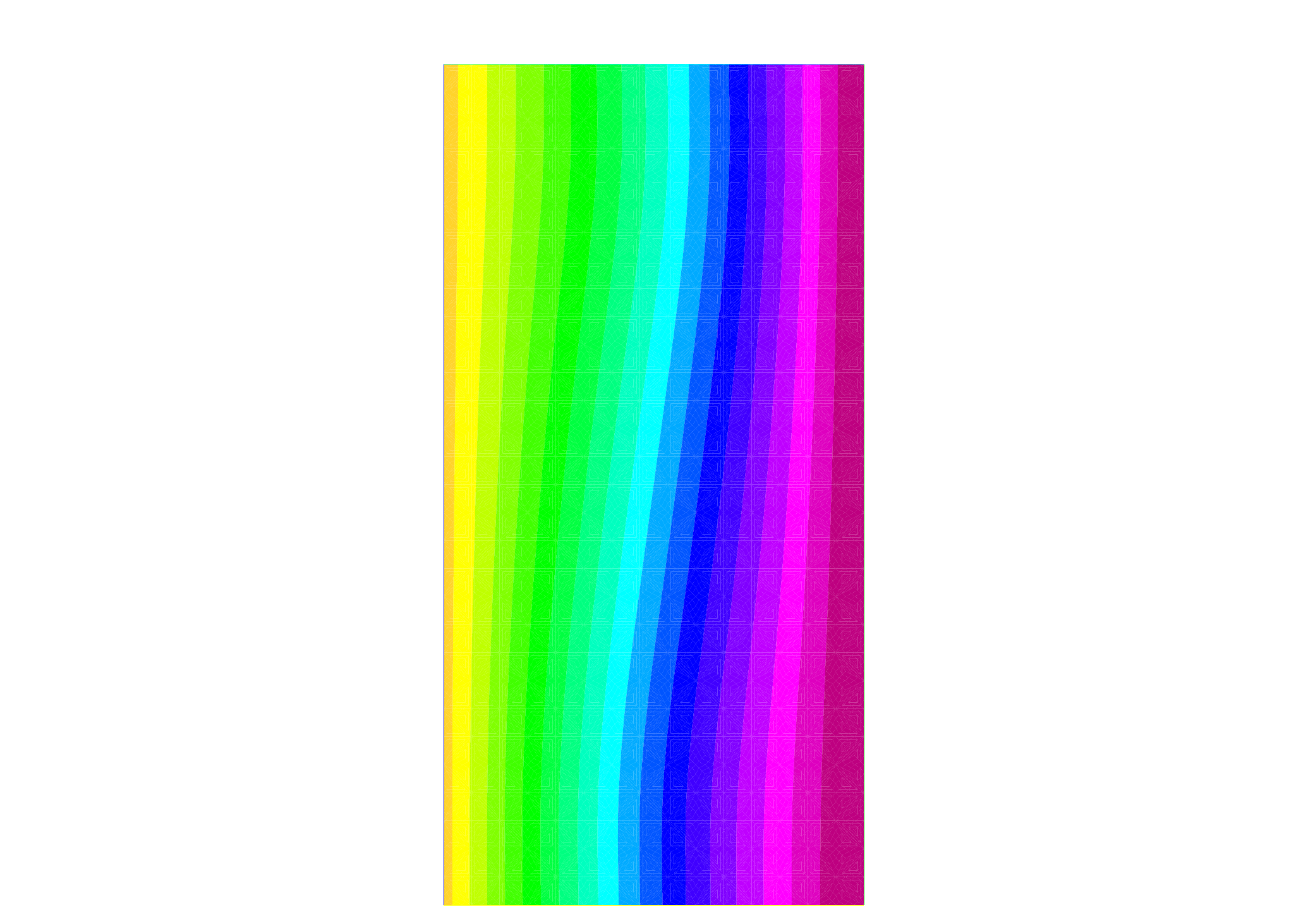,width=0.35\textwidth}
	}}
	\centerline{\hbox{
			\epsfig{figure=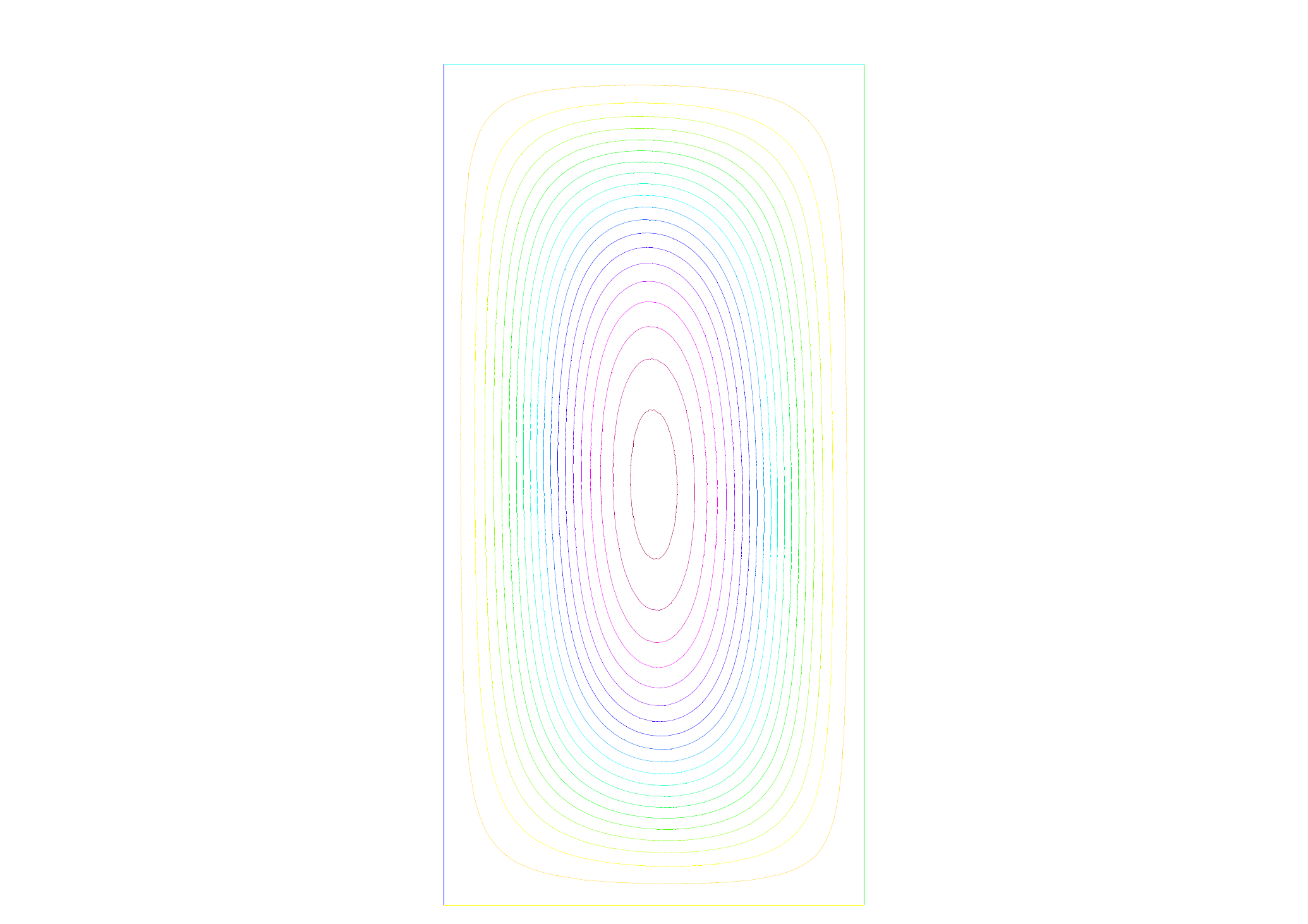,width=0.35\textwidth}
			\epsfig{figure=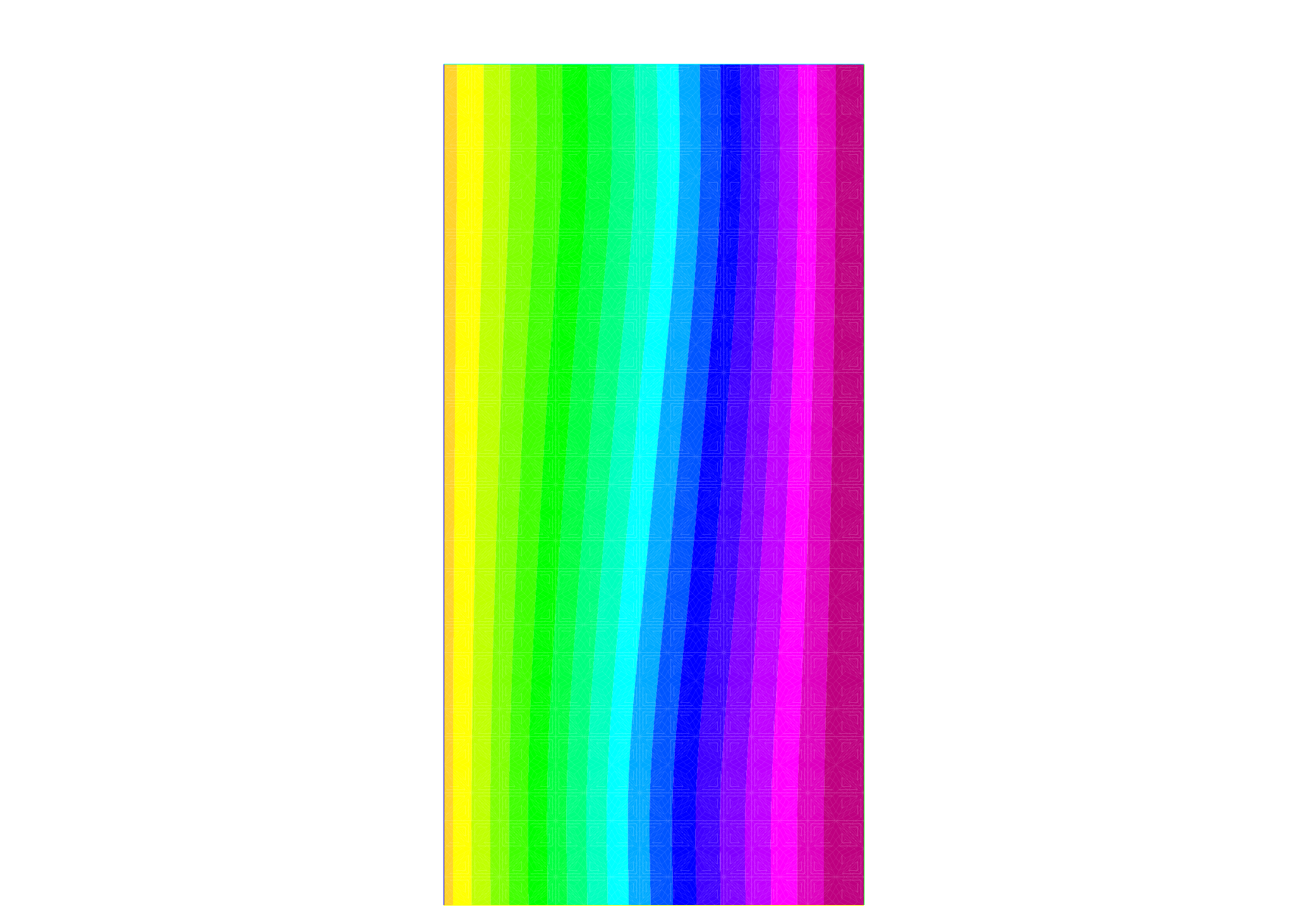,width=0.35\textwidth}
			\epsfig{figure=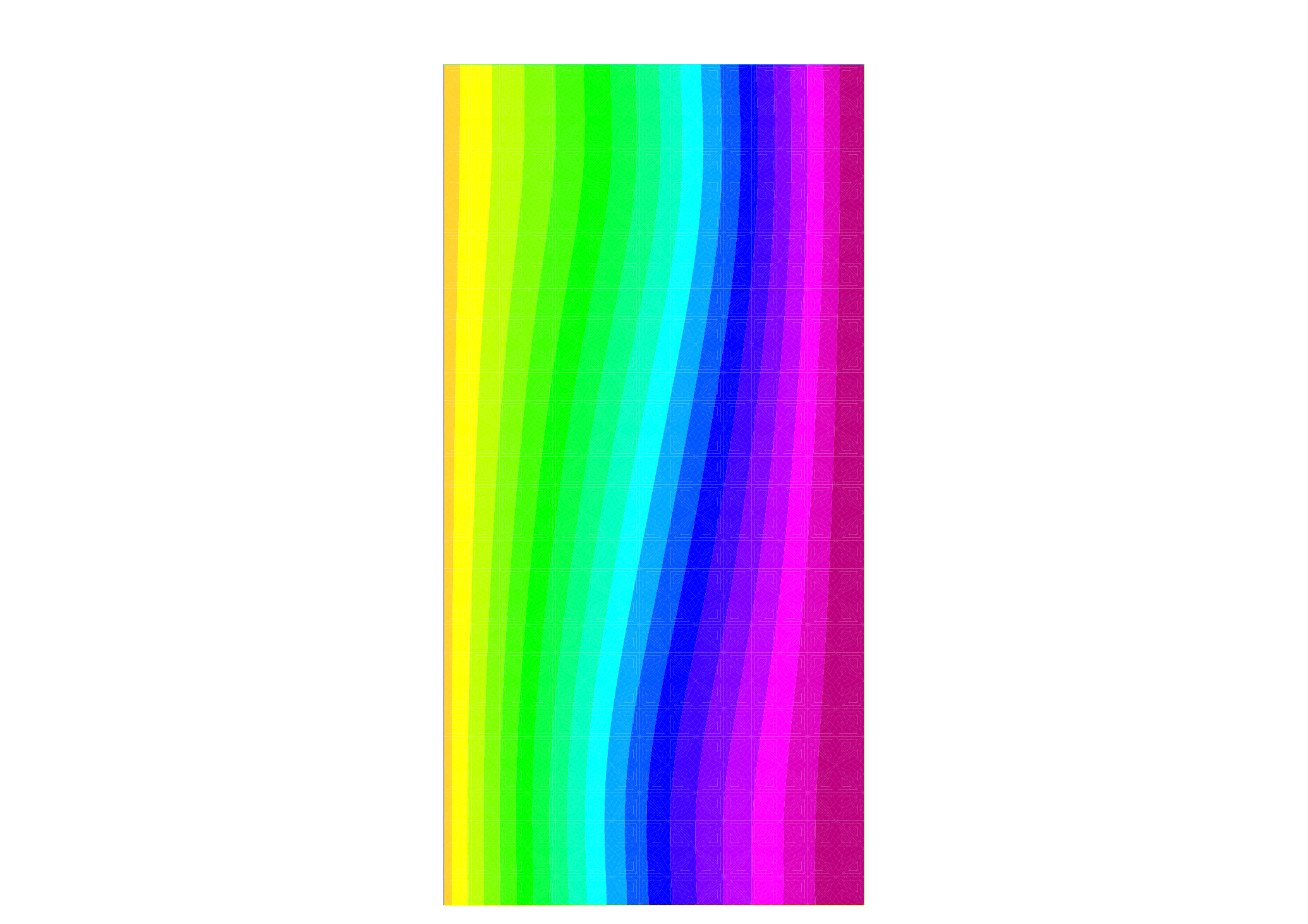,width=0.35\textwidth}
	}}
	\caption{\label{fig:cav2}Streamlines, temperature contours and concentration contours for the scheme (up) and for DNS (down) for $\mu_1=\mu_2=\mu_3=10$  }
\end{figure}
Clearly for both $\mu_1=\mu_2=\mu_3=1$ and$\mu_1=\mu_2=\mu_3=10$ cases, pictures are almost identical and agrees with DNS and previous studies \cite{bizimdouble} and \cite{chamka}.

Now considering the case $\mu_1>0, \mu_2=0, \mu_3=0$ we obtain Figures \ref{fig:cav3} and \ref{fig:cav4}.
\begin{figure}[H]
	\centerline{\hbox{
			\epsfig{figure=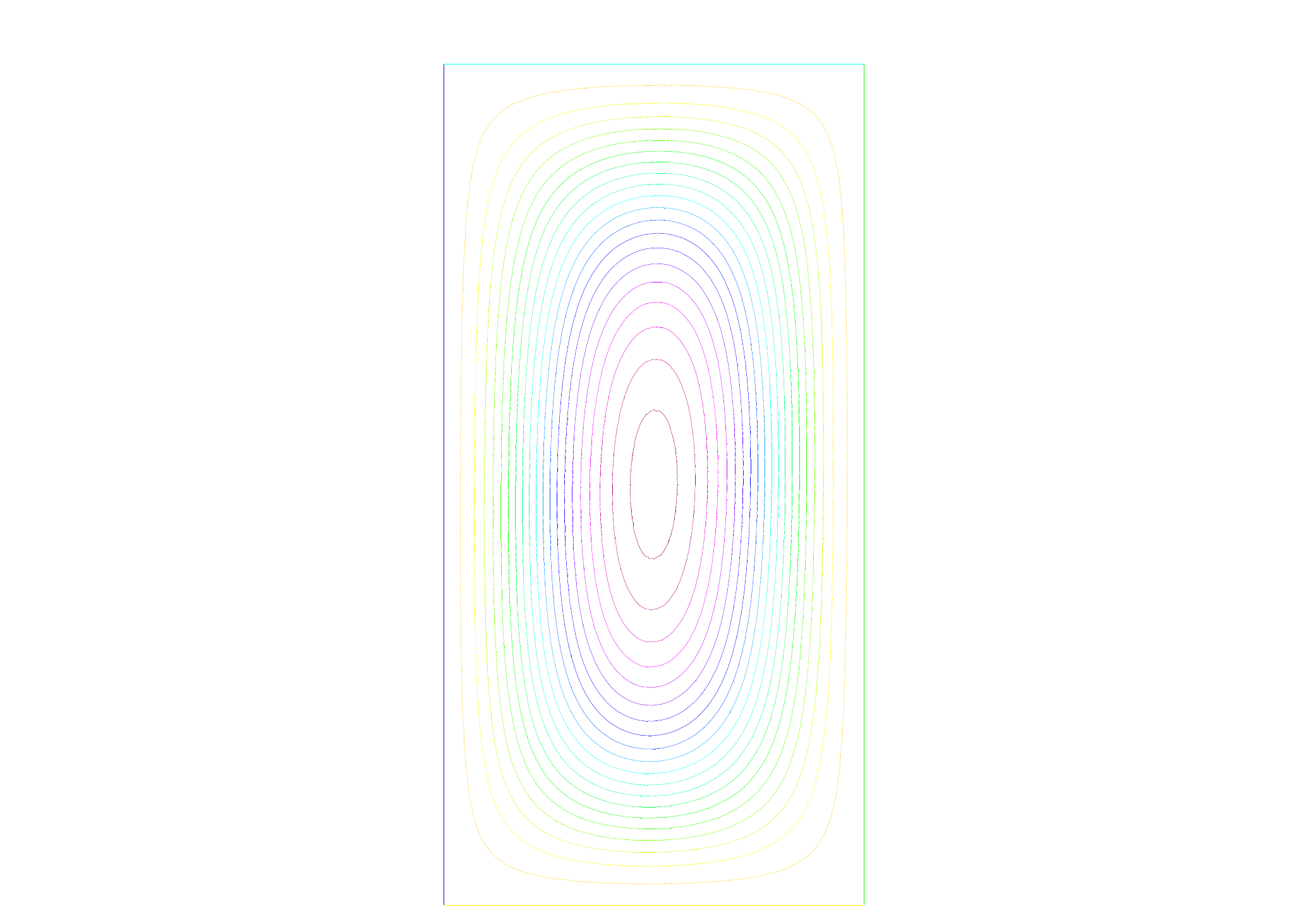,width=0.35\textwidth}
			\epsfig{figure=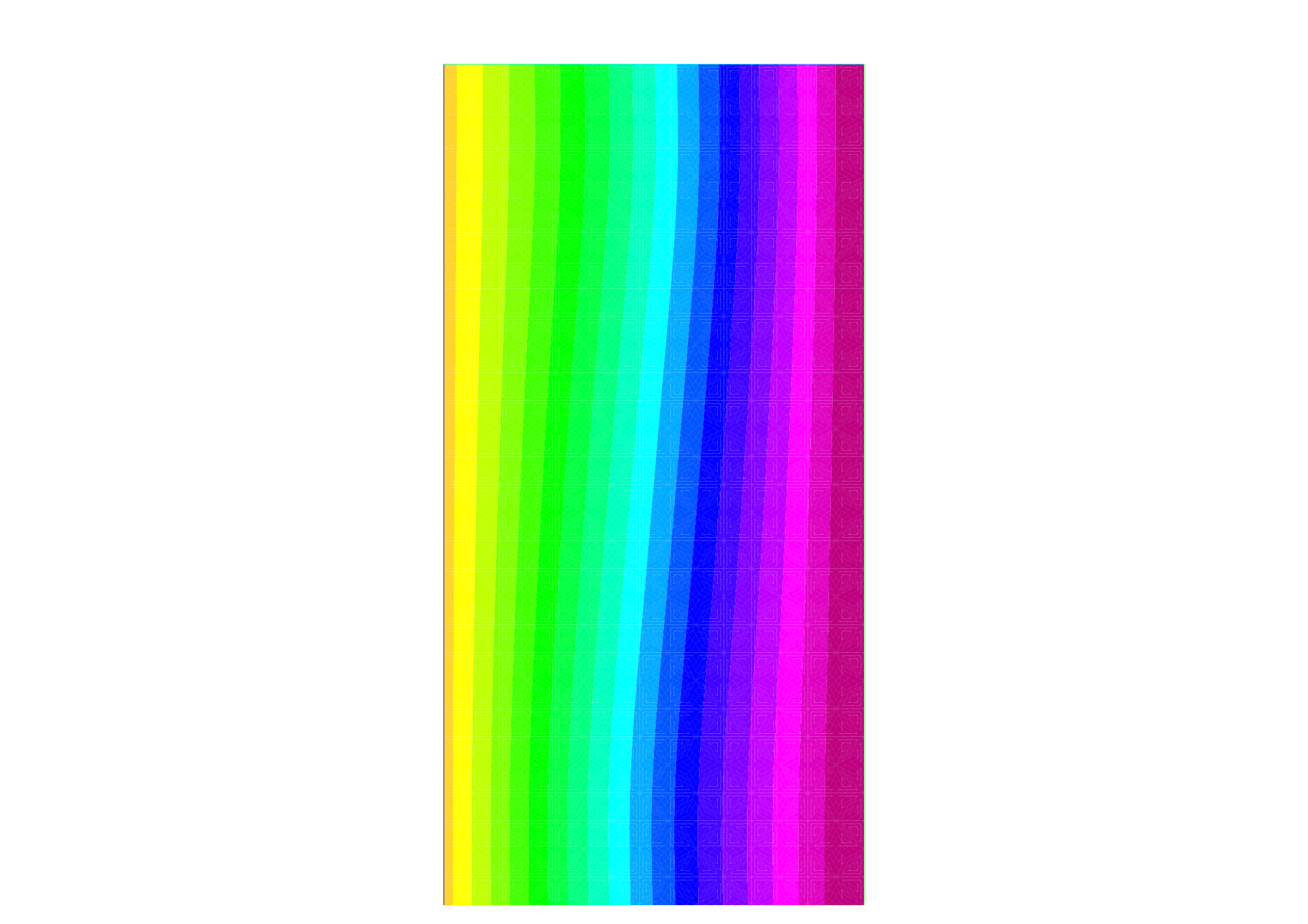,width=0.35\textwidth}
			\epsfig{figure=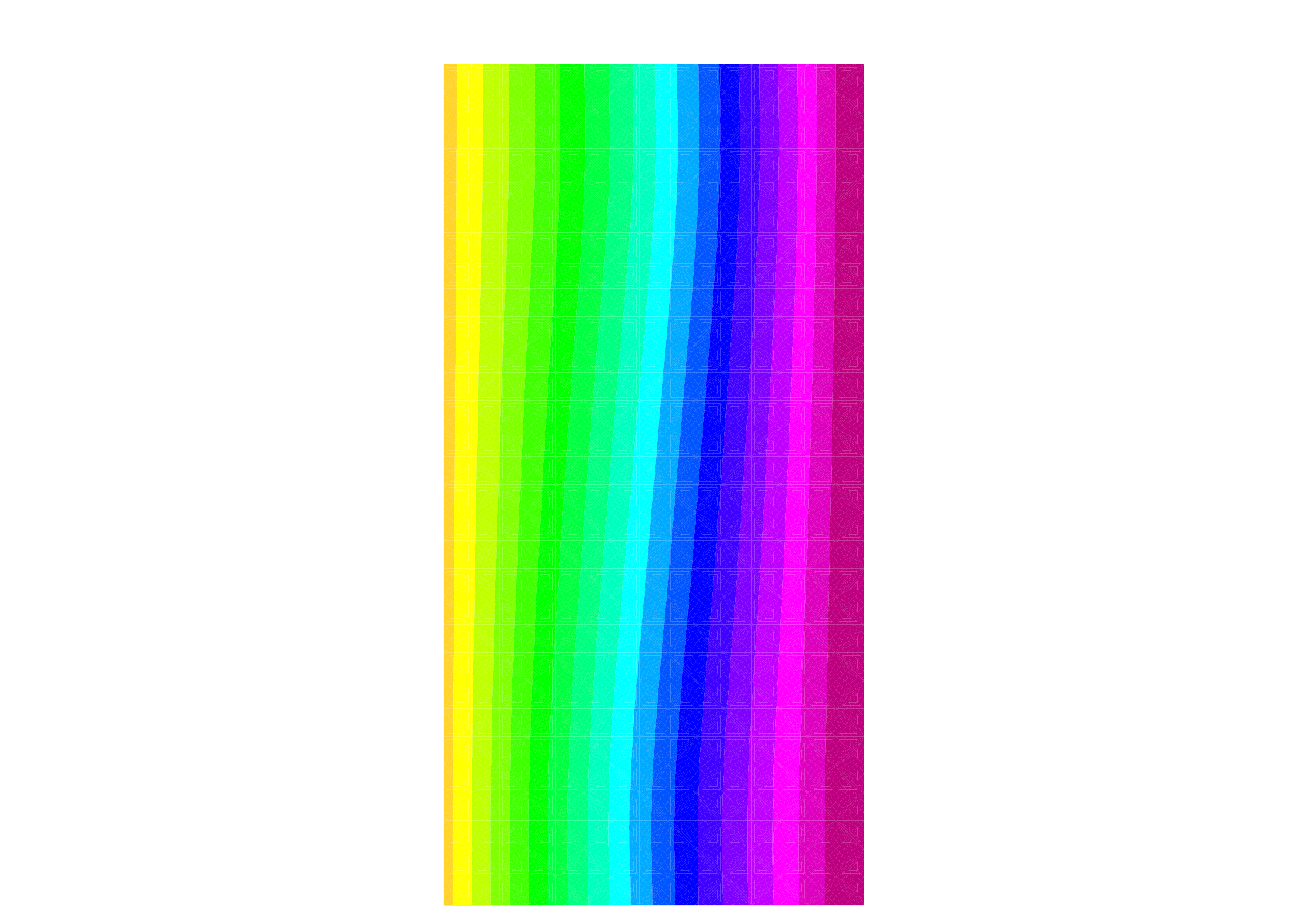,width=0.35\textwidth}
	}}
	\centerline{\hbox{
			\epsfig{figure=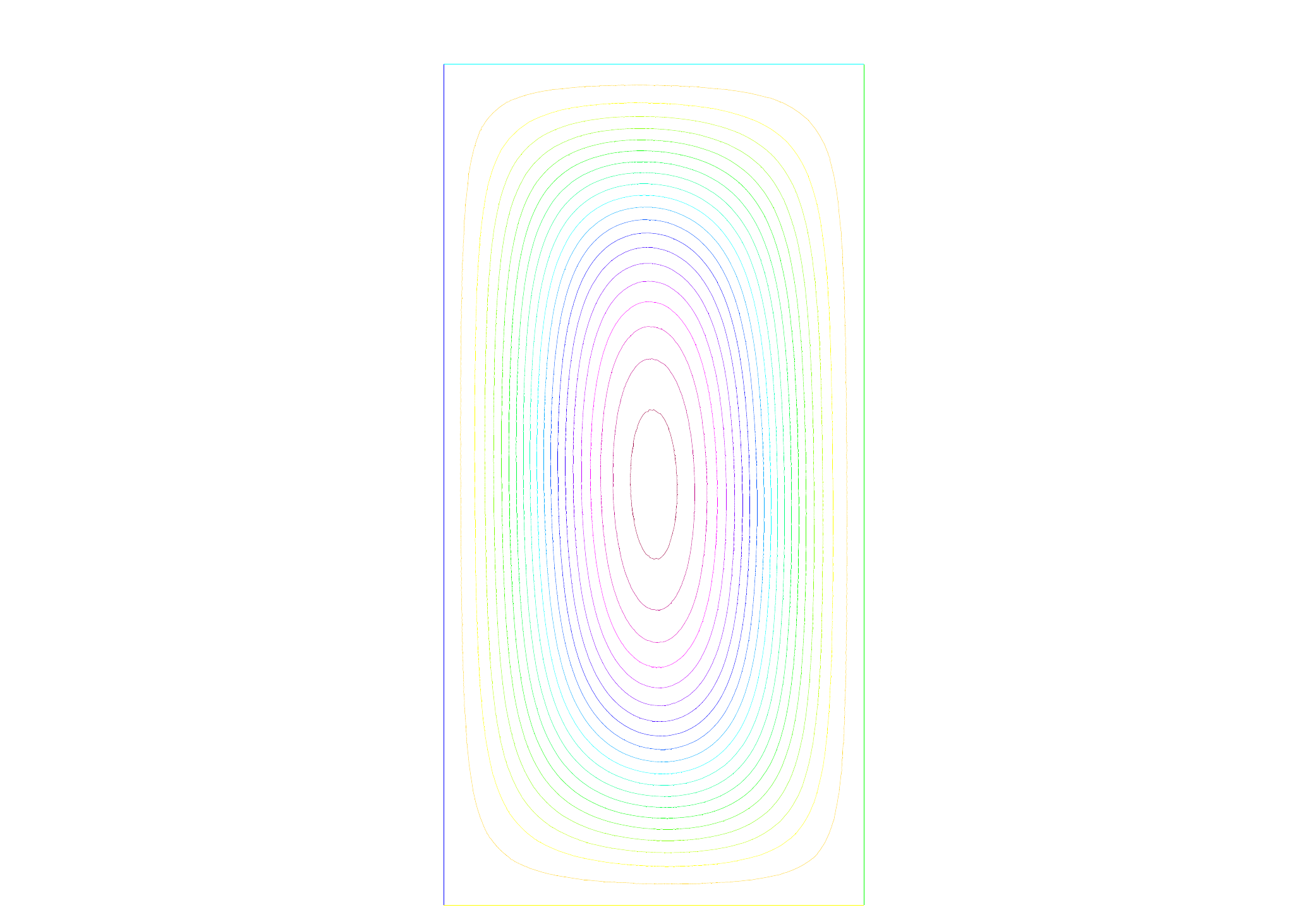,width=0.35\textwidth}
			\epsfig{figure=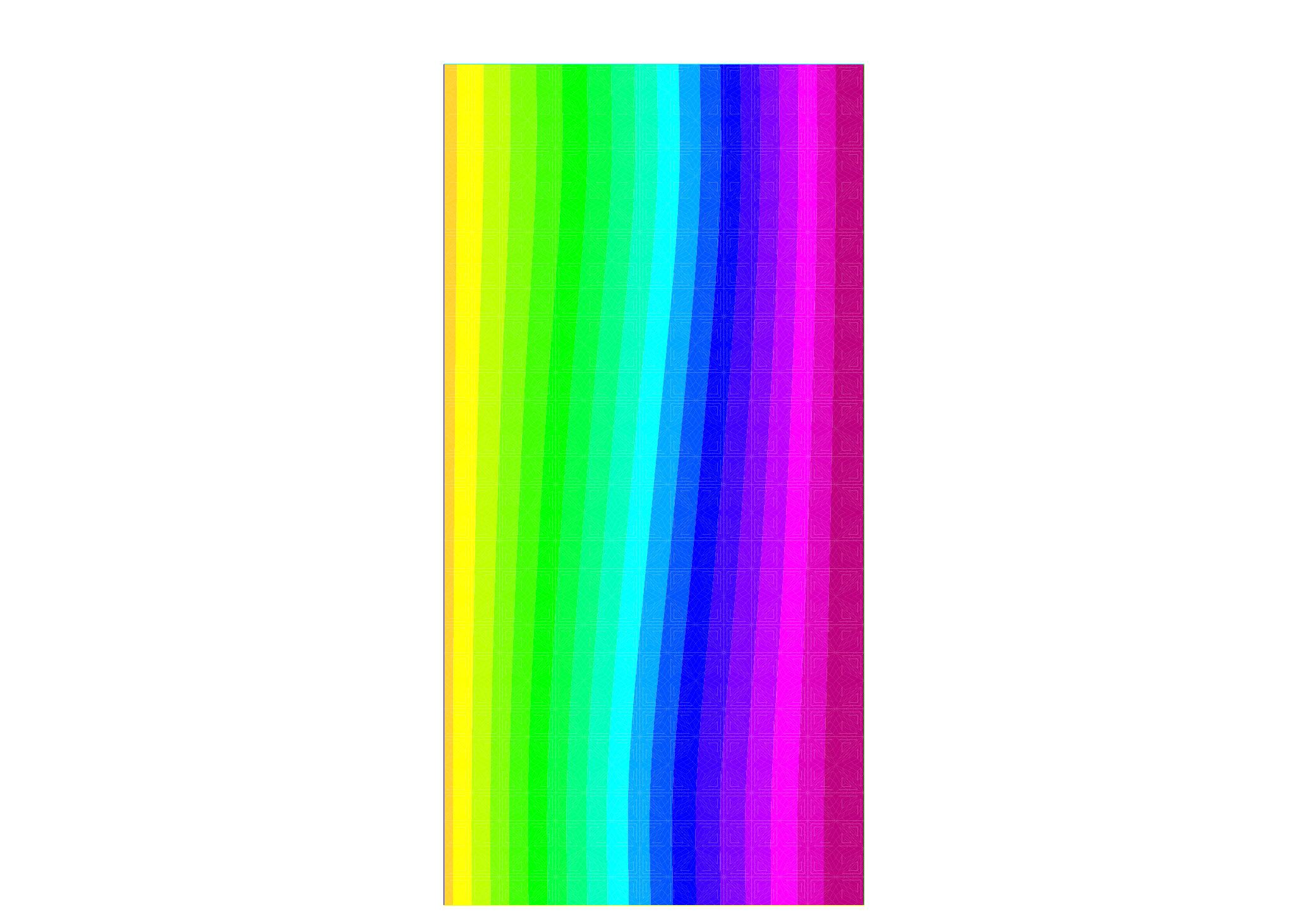,width=0.35\textwidth}
			\epsfig{figure=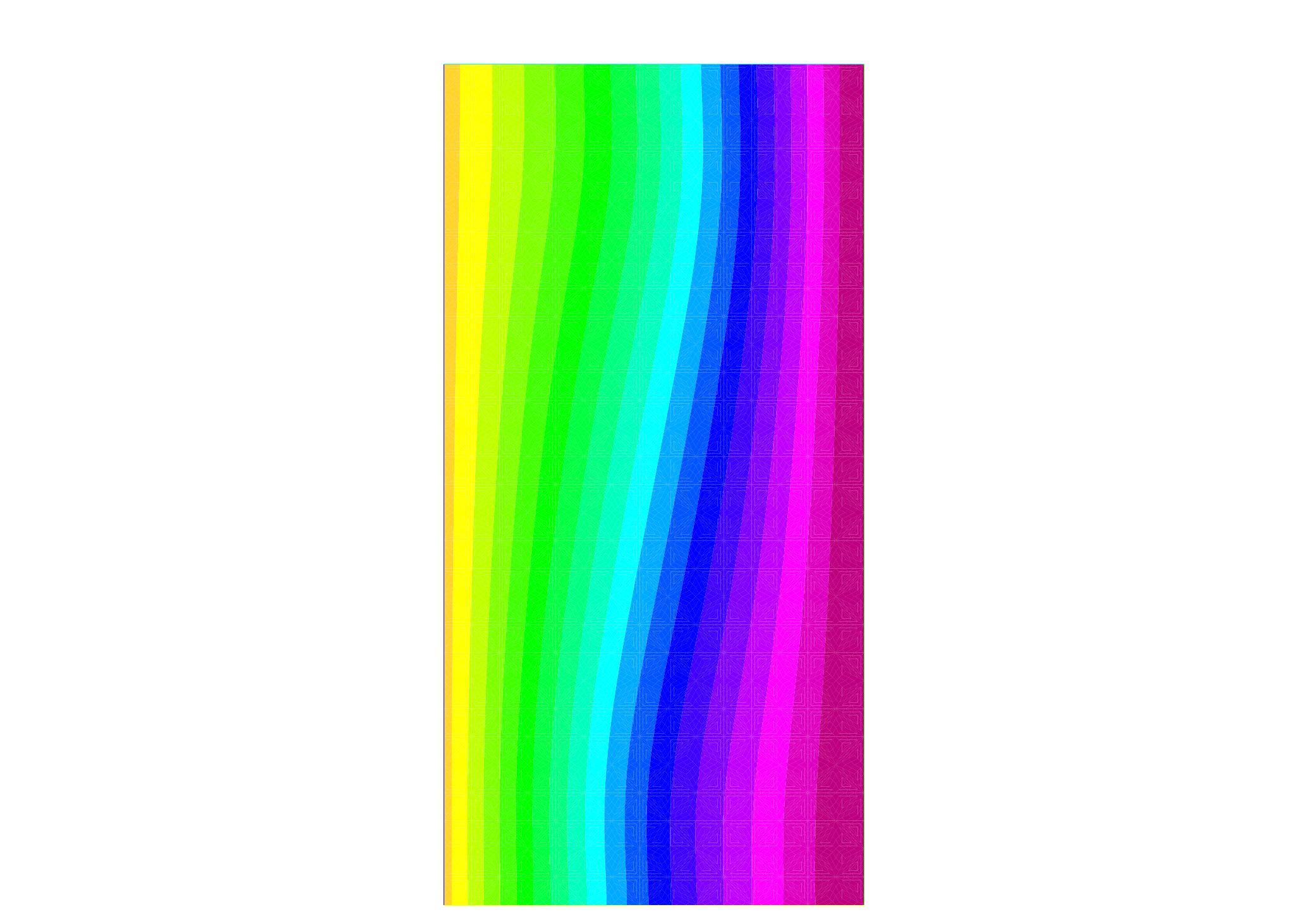,width=0.35\textwidth}
	}}
	\caption{\label{fig:cav3}Streamlines, temperature contours and concentration contours for the scheme (up) and for DNS (down) for $\mu_1=1,\,\,\mu_2=\mu_3=0$  }
\end{figure}

\begin{figure}[H]
	\centerline{\hbox{
		\epsfig{figure=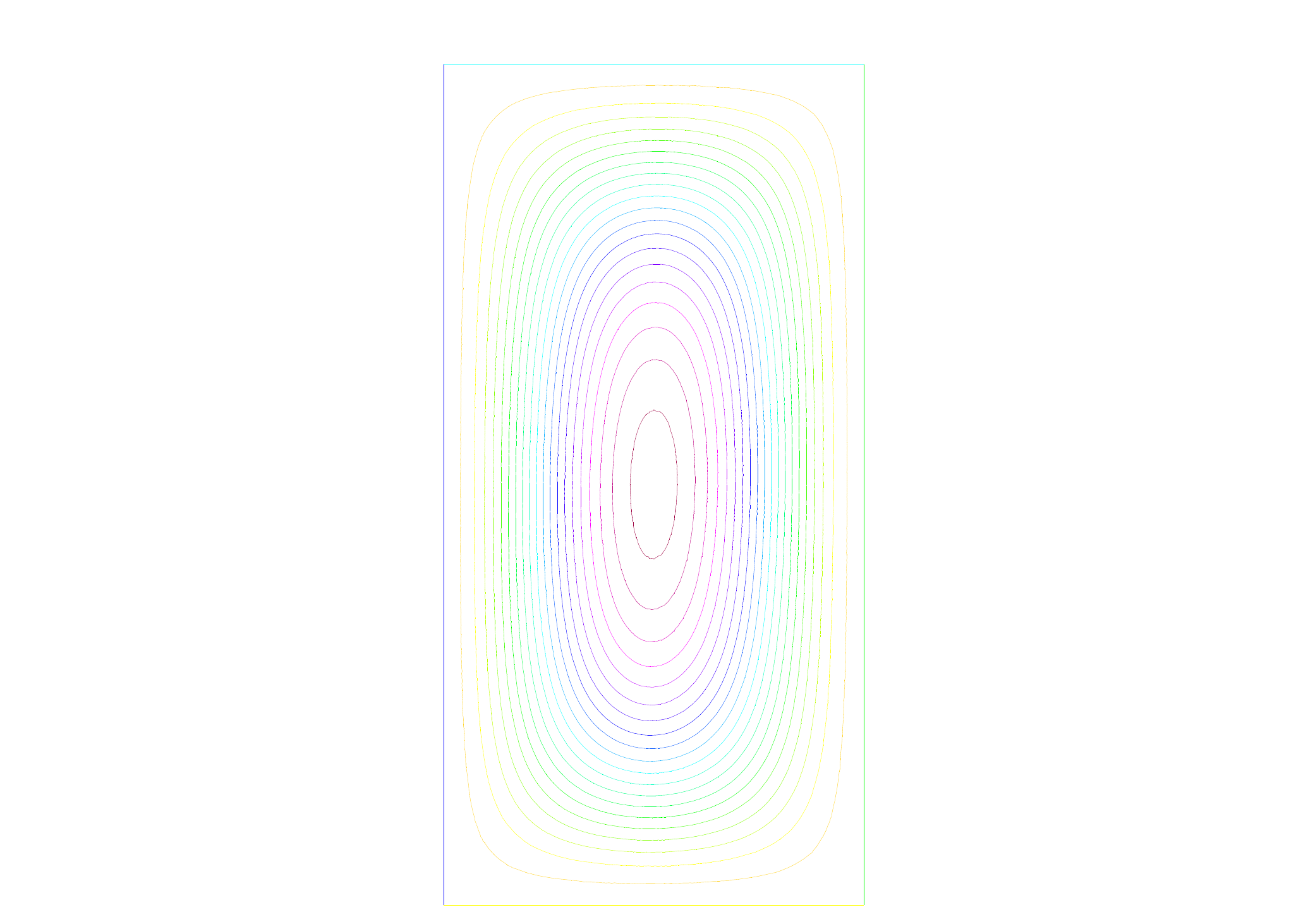,width=0.35\textwidth}
		\epsfig{figure=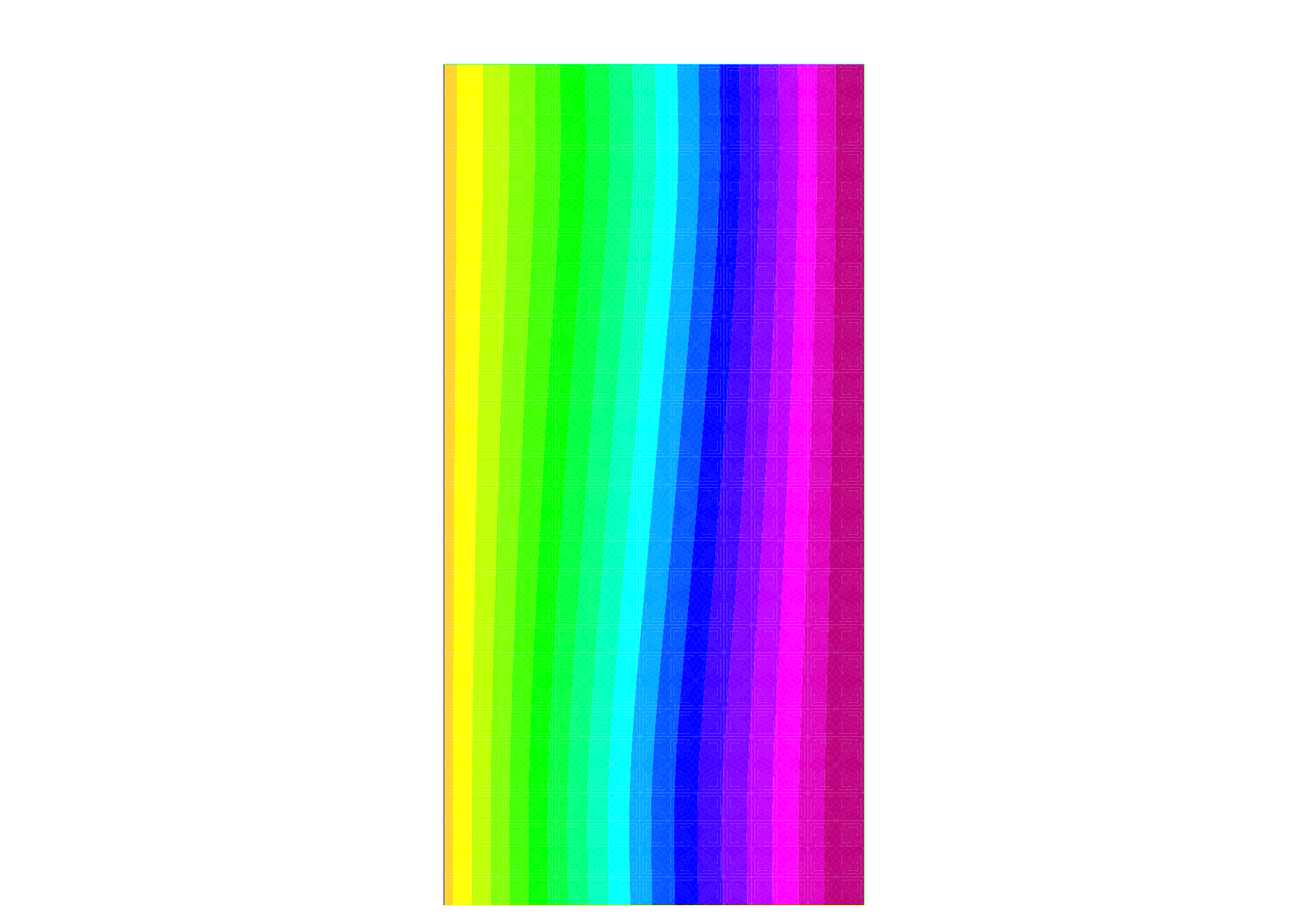,width=0.35\textwidth}
		\epsfig{figure=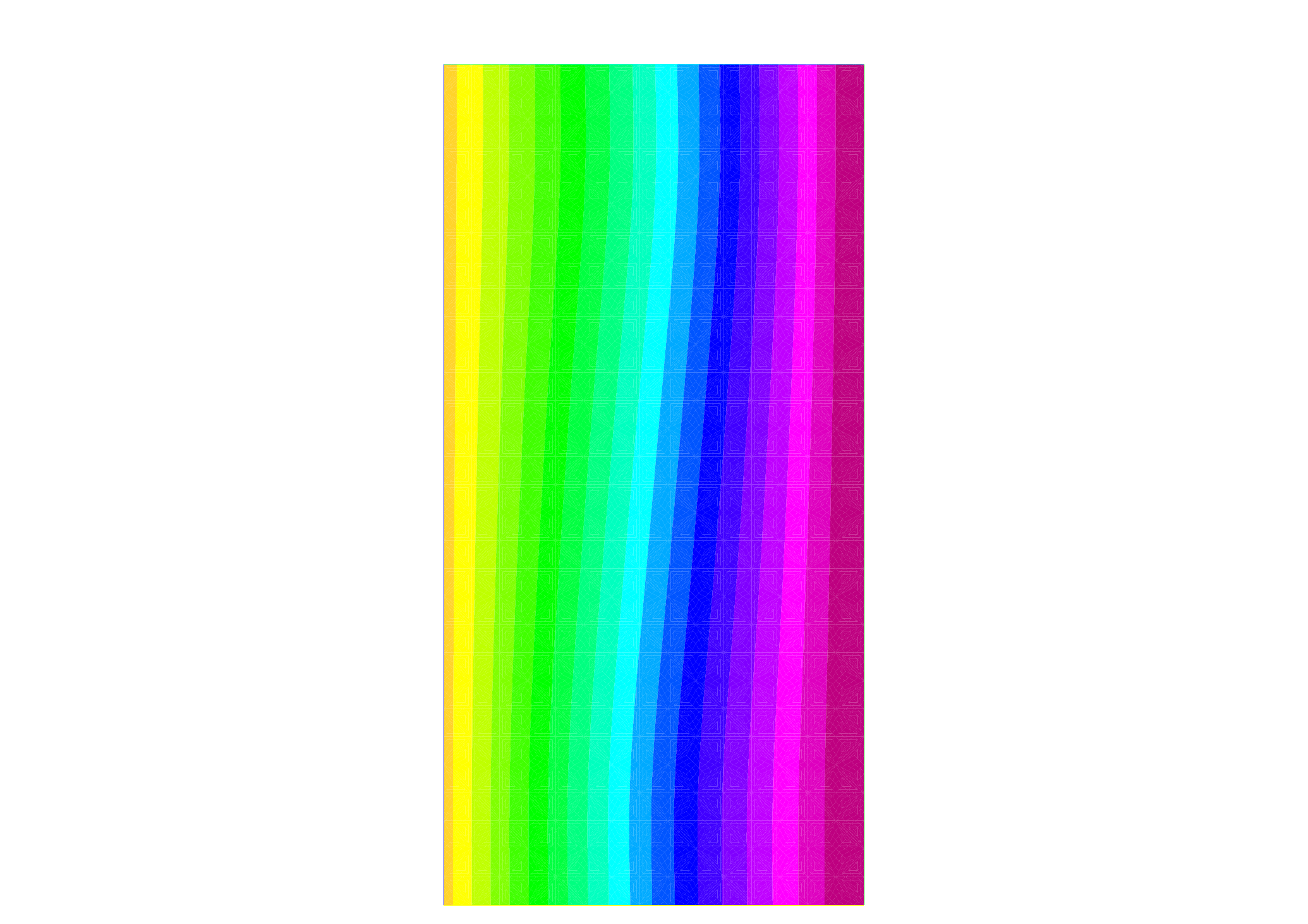,width=0.35\textwidth}
}}
\centerline{\hbox{
		\epsfig{figure=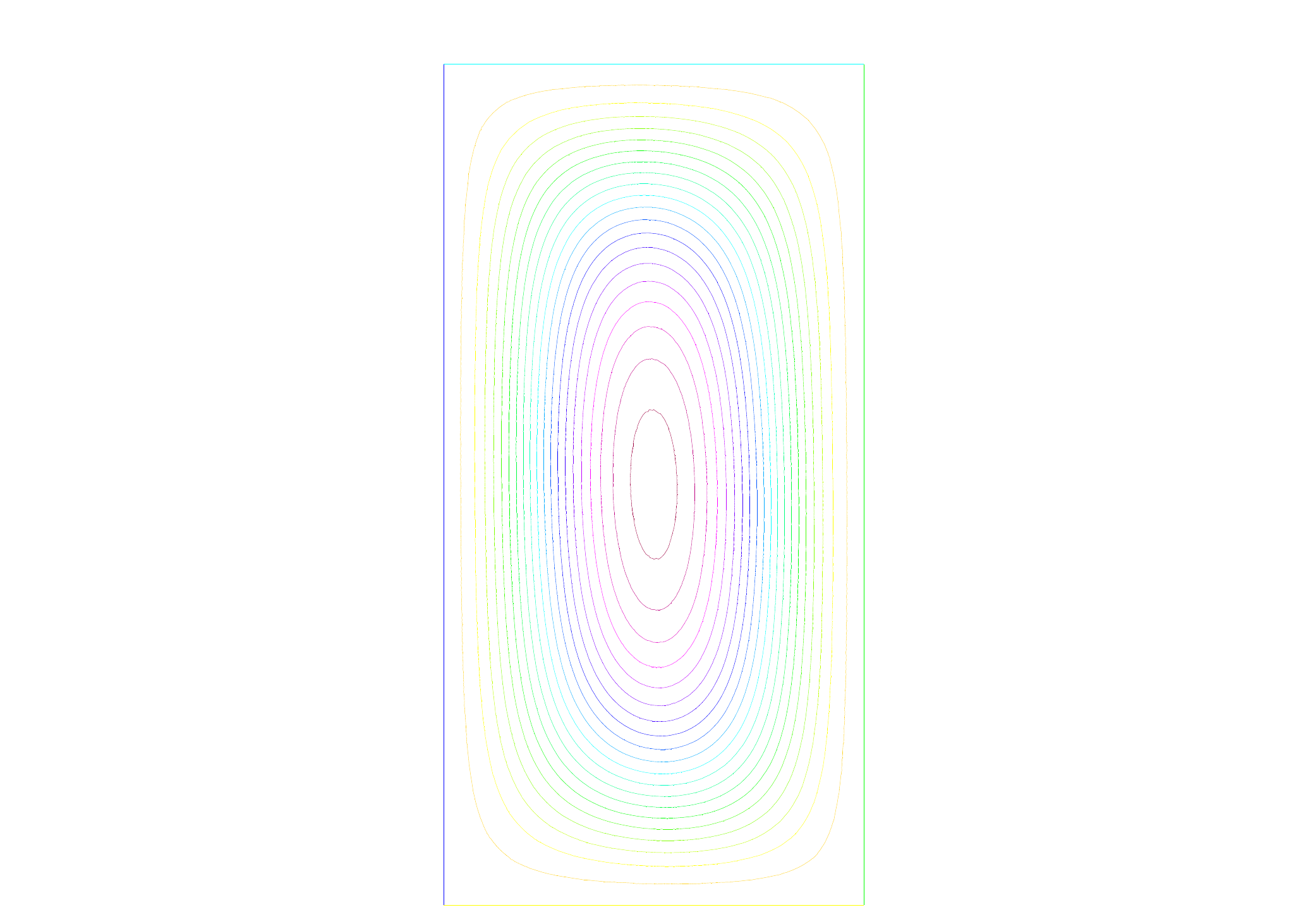,width=0.35\textwidth}
		\epsfig{figure=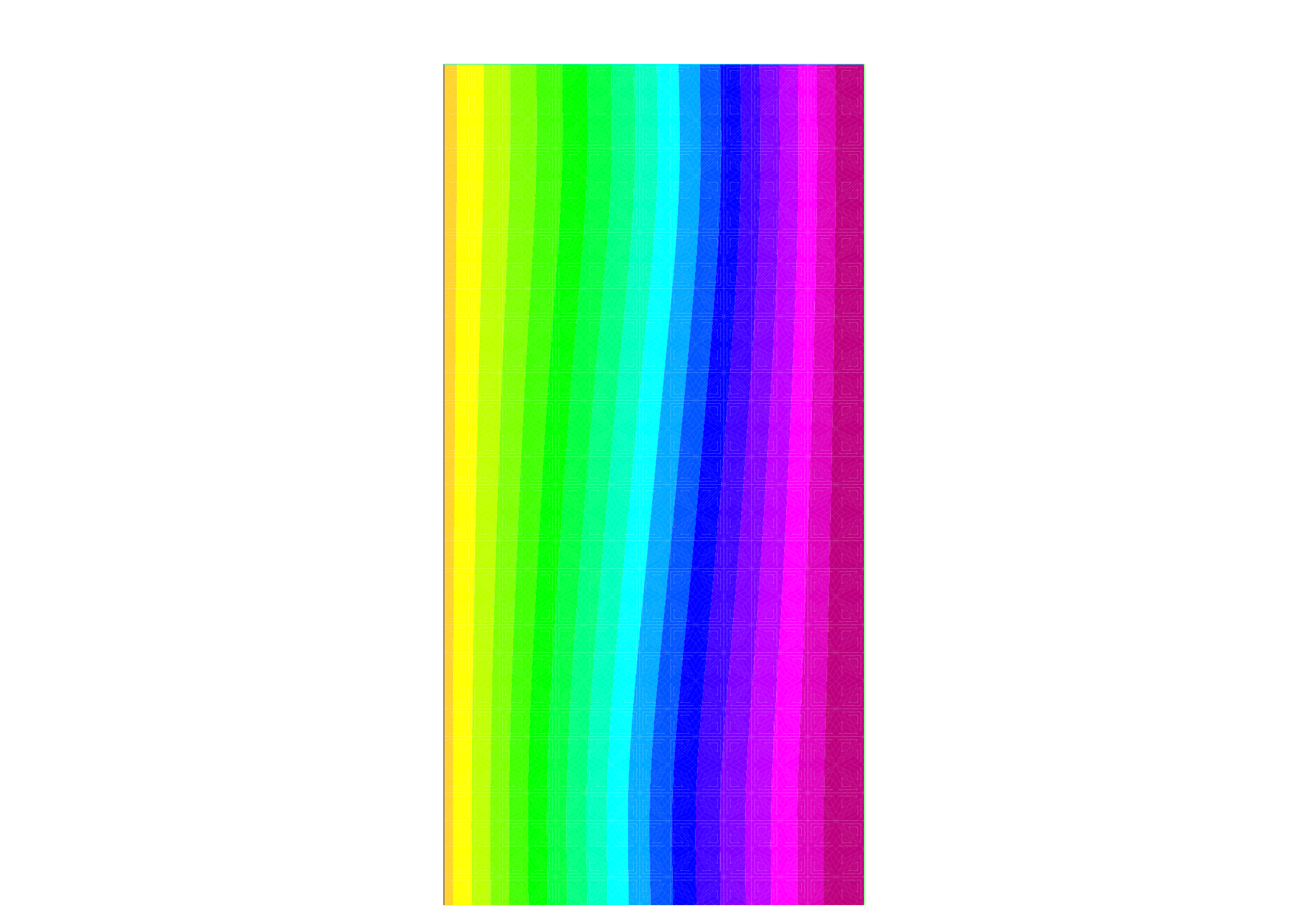,width=0.35\textwidth}
		\epsfig{figure=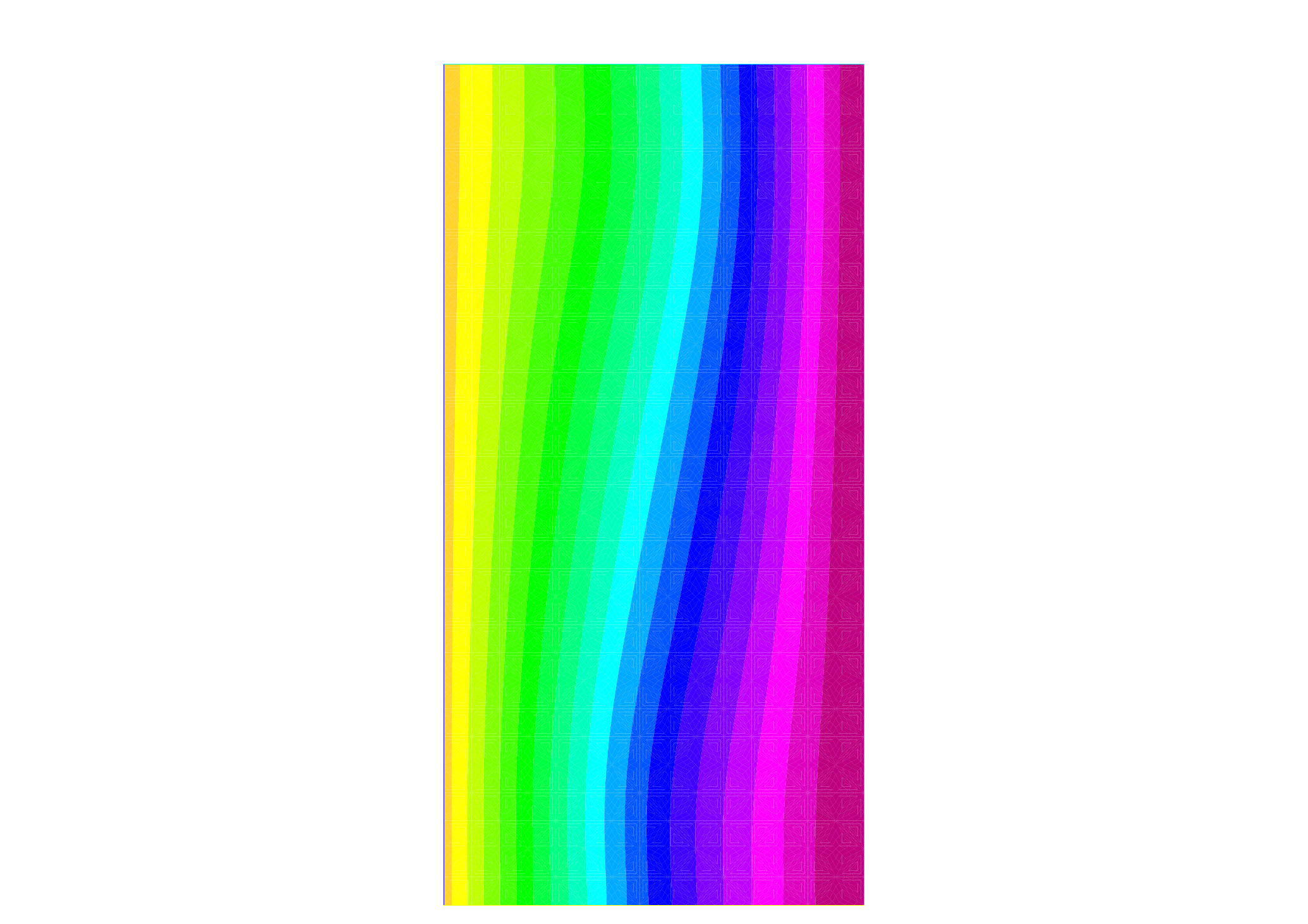,width=0.35\textwidth}
}}
	\caption{\label{fig:cav4}Streamlines, temperature contours and concentration contours for the scheme (up) and for DNS (down) for $\mu_1=10,\,\,\mu_2=\mu_3=0$  }
\end{figure}
These figures are also identical to each other and $\mu_1>0, \mu_2>0, \mu_3>0$ case. We cannot see any effect of changing $\mu_1$ value too. The only difference to be mentioned here is, the CPU time for running $\mu_1>0, \mu_2>0, \mu_3>0$ case is slightly shorter than the case $\mu_1>0, \mu_2>=\mu_3=0$. We can conclude from this test that, our scheme presented with Algorithm \ref{algbe} convergences to DNS and produce correct result with both nudging all equations and nudging only velocity equation with random initial data.
\section{Conclusions}
We have presented and investigated a continuous data assimilation scheme cast on Darcy-Brinkman equations. The scheme involves a first or second order time discretization along with a finite element spatial discretization. We have given long time stability and accuracy analyses for different cases of nudging parameters. These finding are tested via different numerical experiments confirming the excepted error order and usage in a practical test case.

As a further research topic, we will consider to incorporate the nudging terms in some stabilization methods like VMS stabilization for different kinds of flow problems.

\end{document}